\documentclass[12pt]{amsart}

\usepackage{amsmath}
\usepackage{amssymb}
\usepackage{mathdots}
\usepackage{tikz}
\usepackage{adjustbox}
\usepackage[all]{xy}

\numberwithin{equation}{section}

\newtheorem{thm}{Theorem}[section]

\newtheorem{lem}[thm]{Lemma}
\newtheorem{prop}[thm]{Proposition}
\newtheorem{cor}[thm]{Corollary}
\newtheorem{rem}[thm]{Remark}
\newtheorem{exam}[thm]{Example}
\newtheorem{exam-nota}[thm]{Example-Notation}
\newtheorem{rem-nota}[thm]{Remark-Notation}
\newtheorem{nota}[thm]{Notation}
\newtheorem{dfn}[thm]{Definition}

\newtheorem{dfn-nota}[thm]{Definition-Notation}

\newtheorem{dfn-lem}[thm]{Lemma-Definition}
\newtheorem{dfn-prop}[thm]{Proposition-Definition}

\newcommand{\beqa}{\begin{eqnarray*}}
\newcommand{\eeqa}{\end{eqnarray*}}

\newcommand{\cZ}{{\mathcal Z}}

\newcommand{\alphatilde}{{\tilde{\alpha}}}

\newcommand{\Wn}{{W}_n}
\newcommand{\cW}{{ W}}

\newcommand{\fa}{\mbox{${\mathfrak a}$}}

\newcommand{\fg}{\mbox{${\mathfrak g}$}}

\newcommand{\fl}{\mbox{${\mathfrak l}$}}

\newcommand{\fh}{\mbox{${\mathfrak h}$}}

\newcommand{\fp}{\mbox{${\mathfrak p}$}}
\newcommand{\fr}{\mbox{${\mathfrak r}$}}
\newcommand{\fb}{\mbox{${\mathfrak b}$}}

\newcommand{\eps}{\epsilon}

\newcommand{\PR}{\mbox{${\mathbb P}$}}

\newcommand{\C}{\mbox{${\mathbb C}$}}

\newcommand{\Ad}{{\rm Ad}}

\newcommand{\fgl}{\mathfrak{gl}}

\newcommand{\B}{\mathcal{B}}

\newcommand{\ms}{m(s_{\alpha})}

\newcommand{\he}{\hat{e}}

\newcommand{\F}{\mathcal{F}}

\newcommand{\Borbitspace}{B_{n-1}\backslash\B_{n}}
\newcommand{\Sh}{\mathcal{S}h}
\newcommand{\calO}{\mathcal{O}}

\newcommand{\Sp}{\mathcal{S}p}

\newcommand{\BQ}{B\backslash Q_{1,i+1}}
\begin{document}
\title{Orbits on a product of two flags and a line and the Bruhat order, II}
\author[M. Colarusso]{Mark Colarusso}
\address{Department of Math and Stats, University of South Alabama, Mobile, AL, 36608}
\email{mcolarusso@southalabama.edu}
\author[S. Evens]{Sam Evens}
\address{Department of Mathematics, University of Notre Dame, Notre Dame, IN, 46556}
\email{sevens@nd.edu}

\maketitle
\begin{abstract}
Let $G=GL(n)$ be the $n\times n$ complex general linear group and let $\B_{n}$ be its flag variety.  A Borel subgroup $B$ of $G$ acts on $\B_{n}\times \mathbb{P}^{n-1}$ diagonally with finitely many orbits.  In this paper, we give an embedding of the $B$-orbits on $\B_{n}\times \mathbb{P}^{n-1}$ into the $B$-orbits on the flag variety $\B_{n+1}$ of $GL(n+1)$ and show that this correspondence respects closure relations and preserves monoid actions.  As a consequence both closure relations and monoid actions on the set of all $B$-orbits on $\B_{n}\times\mathbb{P}^{n-1}$ can be understood via the Bruhat order on the symmetric group on $n+1$ letters by using our results in \cite{Shpairs}.  This amplifies work of Magyar \cite{Magyar} by making the closure relation more transparent and allows us to compute the monoid action using Demazure products.   If $S_i$ is the stabilizer in $B$ of the line through the ith standard basis vector, we give an embedding of the $S_i$-orbits on $\B_n$ into the $B$-orbits in a single $G$-orbit in $\B_{n+1},$ and this embedding plays an essential role in the above results.  We extend results from our papers \cite{CE21I}, \cite{CE21II}, and \cite{Shpairs}, and in particular show that for $S_i$-orbits on $\B_n,$ the closure ordering is given by the Richardson-Springer standard order.
\end{abstract}

\tableofcontents

\section{Introduction}
Let $G=G_{n}=GL(n)$ be the $n\times n$ complex general linear group and let $\B_{n}$ be its flag variety.  It is well known that $G$ has finitely many orbits on the triple product 
$\B_{n}\times\B_{n}\times\mathbb{P}^{n-1}$ under the diagonal action.  These orbits have been an object of study for some time (see \cite{Magyar}, \cite{Travkin} for example) and are of broad interest since they are related to the study of mirabolic $\mathcal{D}$-modules which play an important role in the study of category $\mathcal{O}$ for rational Cherednik algebras \cite{FG}.  Magyar parameterizes these orbits in terms of decorated permutations and shows that the closure ordering on these orbits is given  in part by a numerical criterion on the set of decorated permutations.  This paper is motivated largely by our desire to understand the closure relations on these orbits in a simpler fashion and obtain a clearer picture of their geometry.  We build on our previous work in \cite{Shpairs} and \cite{Siorbits} to show that $G$-orbits on $\B_{n}\times\B_{n}\times \mathbb{P}^{n-1}$ are parameterized by a certain subset of $W_{n+1} \times W_{n+1},$ where $W_{n+1}$ is the symmetric group on $n+1$ letters, and moreover, the closure  relation on the orbits is the restriction of the product of the Bruhat orders on $W_{n+1}.$   We establish analogous results which compare certain monoid actions on $G$-orbits on $\B_{n}\times\B_{n}\times\mathbb{P}^{n-1}$ to familiar monoid actions on $W_{n+1} \times W_{n+1}$ defined using Demazure products.  These monoid actions give useful insight into how to relate the geometry of two orbit closures.  The results on the monoid actions are crucial to establishing our new description of the closure relations and the monoid action does not appear in Magyar's earlier work, but does appear in Travkin's work, although for a different purpose.



In more detail, let $G_{n+1}:=GL(n+1)$ and let $\B_{n+1}$ be its flag variety.  Embed $G$ in the top lefthand corner of $G_{n+1}$.  In our recent work \cite{CE21I}, \cite{CE21II}, and \cite{Shpairs}, we study the 
geometry and combinatorics of the $B$-orbits on $\B_{n+1}$, where $B\subset G$ is the standard Borel subgroup of invertible $n\times n$ upper triangular matrices.  For an algebraic group $A$ acting on a variety $Y$, we let $A\backslash Y$ denote the set of $A$-orbits on $Y.$  In the current paper, we establish a correspondence between  $G\backslash(\B_{n}\times\B_{n}\times\mathbb{P}^{n-1}),$  and a certain subset of $B\backslash\B_{n+1}$, which generalizes some of the results from \cite{CE21I}.  
To construct this correspondence, we make use of the fact that the group $G$ acts on $\B_{n+1}$ with finitely many orbits.  These orbits are labelled as $Q_{i,j}$, where $1\leq i\leq j\leq n+1$ (see Proposition \ref{prop:typeAflag} below for more details).  Consider the open subvariety of $\B_{n+1}$ given by 
\begin{equation}\label{eq:introXdefn}
\mathfrak{X}:=\displaystyle\bigcup_{i=1}^{n} Q_{1,i+1}.
\end{equation}
 One of the basic results of the paper is: 
\begin{thm}\label{thm:introthmbig}(see Theorems \ref{thm:globalcorresp} and \ref{thm:globalclosure})
There is a one-to-one correspondence between the set of $G$-orbits on the product $\B_{n}\times\B_{n}\times\mathbb{P}^{n-1}$ and the $B$-orbits on $\mathfrak{X}$, i.e.,
\begin{equation}\label{eq:bigintrocorresp}
G\backslash(\B_{n}\times\B_{n}\times\mathbb{P}^{n-1})\longleftrightarrow B\backslash\mathfrak{X},\; \calO_{\Delta}\mapsto \calO_{\Delta}^{op}. 
\end{equation}
Further, the correspondence in (\ref{eq:bigintrocorresp}) preserves closure relations, so that if $\calO_{\Delta}^{\prime},\, \calO_{\Delta}\in G\backslash(\B_{n}\times\B_{n}\times\mathbb{P}^{n-1})$ with corresponding orbits $\calO_{\Delta}^{\prime\, op}, \, \calO_{\Delta}^{op}\in B\backslash\mathfrak{X} $, then 
\begin{equation}\label{eq:introclosure}
\calO_{\Delta}^{\prime}\subset\overline{\calO_{\Delta}}\Leftrightarrow\calO_{\Delta}^{\prime\, op}\subset \overline{\calO_{\Delta}^{op}}.
\end{equation}
\end{thm} 

In \cite{Magyar}, Magyar parameterizes the $G$-orbits on $\B_{n}\times\B_{n}\times\mathbb{P}^{n-1}$ using \emph{decorated} permutations.  These are pairs $(w,\Delta)$ where $w\in W_{n}$ is an element of the symmetric group on $n$ letters and $\Delta=\{j_{1}<\dots <j_{k}\}\subset \{1,\dots, n\}$ is a subsequence such that $w^{-1}(j_{k})<\dots<w^{-1}(j_{1})$.  Magyar describes the closure relation amongst orbits in part using a numerical statistic involving both $w$ and the subsequence $\Delta$ (see Equation (\ref{eq:littledelta})).  Using Theorem \ref{thm:introthmbig} and our work in \cite{Shpairs}, we obtain a more transparent combinatorial description of the closure relations.  By results in \cite{Shpairs}, $B$-orbits on $\B_{n+1}$ can be parameterized by certain pairs $(w, u)\in W_{n+1}\times W_{n+1}$ and the closure relation on $B\backslash\B_{n+1}$ corresponds to the product Bruhat order on $W_{n+1}\times W_{n+1}$.  In more detail, given a $B$-orbit $\cZ\in B\backslash\B_{n+1}$, the corresponding pair of Weyl group elements $\widetilde{\Sh}(\cZ)=(w, u)$ is referred to as the \emph{standardized Shareshian pair} of the orbit $\cZ$. Theorem 3.4 of \cite{Shpairs} asserts that for $\cZ^{\prime},\, \cZ\in B\backslash\B_{n+1}$ with $\widetilde{\Sh}(\cZ^{\prime})=(x,y)$ and $\widetilde{\Sh}(\cZ)=(w,u)$, we have 
\begin{equation}\label{eq:introShorder}
\cZ^{\prime}\subset\overline{\cZ}\Leftrightarrow \widetilde{\Sh}(\cZ^{\prime})\leq \widetilde{\Sh}(\cZ)\Leftrightarrow x\leq w\mbox{ and } y\leq u,
\end{equation}
where $\leq$ denotes the Bruhat order on $W_{n+1}$.  Since $B\backslash\mathfrak{X}\subset B\backslash \B_{n+1}$, Theorem \ref{thm:introthmbig} implies:
\begin{cor}\label{c:introcor}
There exists an embedding of the poset of $G$-orbits $G\backslash(\B_{n}\times\B_{n}\times \mathbb{P}^{n-1})$ into the poset of standardized Shareshian pairs in $W_{n+1}\times W_{n+1}$ equipped with the product Bruhat order as in (\ref{eq:introShorder}).  That is to say, for $G$-orbits $\mathcal{O}_{\Delta}^{\prime}$ and $\mathcal{O}_{\Delta}$ corresponding via (\ref{eq:bigintrocorresp}) to the $B$-orbits on $\mathfrak{X}$, ${\mathcal{O}_{\Delta}^{\prime}}^{op}$ and $\mathcal{O}_{\Delta}^{op}$ respectively, with $\widetilde{\Sh}({\mathcal{O}_{\Delta}^{\prime}}^{op})=(x,y)$ and $\widetilde{\Sh}(\mathcal{O}_{\Delta}^{op})=(w,u)$, we have 
$$
\calO_{\Delta}^{\prime}\subset\overline{\calO_{\Delta}}\Leftrightarrow\calO_{\Delta}^{\prime\, op}\subset\overline{\calO_{\Delta}^{op}}\Leftrightarrow\widetilde{\Sh}(\mathcal{O}_{\Delta}^{\prime\, op})\leq \widetilde{\Sh}(\mathcal{O}_{\Delta}^{op})\Leftrightarrow x\leq w\mbox{ and } y\leq u.
$$   
\end{cor}

 The $G$-orbits on $\B_{n}\times\B_{n}\times\mathbb{P}^{n-1}$ come equipped with a monoid action by two copies of the standard simple roots $\Pi_{\fg}$ of $\fg$, one for each factor of $\B_{n}$ in the product $\B_{n}\times\B_{n}\times\mathbb{P}^{n-1}$.  This monoid action plays a key role in the proof of (\ref{eq:introclosure}) and does not appear in the earlier work of Magyar in \cite{Magyar}.  In \cite{CE21I}, we construct an \emph{extended monoid} action on $B\backslash\B_{n+1}$ using both simple roots of $\fg$ and of $\fg_{n+1}$.  This monoid action preserves $B$-orbits on the open subvariety $\mathfrak{X}$ and is easily be computed in terms of Shareshian pairs using Demazure products.  Recall for a simple reflection $s\in W_{n+1}$ and permutation $w\in W_{n+1}$ the Demazure product of $w$ and $s$ is given by:
\begin{equation}\label{eq:introrightDem}
w\star s:=\left\{ \begin{array}{ll} w & \text{ if } \ell(ws)< \ell(w) \\
ws & \text{ if } \ell(ws)>\ell(w) \end{array} \right.
\end{equation}
One can also define $s\star w$ analogously using $sw$ instead of $ws$.  
The Demazure product on the right and left generates right and left monoid actions respectively of the simple reflections on $W_{n+1}$.  These monoid actions are then used to construct the two-sided weak order on $W_{n+1}$.  The second major result of \cite{Shpairs} is to show that the extended monoid action on $B\backslash\B_{n+1}$ can be computed in terms of Shareshian pairs using Demazure products (see Theorem 4.8 of \emph{loc. cit.}).  This fact along with following theorem allows us to completely understand the monoid action of $\Pi_{\fg}\coprod\Pi_{\fg}$ on $G\backslash(\B_{n}\times\B_{n}\times\mathbb{P}^{n-1})$ in terms of the easily computed Demazure product on $W_{n+1}$.

\begin{thm}\label{thm:intromonoid}\emph{(see Theorem \ref{thm:monoidcorresp} and  Propositions \ref{p:orbitchange}, \ref{p:firstweirdmonoid},\,and \ref{p:unstablecorrespondence})}
The orbit correspondence in Equation (\ref{eq:bigintrocorresp}) intertwines the monoid action of $\Pi_{\fg}\coprod\Pi_{\fg}$ on $G\backslash(\B_{n}\times\B_{n}\times\mathbb{P}^{n-1})$ with the extended monoid action of $\Pi_{\fg}\coprod(\Pi_{\fg_{n+1}}\setminus\{\alpha_{1}\})$ acting on $B\backslash\mathfrak{X}$ after switching the factors of $\Pi_{\fg}$ and performing a shift in the second factor. 
\end{thm} 
\noindent Taken together, Theorems \ref{thm:introthmbig} and \ref{thm:intromonoid} imply that understanding the closure relations and monoid actions on $G\backslash (\B_{n}\times\B_{n}\times\mathbb{P}^{n-1})$ essentially reduce to understanding the product of the Bruhat orders and weak orders on $W_{n+1}\times W_{n+1}$ respectively. 

As was mentioned above, the complete set of simple roots of both $\fg$ and $\fg_{n+1}$ act on the orbits $B\backslash\mathfrak{X}$.  The correspondence in (\ref{eq:bigintrocorresp}) implies that there is a hidden action by $\alpha_{1}\in \Pi_{\fg_{n+1}}$ on $G\backslash (\B_{n}\times \B_{n}\times\mathbb{P}^{n-1})$ not covered by the standard monoid actions of the two copies $\Pi_{\fg}\coprod\Pi_{\fg}$ described in Theorem \ref{thm:intromonoid}.  We describe this hidden monoid action in Proposition \ref{p:alpha1corres} and provide an example of it in Example \ref{ex:Q13orbits}.  

The correspondence in (\ref{eq:bigintrocorresp}) is constructed
as follows.  First, observe that there is a natural correspondence:
\begin{equation} \label{eq:introtripledouble}
G\backslash(\B_{n}\times\B_{n}\times\mathbb{P}^{n-1}) \longleftrightarrow B\backslash (\B_{n}\times\mathbb{P}^{n-1})
\end{equation}
given by choosing the base point $\fb$ in the first factor. 
We can then study the $B$-orbits $B\backslash (\B_{n}\times\mathbb{P}^{n-1})$ by considering the projection to the second factor.  
Let $\{e_{1},\dots,  e_{n}\}$ be the standard basis of $\C^{n}$, and for $v\in \C^n$ nonzero, let $[v]=\C \cdot  v\in\mathbb{P}^{n-1}$ be the corresponding point in projective space.  The $n$ $B$-orbits on $\mathbb{P}^{n-1}$ are then given by $\mathcal{O}_{i}:=B\cdot[e_{i}]$ for $i=1,\dots, n$.  Denote by $S_{i}:=\mbox{Stab}_{B}[e_{i}]$ the stabilizer in $B$ of $[e_{i}]$.  We then have natural correspondences
\begin{equation}\label{eq:allthecorres}
B\backslash(B_{n}\times\mathbb{P}^{n-1})\longleftrightarrow\coprod_{i=1}^{n} B\backslash (\B_{n}\times\mathcal{O}_{i}) \longleftrightarrow \coprod_{i=1}^{n} S_{i}\backslash\B_{n}.
\end{equation}
The correspondence in (\ref{eq:bigintrocorresp}) is constructed by first establishing a correspondence between $S_{i}\backslash\B_{n}$ and $B\backslash Q_{1,i+1}$, where $Q_{1,i+1}\subset \mathfrak{X}$ is a $G$-orbit on $\B_{n+1}$.

\begin{thm}\label{thm:introlocalcorres}\emph{[Theorem \ref{thm:firstbig} and Remark \ref{r:corres}]}
There is a one-to-one correspondence between the orbit posets:
$$
S_{i}\backslash\B_{n}\longleftrightarrow B\backslash Q_{1,i+1}, \; Q\mapsto Q^{op}.
$$
\end{thm}
 The bijection in Equation (\ref{eq:bigintrocorresp}) is given by composing the bijection of Theorem \ref{thm:introlocalcorres} with the correspondences in (\ref{eq:introtripledouble}) and (\ref{eq:allthecorres}), using the decomposition of $\mathfrak{X}$ into $G$-orbits.   It would be interesting to interpret the bijection of (\ref{eq:bigintrocorresp}) in a uniform way without using the decomposition into orbits.
 
 The proofs of the above results depend on a careful analysis of $S_i$-orbits, which also establishes new results on $S_i$-orbits of independent interest.  By our results in \cite{Siorbits}, the $S_i$-orbits on $\B_n$ come equipped with a generalization of the extended monoid action for $B$-orbits on $\B_{n+1}$.  A major step in proving Theorem \ref{thm:intromonoid} consists of showing that the correspondence in Theorem \ref{thm:introlocalcorres} intertwines this monoid action with the monoid action on $B\backslash Q_{1,i+1}$ by $\Pi_{\fg}\coprod \Pi_{n+1, cpt}$, where $\Pi_{n+1, cpt}$ are the so-called compact imaginary roots for the $G$-orbit $Q_{1,i+1}$ (Theorem \ref{thm:monoidcorresp}).   Using this fact along with our work in \cite{CE21II}, we then show that every $S_{i}$-orbit is related in the weak order defined by the extended monoid action to a zero dimensional $S_{i}$-orbit which implies that the closure ordering on $S_{i}\backslash\B_{n}$ is given by the standard ordering of Richardson-Springer (Theorem \ref{thm:Siweak} and Corollary \ref{c:standardorder}).   As usual, we identify $\B_n$ with flags in $\C^n.$    Consider the $i$-cycle ${\sigma}_{i}:= (i,\,i-1,\dots, 2, \,1) \in W_n.$   For $w\in W_n,$ we also denote by $w$ the permutation matrix in $G$ of $w.$  Let $\mathcal{E}_n$ be the flag in $\B_n$ stabilized by $B$ and let $\mathcal{E}^{i} = \sigma_i \cdot \mathcal{E}_{n}$ with stabilizer $B^{i}=\sigma_i B \sigma_i^{-1}$ in $G.$  In Theorem 3.1 of \cite{Siorbits}, we prove that the $S_{i}$-orbits on $\B_{n}$ are of the form $Q=(B\cdot w \cdot \mathcal{E}_{n}) \cap (B^{i}\cdot u^{i} \cdot \mathcal{E}^{i})$  for unique permutations $w$ and $u^{i}$ in $W_{n}.$   We call $(w,u^{i})$ the $i$-\emph{Shareshian pair} of $Q.$   Using the fact that every $S_{i}$-orbit is related in the weak order to a zero dimensional orbit and an inductive argument using the monoid action, we can compute the $i$-Shareshian pair of an $S_{i}$-orbit $Q$ in terms of the Shareshian pair of the corresponding $B$-orbit $Q^{op}\in B\backslash Q_{1,i+1}$ (Theorem \ref{thm:localShareshian}).   Theorem \ref{thm:localShareshian} allows us to express the correspondence in (\ref{eq:bigintrocorresp}) in terms of Magyar's parameterization of $G\backslash(\B_{n}\times\B_{n}\times\mathbb{P}^{n-1})$ using decorated permutations (Theorem \ref{thm:globalSh} and Proposition \ref{p:marked}).  We then combine the basic result from \cite{Magyar} with some estimates of our own to prove (\ref{eq:introclosure}) (Proposition \ref{p:orderpreserving} and Theorem \ref{thm:globalclosure}). Using Theorem \ref{thm:localShareshian} along with Propositions \ref{p:orbitchange} and \ref{p:firstweirdmonoid}-\ref{p:alpha1corres}, we can understand the remaining monoid actions on  $G\backslash(\B_{n}\times\B_{n}\times\mathbb{P}^{n-1})$ and complete the proof of Theorem \ref{thm:intromonoid}.

    As another application of our results to $S_i$-orbits, in Corollary \ref{c:closure} of this paper, we prove that if we consider $S_i$-orbits $Q$ and $Q^{\prime}$ with $i$-Shareshian pairs $(w,u^{i})$ and $(y,v^{i}),$ then $Q^{\prime} \subset \overline{Q}$ if and only $y \le w$ and $\sigma_{i}^{-1}v^{i}\sigma_{i} \le  \sigma_{i}^{-1}u^{i}\sigma_{i}$ in the Bruhat order.  From this result, we deduce in Corollary \ref{c:Schubertintersect} that $\overline{Q} = \overline{B\cdot w(\mathcal{E}_{n})} \cap \overline{B^{i}\cdot u^{i}(\mathcal{E}^{i})},$ which means that the closure of $Q$ is the closure of the intersection of two Schubert varieties defined using different Borel subgroups.

 This paper is organized as follows.  In Section \ref{s:prelim}, we introduce notation and make some basic observations about $G$-orbits on $\B_{n} \times \B_{n} \times \mathbb{P}^{n-1}.$  In Section \ref{s:Gorbits}, we prove Theorem \ref{thm:introlocalcorres} and use it to deduce   Theorem  \ref{thm:globalcorresp}, which is a key ingredient in the proof of Theorem \ref{thm:introthmbig}.  In Section \ref{s:monoidactions}, we prove a number of results comparing monoid actions on different sides of the correspondence in (\ref{eq:bigintrocorresp})
and deduce consequences about the weak and standard orders of Richardson-Springer on $S_{i}\backslash\B_{n}$.  In Section \ref{s:Shareshian}, we prove further properties of $S_i$-orbits on $\B_{n}$ by using the notion of $i$-Shareshian pairs and our results on the monoid action. In Section \ref{s:global}, we discuss decorated permutations and use results of Section \ref{s:Shareshian} to express the correspondence in (\ref{eq:bigintrocorresp}) in terms of decorated permutations.  We use this to prove that (\ref{eq:bigintrocorresp}) preserves closure relations.
  We conclude the paper by providing graphs of the posets $G\backslash(\B_{n}\times \B_{n}\times\mathbb{P}^{n-1})$ for $n=2$ and $n=3,$ which illustrate the utility of our results.

 We would like to thank Martha Precup for useful discussions concerning the results of this paper.

\section{Preliminaries and Conventions}\label{s:prelim}
\subsection{Conventions}\label{ss:conven}
 All algebraic groups and varieties are complex, and the Lie algebra of an algebraic group is labelled by the corresponding fraktur letter.  For an algebraic group $A$ with Lie algebra $\fa$, we denote the adjoint action of $A$ on $\fa$ by $g\cdot x:=\Ad(g)x$ for $g\in A$ and $x\in\fa$.   If an algebraic group $A$ acts on a variety $X,$ we let $A\backslash X$ denote the set of $A$-orbits on $X.$ We regard $A\backslash X$ as a poset by the rule that for $Q_1, Q_2 \in A\backslash X$, $Q_1 \le Q_2$ iff $Q_1 \subset \overline{Q_2}.$  
 
 For $m > 0,$  let $\{e_{1},\dots, e_{m}\}$ denote the standard basis of $\C^{m}$. For a nonzero vector $v$ of $\C^m,$ we let $[v]$ be the corresponding point in $\mathbb{P}^{m-1}.$   For an ordered linearly independent sequence of nonzero vectors $(v_1, v_2, \dots, v_m)$, let $V_i$ be the span of $\{ v_1, v_2, \dots, v_i \}$ for $i=1, \dots, m,$
and let $(v_1 \subset v_2 \subset \dots \subset v_m)$ denote the flag $V_1 \subset V_2 \subset \dots \subset V_m.$  Let 
\[
\mathcal{E}_m = (e_1 \subset e_2 \subset \dots \subset e_m)
\] 
denote the standard flag in $\C^m.$   Let $B_m$ be the Borel subgroup of $GL(m)$ of upper triangular matrices of $GL(m)$, and let $\B_{m}$ be the flag variety of $GL(m)$.  As usual, we identify $\B_{m}=GL(m)/B_m$ with the variety of flags of $\C^m$ and also with the Borel subalgebras of $\fg\fl(m),$ using the natural linear action on flags and the adjoint action on the variety of Borel subalgebras.

 Let $H_{m}\subset G_{m}$ be the standard Cartan subgroup of $m\times m$ invertible diagonal matrices.  We identify the Weyl group $W_m=N_{G_{m}}(H_{m})/H_{m}$ of $G_{m}$ with the symmetric group on $m$ letters.  If $\tilde{\fb}\in \B_{m}$ is an $H_{m}$-stable Borel subalgebra of $\fg_{m}$ which stabilizes the flag $\tilde{\mathcal{E}}$, then $\Ad(\dot{w}) \tilde{\fb}$ is independent of the choice of representative $\dot{w}$ of $w\in W_{m}$.  We write $\Ad(w)\tilde{\fb}$ or simply $w(\tilde{\fb})$ for the Borel subalgebra stabilizing the flag $w(\tilde{\mathcal{E}})$.

 Let $\Pi_{\fg_{m}}=\{ \alpha_1, \dots, \alpha_{m-1} \}$ be the simple roots for the $\fh_{m}$-action on $\fb_{m}$, and let $S=\{ s_{1}, \dots, s_{m-1} \}$ be the corresponding simple reflections in $W_{m}.$

For a fixed integer $n,$ we let $U_n$ be the span of $\{ e_1, \dots, e_n \}$ in $\C^{n+1},$ and we let $G\cong GL(n)$ be the subgroup of $GL(n+1)$ that stabilizes the subspace $U_n$ and the fixes the vector $e_{n+1},$ which embeds $G$ into $GL(n+1)$ as the upper left hand corner.  We let $B=B_{n}.$ 
   For $i=1, \dots, n,$ we call the vector
\begin{equation}\label{e:hatvectordef}
\hat{e}_{i}=e_i + e_{n+1}
\end{equation}
a hat vector of index $i.$

\subsection{$B$-orbits on $\B_n \times \mathbb{P}^{n-1}$}

 We seek to understand the orbits of $G$ on 
the product $\B_{n}\times\B_{n}\times\mathbb{P}^{n-1}$ under the diagonal action.  
 The orbits of  $G\backslash (\B_{n}\times\B_{n}\times\mathbb{P}^{n-1})$ are in one-to-one correspondence with 
$B$-orbits on the product $\B_{n}\times \mathbb{P}^{n-1}$ by choosing the base point $\fb$  in the first factor: 
\begin{equation}\label{eq:Gdeltacorres}
G\backslash(\B_{n}\times\B_{n}\times\mathbb{P}^{n-1})\longleftrightarrow B\backslash (\B_{n}\times \mathbb{P}^{n-1}) \mbox{ given by } G\cdot (\fb,\fb^{\prime},[v])\leftrightarrow B\cdot (\fb^{\prime}, [v]).
\end{equation}  
 There are $n$ $B$-orbits on $\mathbb{P}^{n-1}$  given by $\mathcal{O}_{i}:=B\cdot[e_{i}]$ for $i=1,\dots, n.$
  We can decompose the set $B\backslash (\B_{n}\times\mathbb{P}^{n-1})$ as
\begin{equation}\label{eq:orbitspacedecomp}
B\backslash (\B_{n}\times\mathbb{P}^{n-1})=\bigcup_{i=1}^{n} B\backslash (\B_{n}\times\mathcal{O}_{i}).
\end{equation}
Fix an $i$ with $1\leq i \leq n$ and consider the $B$-orbits on $\B_{n}\times\mathcal{O}_{i}$.  Let
$S_{i}:=\mbox{Stab}_{B}[e_{i}]$ and note that 
\begin{equation}\label{eq:dimfacts}
\dim(\mathcal{O}_{i})=i-1,\, \dim(S_i)=\dim(B)-i+1, \text{and} \dim(B\cdot e_i)=i.
\end{equation}
  Then we have a one-to-one correspondence 
\begin{equation}\label{eq:Sicorres}
B\backslash (\B_{n}\times\mathcal{O}_{i})\longleftrightarrow S_{i}\backslash \B_{n} \mbox{ given by } B\cdot(\fb^{\prime},[e_{i}])\longleftrightarrow S_{i}\cdot\fb^{\prime}.
\end{equation}
\begin{rem}\label{r:dimcorr}
In these correspondences, the stabilizers coincide, i.e.,
\begin{equation}
(G)_{(\fb, \fb^{\prime}, [e_i])}=B_{(\fb^{\prime},[e_i])}=(S_i)_{\fb^{\prime}}.
\end{equation} 
It follows from (\ref{eq:dimfacts}) that
\begin{equation}\label{e:dimcorr}
\dim(G\cdot (\fb, \fb^{\prime}, [e_i]))=\dim(B\cdot (\fb^{\prime},[e_i])) + \dim(G/B) = \dim(S_i \cdot \fb^{\prime}) + \dim(G/B) + i-1.
\end{equation}
\end{rem}
\noindent One purpose of this paper is to study the collection of orbits $S_{i}\backslash\B_{n}$.  
\begin{rem}\label{r:specialcase}
We note that if $i=1$, then $S_{1}=B$ and the $B$-orbits on $\B_{n}$ are given by the Bruhat decomposition and their closure ordering is given by the Bruhat ordering on $W_{n}$.  If $i=n$, then $S_{n}=B_{n-1}$ (up to centre) and the structure of $B_{n-1}\backslash\B_{n}$ is described extensively in \cite{Hashi}, \cite{CE21I}, \cite{CE21II}, and \cite{Shpairs}. 
\end{rem}

\section{Correspondence Between $G$-orbits on $\B_{n}\times\B_{n}\times\mathbb{P}^{n-1}$ and $B$-orbits on $\mathfrak{X}$}\label{s:Gorbits}
 In Theorem 3.10 of \cite{CE21I}, we prove that the $B$-orbits on the flag variety $\B_{n+1}$ of $G_{n+1}$ contained in the open $G$-orbit on the flag variety $\B_{n+1}$ are in one-to-one correspondence with $B_{n-1}$-orbits on $\B_{n}.$ 
 In this section, we prove Theorem \ref{thm:introlocalcorres}, which is a generalization of Theorem 3.10 in \cite{CE21I} by showing that for any $i\in\{1,\dots, n\},$ the $S_{i}$-orbits on $\B_{n}$ correspond to $B$-orbits on a $G$-orbit we call $Q_{1,i+1}$ in the flag variety of $\B_{n+1}$ of $G_{n+1}$.  


To prove Theorem \ref{thm:introlocalcorres}, we need to use specific representatives of the $G$-orbits on $\B_{n+1}$.  We now briefly remind the reader of the classification of these orbits.  For a more detailed explanation of the theory, see \cite{Collingwood}, Example 10.5 in \cite{RSexp}, and \cite{CEexp} amongst other sources. 
\subsection{Classification of $G$-orbits on $\B_{n+1}$}
Using the embedding of $G=GL(n)$ in $GL(n+1)$ from Section \ref{ss:conven},  we get a $G$-action on $\B_{n+1}.$   These orbits are given as follows.


\begin{prop}[\cite{CEexp}, Section 4.4] \label{prop:typeAflag}
Let $\fg_{n+1}=\fgl(n+1)$ and $\fg=\fgl(n)$.  
\begin{enumerate}
\item The number of $G$-orbits on $\B_{n+1}$ is ${n+2\choose 2}$.  
\item For $i=1,\dots, n+1$, let $w_{i}$  be the cycle $(n+1, n, \dots, i)$ in the symmetric group $W_{n+1}$, 
and let $\fb_{i,i}:=w_{i}(\fb_{n+1})$.  The distinct closed $G$-orbits on $\B_{n+1}$ are the $G$-orbits $Q_{i,i}=G\cdot \fb_{i,i}$ so there are exactly $n+1$ closed $G$-orbits.  Further, the Borel subalgebra $\fb_{i,i}$ is 
the stabilizer of the flag:
\begin{equation}\label{eq:flagi}
\mathcal{E}_{i,i}:=(e_{1}\subset\dots \subset e_{i-1}\subset\underbrace{e_{n+1}}_{i}\subset e_{i}\subset \dots \subset e_{n}). 
\end{equation}
\item The non-closed $G$-orbits are of the form $Q_{i,j}=B\cdot \fb_{i,j}$ 
with $1\leq i< j\leq n+1$ and where the Borel subalgebra $\fb_{i,j}$ is the stabilizer of the flag 
\begin{equation}\label{eq:flagij}
\mathcal{E}_{i,j}:=(e_{1}\subset\dots\subset e_{i-1} \subset \underbrace{\hat{e}_{j-1}}_{i}\subset e_{i}\subset\dots\subset e_{j-2}\subset\underbrace{e_{n+1}}_{j}\subset e_{j}\subset\dots\subset e_{n}),
\end{equation}
where $\hat{e}_{j-1}=e_{j-1}+e_{n+1}$ is defined in Equation (\ref{e:hatvectordef}).
Further, the codimension of $Q_{i,j}$ is $n-(j-i)$.
In particular, the unique open orbit is $Q_{1,n+1}$, and it contains the Borel subalgebra 
which stabilizes the flag:
\begin{equation}\label{eq:openflag}
\mathcal{E}_{1,n+1}:=(\hat{e}_{n}\subset e_{1}\subset\dots\subset e_{n-1}\subset e_{n+1}). 
\end{equation}
\end{enumerate}
\end{prop}

\begin{rem}\label{r:Kclosure}
It is well-known that 
$$
Q_{i,j}\subset\overline{Q_{k,\ell}}\Leftrightarrow i\geq k\mbox{ and } j\leq \ell.
$$
\end{rem}\par\noindent

 The following diagram indicates the $GL(n)$-orbits on $\B_{n+1}$ for the case where $n=3$, together with the order relation given by closure, where $Q^{\prime} \subset \overline{Q}$ if and only if there is a sequence of lines going down from $Q^{\prime}$ to $Q.$  For general $n$, the diagram has the same shape, with $n+1$ closed orbits, $n$ orbits one line below, and so forth, until we have one orbit on the last line.

\hspace{11em}
\begin{equation}\label{eq:Kdiag}
\begin{tikzpicture} 
  [scale=1.3,auto=center,every node/.style={circle,fill=white!20}] 
\node (a1) at (5,4) {$Q_{1,1}$};
\node (a2) at (6,4) {$Q_{2,2}$};
\node (a3) at (7,4) {$Q_{3,3}$};
\node (a4) at (8,4) {$Q_{4,4}$};
\node (a5) at (5.5,3) {$Q_{1,2}$};
\node (a6) at (6.5,3) {$Q_{2,3}$};
\node (a7) at (7.5,3) {$Q_{3,4}$};
\node (a8) at (6,2) {$Q_{1,3}$};
\node (a9) at (7,2) {$Q_{2,4}$};
\node (a10) at (6.5,1) {$Q_{1,4}$};

\draw (a1) -- (a5); 
  \draw (a2) -- (a5);  
  \draw (a2) -- (a6);  
  \draw (a3) -- (a6);  
  \draw (a3) -- (a7);  
  \draw (a4) -- (a7);  
  \draw (a5) -- (a8);  
\draw (a6) -- (a8);  
\draw (a6) -- (a9); 
\draw (a7) -- (a9);  
\draw (a8) -- (a10);  
\draw (a9) -- (a10);  
 
\end{tikzpicture}
\end{equation}

\subsection{Correspondence between $S_{i}\backslash\B_{n}$ and $B\backslash Q_{1,i+1}$}

We can now state the first basic result of the paper. 

\begin{thm}\label{thm:firstbig}
Let $i\geq 1$ and consider the $G$-orbit $Q_{1,i+1}=G\cdot \mathcal{E}_{1,i+1}\subset \B_{n+1}$, 
where we take $i=1$ and $j=i+1$ in part (3) of Proposition \ref{prop:typeAflag}, so that the flag
\begin{equation}\label{eq:Fioneflag}
\mathcal{E}_{1,i+1}=(\hat{e}_{i}\subset e_{1}\subset\dots\subset e_{i-1}\subset\underbrace{e_{n+1}}_{i+1}\subset e_{i+1}\subset \dots \subset e_{n}), 
\end{equation}
with $\hat{e}_{i}=e_{i}+e_{n+1}$ (see (\ref{e:hatvectordef})).  Then there is a one-to-one correspondence between $B$-orbits contained in the $G$-orbit $Q_{1,i+1}$ and $S_{i}$-orbits on $\B_{n}$, 
\begin{equation}\label{eq:firstcorresp}
B\backslash Q_{1,i+1}\longleftrightarrow S_{i}\backslash\B_{n}, \text{ given by }
Bg\, \mbox{\emph{Stab}}_{G}(\mathcal{E}_{1,i+1}) \to S_ig^{-1}B.
\end{equation}
\end{thm}
\begin{proof}
We claim that 
\begin{equation}\label{eq:stabclaim}
\mbox{Stab}_{G}(\mathcal{E}_{1,i+1})=\mbox{Stab}_{B}(e_{i}).
\end{equation}
Suppose first that $g\in G$ fixes the flag $\mathcal{E}_{1,i+1}$ in (\ref{eq:Fioneflag}).  For $j=1,\dots, n$, let $U_{j}=\mbox{span}\{e_{1},\dots, e_{j}\}$.  Recall that $G=\{g\in G_{n+1}:\, g\cdot e_{n+1}=e_{n+1}\mbox{ and } g\cdot U_{n}=U_{n}\}.$  Since $g\in \mbox{Stab}_{G}(\mathcal{E}_{1,i+1})$, it must fix the line 
$\C(\hat{e}_{i})=\C(e_{i}+e_{n+1})$.  But since $g\cdot e_{n+1}=e_{n+1}$, we deduce that $g\cdot e_{i}=e_{i}.$
Further, for any $j\in\{1,\dots, i-1\}$, $g\cdot U_{j}\subset \C(e_{i}+e_{n+1})+U_{j}$ and 
for $j\in \{i+1,\dots, n\}$, $g\cdot U_{j}\subset (U_{j}+\C e_{n+1})$.  But $U_{n}\cap (\C(e_{i}+e_{n+1})+ U_{j})=U_{j}$ and $U_{n}\cap (U_{j}+\C e_{n+1})=U_{j}$, so it follows that $g\cdot U_{j}=U_{j}$ for all $j=1,\dots, n$.  Thus, $g\in \mbox{Stab}_{B}(e_{i})$.  Conversely, any element of the group $\mbox{Stab}_{B}(e_{i})$ can be easily seen to stabilize the flag $\mathcal{E}_{1,i+1}$.   This establishes (\ref{eq:stabclaim}) and it follows easily that $Z\cdot \mbox{Stab}_{G}(\mathcal{E}_{1,i+1})=\mbox{Stab}_{B}[e_{i}]=S_{i},$ where $Z$ is the centre of $G.$

Hence,  the $B$-orbits in the $G$-orbit $Q_{1,i+1}$ are in one-to-one correspondence with the elements of the double coset poset  $B\backslash G/ \mbox{Stab}_{B}(e_{i})$.  This double coset poset is then in one-to-one correspondence with the poset of $S_{i}$-orbits on $G/B=\B_{n}$ via the self map of $G$ given by inversion $g\mapsto g^{-1}$, which implies the Theorem.

\end{proof}
\begin{nota}\label{n:op}
For $\cZ\in B\backslash Q_{1,i+1}$, we denote the corresponding $S_{i}$-orbit on $\B_{n}$ given by Theorem \ref{thm:firstbig} by $\cZ^{op}$.  
\end{nota}

\begin{rem}\label{r:corres}
 We can realize this correspondence geometrically as follows.  Let $q: G\to Q_{1,i+1}\cong G/(\mbox{Stab}_{B}(e_{i}))$ and $\pi: G\to \B_{n}$ be the natural projections, and let $\psi: G\to G$ be the inversion map, i.e., $\psi(g)=g^{-1}$.  Then 
\begin{equation*}
\cZ^{op}=\pi(\psi(q^{-1}(\cZ))).
\end{equation*}
More generally, if $\mathcal{Y}\subset Q_{1,i+1}$ is any $B$-stable subvariety, it 
corresponds to an $S_{i}$-stable subvariety of $\B_{n}$ in the same manner:
\begin{equation}\label{eq:Yop}
\mathcal{Y}^{op}:=\pi(\psi(q^{-1}(\mathcal{Y}))).
\end{equation}
Using this last equation and Equation (\ref{eq:dimfacts}), we note that 
\begin{equation}\label{eq:opdim}
\begin{split}
\dim \mathcal{Y}^{op}&=\dim \mathcal{Y}+\dim(\mbox{Stab}_{B}(e_{i}))-\dim B\\
&=\dim\mathcal{Y}-i.
\end{split}
\end{equation}
Further, the correspondence preserves closure relations.  That is to say, 
\begin{equation}\label{eq:closurecorres}
(\overline{\mathcal{Y}})^{op}=\overline{\mathcal{Y}^{op}}.
\end{equation}
\end{rem}

Combining the correspondence in Equation (\ref{eq:firstcorresp}) with the one in Equation (\ref{eq:Sicorres}), we obtain one-to-one correspondences between the following orbit posets for any $i=1,\dots, n$:
\begin{equation}\label{eq:biggercorresp}
\begin{split}
 B\backslash(\B_{n}\times \mathcal{O}_{i})\longleftrightarrow S_{i}\backslash\B_{n}\longleftrightarrow B\backslash Q_{1,i+1},\\
 B\cdot (gB,[e_i]) \longleftrightarrow S_i\cdot gB \longleftrightarrow B\cdot g^{-1}B_{1,i+1},
\end{split}  
\end{equation}
where $B_{1,i+1}\subset GL(n+1)$ is the Borel subgroup of $GL(n+1)$ stabilizing the flag $\mathcal{E}_{1,i+1}.$   Combining (\ref{eq:biggercorresp}) with Equation (\ref{eq:orbitspacedecomp}), we obtain a one-to-one correspondence between the set of all $B$-orbits on $\B_{n}\times\mathbb{P}^{n-1}$ and all $B$-orbits on the subvariety $\mathfrak{X}$ of $\B_{n+1}$ given by the following union of $G$-orbits on $\B_{n+1}$:
\begin{equation}\label{eq:frakX}
\mathfrak{X}:=\displaystyle\bigcup_{i=1}^{n} Q_{1,i+1}.
\end{equation}
The variety $\mathfrak{X}$ is the union of the orbits on the leading left edge of the Bruhat graph with the exception of the closed orbit $Q_{1,1}$ (see (\ref{eq:Kdiag}) for $n=3$ case).  The following theorem is a consequence of the preceding remarks.  

\begin{thm}\label{thm:globalcorresp}
There is a one-to-one correspondence between $B$-orbits on $\B_{n}\times \mathbb{P}^{n-1}$ and $B$-orbits on the open subvariety $\mathfrak{X}=\cup_{i=1}^{n} Q_{1,i+1}$.  Combining this fact with the correspondence in (\ref{eq:Gdeltacorres}), we obtain correspondences 
\begin{equation}\label{eq:secondcorresp}
G\backslash(\B_{n}\times\B_{n}\times\mathbb{P}^{n-1})\longleftrightarrow B\backslash (\B_{n}\times\mathbb{P}^{n-1} )\longleftrightarrow B\backslash \mathfrak{X}.
\end{equation}
\end{thm}

 It would be interesting to interpret the bijection between $G\backslash(\B_{n}\times\B_{n}\times\mathbb{P}^{n-1})$ and $  B\backslash \mathfrak{X}$ in a uniform way without using the decomposition into $S_i$-orbits.

\begin{rem}\label{rem:orbitlabels}
We have introduced bijections between orbits in four different contexts.  To make the context clear, we use the notations
\begin{equation} \label{eq:three}  
\mathcal{O}_{\Delta,\fb^{\prime}}:=G\cdot (\fb, \fb^{\prime}, [e_{i}])\longleftrightarrow \mathcal{O}_{B,\fb^{\prime}}:=B\cdot (\fb^{\prime}, [e_{i}])\longleftrightarrow Q_{\fb^{\prime}}:=S_{i}\cdot \fb^{\prime},
\end{equation}
and for a $S_i$-orbit $Q=Q_{g\cdot \fb}$ on $\B_n,$ we denote the corresponding $B$-orbit on $Q_{1,i+1}$ by $Q_{g\cdot \fb}^{op}={\cZ}_{g^{-1}\cdot \fb_{1,i+1}}:=Bg^{-1}\cdot \fb_{1,i+1}.$  More generally, we use labels $\mathcal{O}_{\Delta, \{-\}}$ for diagonal $G$-orbits, $\mathcal{O}_{\{-\}}$ for $B$-orbits on $\B_n \times \mathbb{P}^{n-1},$ $Q_{\{-\}}$ for $S_i$-orbits, and ${\cZ}_{\{-\}}$ for $B$-orbits on $\mathfrak{X}$.
\end{rem}

\begin{rem}\label{rem:oppusage}
For an $S_i$-orbit $Q_i$ on $\B_{n}$ and a $B$-orbit ${\cZ}_{i}$ on ${\mathfrak{X}}$ which correspond to each other in the above correspondence, we denote ${\cZ}_{i}=Q_{i}^{op}$ and $Q_{i}={\cZ}_{i}^{op}$ interchangeably.  The reader can use the orbit labels $Q$ and $\cZ$ to distinguish between these two usages.  If $Q_{i}$ also corresponds to a $G$-orbit $\mathcal{O}_{\Delta,i}$ or a $B$-orbit $\mathcal{O}_{i}$ on $\B_n \times \mathbb{P}^{n-1},$ then we also write ${\cZ}_{i}=\mathcal{O}_{\Delta, i}^{op}$ and ${\cZ}_{i}=\mathcal{O}_{i}^{op}$ or $\mathcal{O}_{\Delta, i} ={\cZ}_{i}^{op},$ etc.,  in different instances.    The reader should use the orbit labels $Q,$ $\mathcal{O}_{\Delta},$ $\mathcal{O},$ and $\cZ$ to determine the specific context.
\end{rem}


\section{Orbit Correspondence and Monoid Actions}\label{s:monoidactions}
\subsection{Background on Monoid Actions}\label{ss:monoidbackground}

For more details on the subsequent material, we refer the reader to \cite{RS}, \cite{Vg}, \cite{CEexp}, \cite{CE21I} and other sources.
Let $R$ be a connected reductive algebraic group, let $\B=\B_{R}$ be the flag 
variety of $R$, and let $M$ be an algebraic subgroup of $R$ acting on $\B$ with finitely many orbits.  Identify $\B \cong R/B_{R}$, for a Borel subgroup $B_{R}\subset R$ and let $\fb_{\fr}=\mbox{Lie}(B_{R})\subset\fr$.  
Let $\Pi_{\fr}$ be the set of simple roots defined by the Borel subalgebra $\fb_{\fr}$, and let $S_{R}$ be the simple reflections of the Weyl group $W$ of $R$ corresponding to $\Pi_{\fr}$.  For $\alpha\in \Pi_{\fr}$, let $\mathcal{P}_{\alpha}$ be the variety of all parabolic subalgebras of $\fr$ of type $\alpha$ and consider the $\mathbb{P}^{1}$-bundle $\pi_{\alpha}:\B \to {\mathcal{P}}_{\alpha}$.  
For $\alpha \in \Pi_{\fr}$ with corresponding reflection $s=s_{\alpha} \in W$, we define an operator $m(s)$ on the set of orbits $M\backslash \B$ following the above sources.  For $Q_M \in M\backslash \B$, let $m(s)*Q_{M}$ be the unique $M$-orbit open and dense in $\pi_{\alpha}^{-1}(\pi_{\alpha}(Q_{M})).$  Note that $\pi_{\alpha}: \pi_{\alpha}^{-1}(\pi_{\alpha}(Q_{M}))\to \pi_{\alpha}(Q_{M})$ is an $M$-equivariant 
$\mathbb{P}^{1}$-bundle. Thus, $\dim(\pi_{\alpha}^{-1}(\pi_{\alpha}(Q_{M})))=\dim(\pi_{\alpha}(Q_{M}))+1,$ and since $Q_M \subset \pi_{\alpha}^{-1}(\pi_{\alpha}(Q_{M}))$, it follows that the orbit 
\begin{equation}\label{eq:basicmonoid} 
Q_{M} \neq m(s)*Q_{M}  \text{ if and only if } \dim(m(s)*Q_{M})=\dim(Q_{M})+1.
\end{equation}


Computation of $m(s)*Q_{M}$ depends on the \emph{type of the root} $\alpha$ for the orbit $Q_{M}$, which is determined as follows.  For $\fb^{\prime}\in Q_M$, let $\fp_{\alpha}^{\prime} = \pi_{\alpha}(\fb^{\prime})$ and let $B^{\prime}$ and $P_{\alpha}^{\prime}$ be the corresponding parabolic subgroups of $R$, and let $V_{\alpha}^{\prime}$ be the the solvable radical of $P_{\alpha}^{\prime}$.  Consider the group 
$S_{\alpha}^{\prime}:= P_{\alpha}^{\prime}/V_{\alpha}^{\prime}$ isogenous to $SL(2)$ and its subgroup
\begin{equation}\label{eq:Malpha}
M_{\alpha, \fb^{\prime}}:= (M \cap P_{\alpha}^{\prime})/(M\cap V_{\alpha}^{\prime}).
\end{equation}
Using the standard identification $\pi_{\alpha}^{-1}(\pi_{\alpha}(Q)) \cong M\times_{M\cap P_{\alpha}^{\prime}} P_{\alpha}^{\prime}/B^{\prime},$ we see that $M$-orbits in
 $\pi_{\alpha}^{-1}(\pi_{\alpha}(Q))$ correspond to $M_{\alpha, \fb^{\prime}}$-orbits in 
$P^{\prime}_{\alpha}/B^{\prime}\cong {\PR}^1.$  From the classification of subgroups of $SL(2)$ with finitely many orbits on ${\PR}^1,$ we are in one of the following three cases.

\begin{dfn}\label{d:roottype}
\par\noindent (1) If $M_{\alpha, \fb^{\prime}}$ is solvable and contains the unipotent radical of a Borel subgroup of $S_{\alpha}^{\prime}$, then $\alpha$ is called a complex root for $Q_M$.  If $M_{\alpha, \fb^{\prime}}\cdot \fb^{\prime}=\fb^{\prime}$, then $\alpha$ is complex stable for $Q_M$ and otherwise $\alpha$ is complex unstable for $Q_M$.
\par\noindent (2) If $M_{\alpha, \fb^{\prime}}=S_{\alpha}^{\prime}$, then $\alpha$ is called a compact imaginary root for $Q_M.$
\par\noindent (3) Suppose $M_{\alpha, \fb^{\prime}}$ is one-dimensional and reductive.
If $M_{\alpha,\fb^{\prime}}\cdot \fb^{\prime}=\fb^{\prime}$, then $\alpha$ is called a noncompact imaginary root for $Q_M$, while if $M_{\alpha,\fb^{\prime}}\cdot \fb^{\prime}\not=\fb^{\prime},$ then $\alpha$ is called a real root for $Q_M.$
\end{dfn}

In \cite{RS}, the monoid action is only discussed when $M$ has the same component group as the fixed points of an involution of $R$, and the terminology in this definition comes from the action of the involution on root spaces, or more precisely the action of an associated real form of $R$.   However, as many authors have observed, including \cite{Knoporbits}, the construction works in the same way when $M$ has finitely many orbits on the flag variety.  
 It follows easily from arguments in \cite{RScomp} that the monoid action depends only on the nature of the groups $M_{\alpha, \fb^{\prime}}.$   This is the perspective we use in this paper. 
It is well-known and easy to prove that the type of the root depends only on the orbit $Q_M$ and not on the point $\fb^{\prime}.$ The following result is also well-known. 

\begin{prop}\label{p:stableandnc} \emph{[Lemma 2.1.4 and Lemma 2.4.3 of \cite{RScomp}]}

Let $Q\in M\backslash\B$ with $Q=M\cdot \Ad(v)\fb_{\fr}$ and $\alpha\in \Pi_{\fr}$.
\begin{enumerate}
\item If $\alpha$ is complex stable for $Q$, then $\pi_{\alpha}^{-1}(\pi_{\alpha}(Q))$ consists of two $M$-orbits:
$$
\pi_{\alpha}^{-1}(\pi_{\alpha}(Q))=Q\cup M\cdot \Ad(vs_{\alpha})\fb_{\fr},
$$
and $\ms*Q= M\cdot \Ad(vs_{\alpha})\fb_{\fr}$. 
\item If $\alpha$ is non-compact imaginary for $Q$ then $\pi_{\alpha}^{-1}(\pi_{\alpha}(Q))$ consists of either two (resp. three) $M$-orbits depending on whether $M_{\alpha, \fb^{\prime}}$ is the normalizer of a torus (resp. a torus). The open orbit
$\ms*Q=M\cdot \Ad(vu_{\alpha})\fb_{\fr}$, where $u_{\alpha}\in R$ is the Cayley transform with respect to the root $\alpha$ as defined in Equation (41) of \cite{CEexp} and $\dim Q=\dim (M\cdot \Ad(vs_{\alpha})\fb_{\fr})$.   
\item In all other cases, $m(s_{\alpha})*Q=Q.$
\end{enumerate}
\end{prop}
\begin{rem}\label{r:solvablecase} 
When $M$ is a solvable group, the group $M_{\alpha, \fb^{\prime}}$ in (\ref{eq:Malpha}) is always solvable.  In this setting, the root $\alpha$ is \emph{never} compact imaginary, and in the non-compact imaginary case,  $M_{\alpha,\fb^{\prime}}$ is a torus so that $\pi_{\alpha}^{-1}(\pi_{\alpha}(Q))$ always consists of three $M$-orbits. 
\end{rem}

\begin{rem}\label{r:Cayley}
Let $\fg=\fgl(n)$ and $\fb_{\fr}=\fb$.  Suppose that 
$\Ad(g)\fb\in\B_{n}$ stabilizes the flag $\F=(v_{1}\subset \dots\subset v_{i-1}\subset v_{i}\subset v_{i+1}\subset \dots\subset v_{n}).$  If $\alpha=\alpha_{i},$
 then a computation with flags 
show that $\Ad(gu_{\alpha})\fb$ stabilizes the flag 
$\F^{\prime}=(v_{1}\subset\dots\subset v_{i-1}\subset v_{i}+v_{i+1}\subset v_{i+1}\subset \dots \subset v_{n}).$ 
\end{rem}


Using the monoid action, we can define partial orderings on the set of $M$-orbits 
$M\backslash\B$ that can be used to study the closure relations between orbits.  
Given a sequence $\vec{s}=(s_{1},\dots, s_{k})$ of elements in $S_{R}$ and 
an $M$-orbit $Q\in M\backslash\B$, we let 
\begin{equation} \label{eq:Ssequence}
m(\vec{s})*Q:=m(s_{k})*\dots *m(s_{1})*Q \mbox{ if } k > 0, m(\vec{s})*Q=Q \mbox{ if } k = 0.
\end{equation}
 Then $\mathfrak{M}:=\{ m(\vec{s}): \vec{s} \mbox{ a sequence} \} $ is a finite monoid with $1$, which follows from the well-known fact that the operators $m(s)$ satisfy  braid relations and also the relations $m(s)^{2}=m(s)$ for $s\in S.$  It follows that this monoid generates an action by the $0$-Hecke algebra on the group algebra of $W.$

\begin{dfn}\label{d:weak}
The \emph{weak order } $\leq_w$ is defined by the property that if $Q, Q^{\prime} \in M\backslash\B$, then $Q \leq_w Q^{\prime}$ if and only if $Q^{\prime}=m(\vec{s})*Q$ for some sequence $\vec{s}$ as above.  This is the weakest partial order such that $Q$ is less than or equal to $m(s)*Q$ for each $s\in S_{R}$.
\end{dfn}

The weak order can then be used to define the standard order of Richardson-Springer constructed in \cite{RS}, which is a combinatorially defined order which coincides in many cases with the ordering on orbits given by inclusion of orbit closures.  We refer to the reader to Section 6.1 of \cite{CE21II} or \cite{RS} for more details.

\subsection{Monoid Actions on $G$-orbits on $\B_{n+1}$}\label{ss:Gmonoid}

We apply the general theory of Section \ref{ss:monoidbackground} to the case where $R=G_{n+1}$ and $M=G=GL(n)$.  The $G$-orbits on $\B_{n+1}$ are described in Proposition \ref{prop:typeAflag} above.   We let $\fb_{\fr}=\fb_{n+1}$  be the standard Borel subalgebra of $(n+1)\times (n+1)$ upper triangular matrices.  The monoid action by $\Pi_{\fg_{n+1}}$ on the set of orbits $G\backslash\B_{n+1}$ is described in detail in Examples 4.24 and 4.30 of \cite{CEexp}.  However, we only need to consider monoid actions for the orbits $Q_{1,i+1}$ with $i\in\{1,\dots, n\}$ for this paper. Let $s_k=s_{\alpha_k}$ for a simple root $\alpha_k.$
\begin{prop}\label{p:Kmonoidaction}\emph{(cf. Part (1) of Proposition 4.8 of \cite{CE21I})}
Consider the $G$-orbit $Q_{1,i+1}$ through the flag $\mathcal{E}_{1,i+1}$ in (\ref{eq:Fioneflag}) for $i\in \{1,\dots, n\}$.  
\begin{enumerate}
\item The root $\alpha_{1}$ is real for $Q_{1,2}$ and complex unstable for $Q_{1,i+1}$ for all $i>1$.
\item The root $\alpha_{i}$ is complex unstable for $Q_{1,i+1}$ for $i>1$.
\item The root $\alpha_{i+1}$ is complex stable for $Q_{1,i+1}$ for $1\leq i<n$ and $m(s_{i+1})*Q_{1,i+1}=Q_{1,i+2}.$
\item If $j\neq 1,\,i, \, \text{or } i+1,$ then $\alpha_{j}$ is compact imaginary for $Q_{1,i+1}$.  
\end{enumerate}
\end{prop}

\subsection{Monoid Actions on $S_{i}$-orbits}\label{ss:Simonoid}

The general theory of Section \ref{ss:monoidbackground} can also be applied to the setting of $G$-orbits on 
$\B_{n}\times\B_{n}\times\mathbb{P}^{n-1}$ and $S_{i}$-orbits on $\B_{n}$.  The $G$-orbits on $\B_{n}\times\B_{n}\times\mathbb{P}^{n-1}$ come equipped with a monoid action by $\Pi_{\fg}\times\Pi_{\fg}$, one copy of $\Pi_{\fg}$ for each factor of $\B_{n}$.  For $\alpha \in \Pi_{\fg},$ we denote the monoid action by the first and second factors by $\ms*_{1}$ and $\ms*_{2}$ respectively. 
We first consider the monoid action  $\ms*_{1}$ on a $G$-orbit $\mathcal{O}_{\Delta,\fb^{\prime}}$ on $\B_{n}\times\B_{n}\times\mathbb{P}^{n-1}$ in the notation of Remark \ref{rem:orbitlabels}.  Consider also the corresponding $B$-orbit $\mathcal{O}_{B,\fb^{\prime}}$ and $S_i$-orbit $Q_{\fb^{\prime}}=S_{i}\cdot \fb^{\prime}.$

By definition,
\begin{equation}\label{eq:firstfactorfirst}
\ms*_{1} \mathcal{O}_{\Delta,\fb^{\prime}}=\mbox{ open } G-\mbox{orbit in } G\cdot (P_{\alpha}\cdot\fb, \fb^{\prime}, [e_{i}])\longleftrightarrow \mbox{ open } B-\mbox{orbit in } P_{\alpha}\cdot (\fb^{\prime}, [e_{i}]).
\end{equation}
In general, the monoid action in (\ref{eq:firstfactorfirst}) does not preserve $S_{i}$-orbits.  In \cite{Siorbits}, we study the restriction of this action to a special subset of roots
\begin{equation}\label{eq:extraroots}
\mathfrak{S}_i:=\{\alpha_{1},\dots, \alpha_{i-2},\alpha_{i+1},\dots, \alpha_{n-1}\},
\end{equation}
and show in Equation 5.7 of \emph{loc. cit.} that for $\alpha\in \mathfrak{S}_i$ and $Q_{\fb^{\prime}}$ an $S_{i}$-orbit, we have 
\begin{equation}\label{eq:leftmonoidSi}
\ms*_{1} \mathcal{O}_{\Delta,\fb^{\prime}}\longleftrightarrow \ms*_{L}Q_{\fb^{\prime}}:=\mbox{ open } S_{i}-\mbox{orbit in } \mbox{Stab}_{P_{\alpha}}[e_{i}] \cdot\fb^{\prime}.
\end{equation}

We now consider the monoid action coming from the second factor. Let $\alpha\in\Pi_{\fg}$, and let $P_{\alpha}^{\prime}$ denote the parabolic of type $\alpha$ containing the Borel subgroup $B^{\prime}$ with Lie algebra $\fb^{\prime}.$ It follows from definitions that $\ms*_{2} \mathcal{O}_{\Delta,\fb^{\prime}}$ is the open $G$- orbit in  $G\cdot (\fb, P_{\alpha}^{\prime}\cdot \fb^{\prime}, [e_{i}])$ which corresponds to the open $B$-orbit in
$B\cdot (P_{\alpha}^{\prime}\cdot\fb^{\prime}, [e_{i}]),$ which corresponds to the open $S_i$-orbit in $S_{i}P_{\alpha}^{\prime}\cdot\fb^{\prime}.$
But the last orbit is exactly $\ms*Q_{\fb^{\prime}}$ from our discussion in Section \ref{ss:monoidbackground}.  
Summarizing, 
\begin{equation}\label{eq:secondfactormonoid}
\ms*_{2} \mathcal{O}_{\Delta,\fb^{\prime}}\longleftrightarrow \ms* Q_{\fb^{\prime}}=  \mbox{ the open } S_{i}-\mbox{orbit in } S_{i}P_{\alpha}^{\prime}\cdot\fb^{\prime}.
\end{equation}


\begin{dfn}\label{d:leftandright}\emph{(cf. Definition 5.3 of \cite{Siorbits})}
We refer to the monoid action by simple roots $\Pi_{\fg}$ on $S_{i}\backslash\B_{n}$ given in (\ref{eq:secondfactormonoid}) as the \emph{right monoid} action on $S_{i}$-orbits.  We refer to the monoid action by $\mathfrak{S}_i$ given in Equation (\ref{eq:leftmonoidSi}) as the \emph{left monoid} action on $S_{i}$-orbits. Taken together, we refer to the monoid action by $\mathfrak{S}_i\coprod \Pi_{\fg}$ as the \emph{extended monoid action}.  To emphasize the distinction between the left and right actions, we denote the monoid action in (\ref{eq:secondfactormonoid}) by $\ms*_R := \ms*.$
\end{dfn}

\begin{rem}\label{r:extended}
In the special case where $i=n$, the group $S_{i}$ coincides up to centre with $B_{n-1}\subset G_{n-1}$, the standard upper triangular Borel subgroup of $G_{n-1}$, and the set of roots 
$\mathfrak{S}_i=\{\alpha_{1},\dots,\alpha_{n-2}\}=\Pi_{\fg_{n-1}}$ are the simple roots of  the subalgebra $\fg_{n-1}\subset\fg$.  For $\alpha \in \mathfrak{S}_i,$ the corresponding standard parabolic subgroup $P^{n-1}_{\alpha}$ of $G_{n-1}$  is (up to centre) the subgroup  $\mbox{Stab}_{P_{\alpha}}[e_{n}].$  Thus, in this case the left monoid action of (\ref{eq:leftmonoidSi}) coincides with the left monoid action via roots of $\fg_{n-1}$ on $\Borbitspace$ defined in Section 4 of \cite{CE21I}.  Taken together, Equations (\ref{eq:secondfactormonoid}) and (\ref{eq:leftmonoidSi}) coincide with the extended monoid action by simple roots $\Pi_{\fg_{n-1}}\coprod\Pi_{\fg}$ on $\Borbitspace$ constructed in \emph{loc. cit.}.
\end{rem}

The following Remark generalizes Remark 4.4 of \cite{Shpairs} to the setting of $S_{i}$-orbits on $\B_{n}$ for arbitrary $i$. 
\begin{rem}\label{r:onlythree}
For the extended monoid action on $S_{i}\backslash\B_{n}$, a root $\alpha\in\mathfrak{S}_{i}\coprod \Pi_{\fg}$ is never compact imaginary in the sense of Definition \ref{d:roottype}, and in the non-compact case, the group $M_{\alpha,\fb^{\prime}}$ of (\ref{eq:Malpha}) is always a torus (see part (2) of Proposition \ref{p:stableandnc}).  For $\alpha\in \Pi_{\fg}$, the assertion follows immediately from Remark \ref{r:solvablecase}, since $S_{i}\subset G$ is a solvable group.  For $\alpha\in\mathfrak{S}_{i}$, the group $M_{\alpha, \fb^{\prime}}$ is a subquotient of $P_{\alpha}\cap B^{\prime}$ and therefore is also solvable and the assertion follows. 
\end{rem}

The following observation will be useful later in the paper.  Let $\alpha\in\Pi_{\fg_{n+1}}$ and $\cZ \in B\backslash\B_{n+1}$ with $\cZ\subset Q_{G}$, where $Q_{G}$ is a $G$-orbit on $\B_{n+1}$. Then by Proposition 4.7 of \cite{CE21I} along with Remark \ref{r:extended} for $\fg=\fgl(n+1)$ and $i=n+1$, we have 
\begin{equation}\label{p:old4.7}
 \ms*_{R}\cZ\subset \ms*Q_{G}.
\end{equation}  


\subsection{The Monoid Action and the Correspondence}\label{ss:monoidcorresp}

In this section, we relate the monoid action on $S_{i}\backslash\B_{n}$ to a partially defined monoid action on $B\backslash Q_{1,i+1}$ in the correspondence given by Theorem \ref{thm:firstbig}.
To explain this, we first choose $m\in G_{n+1}$ so that $m\cdot\mathcal{E}_{n+1}=\mathcal{E}_{1,i+1}$ where $\mathcal{E}_{1,i+1}$ is the flag in Equation (\ref{eq:Fioneflag}), so that $Q_{1,i+1}=G m B_{n+1}/B_{n+1}$  For $g\in G,$ consider the $B$-orbit $\cZ=Bgm B_{n+1}/B_{n+1}$  in $Q_{1,i+1}$.  By Equation (\ref{eq:leftmonoidSi}) and Remark \ref{r:extended}, there is a left monoid action for each $\alpha \in \Pi_{\fg}$ on $\BQ$ given by 
 \begin{equation}\label{eq:BQleftmonoid}
 \ms*_{L} \cZ\mbox{ is the open } B-\mbox{orbit in } P_{\alpha}gm B_{n+1}/B_{n+1}.
 \end{equation}
 By Equation (\ref{p:old4.7}) and part (3) of Proposition \ref{p:stableandnc}, we also have a right monoid action by the subset of standard simple roots of $\fg_{n+1}$ consisting of roots which are compact imaginary for the $G$-orbit $Q_{1,i+1}$.  We denote the set of compact imaginary roots of the $G$-orbit $Q_{1,i+1}$ by $\Pi_{n+1, cpt}\subset\Pi_{\fg_{n+1}}$, the index $i$ being clear from context.  By Proposition \ref{p:Kmonoidaction}, we have 
 \begin{equation}\label{eq:compactroots}
  \Pi_{n+1, cpt}= \Pi_{\fg_{n+1}} - \{\alpha_{1},\,\alpha_{i},\, \alpha_{i+1} \}. 
 \end{equation}
For $\alpha\in \Pi_{n+1, cpt}$ it follows from Equation (\ref{eq:secondfactormonoid}) and Remark \ref{r:extended} that
 \begin{equation}\label{eq:BQrightmonoid}
 \ms*_{R} \cZ\mbox{ is the open } B-\mbox{orbit in } Bgm P_{\alpha}^{n+1}/B_{n+1},
 \end{equation}
 where $P_{\alpha}^{n+1}\subset G_{n+1}$ is the standard parabolic subgroup of $G_{n+1}$ corresponding to the root $\alpha$. 
 The following result generalizes Theorem 4.11 of \cite{CE21I} from the setting of $B_{n-1}\backslash\B_{n}$ to $S_{i}\backslash\B_{n}$ for arbitrary $i$.
 \begin{thm}\label{thm:monoidcorresp}
 Let $\cZ=Bgm B_{n+1}/ B_{n+1}$ be a $B$-orbit in $Q_{1,i+1}$, and let ${\cZ}^{op}=S_{i} g^{-1} B/B$ be the corresponding $S_{i}$-orbit on $\B_{n}$ as in Theorem \ref{thm:firstbig}.  
If $\alpha \in \Pi_{\fg}$, then 
 \begin{equation}\label{eq:leftinter}
(\ms*_{L}\cZ)^{op}=\ms*_{R}{\cZ}^{op},
 \end{equation}
 where on the left side the monoid action is the left monoid action on $B\backslash Q_{1,i+1}$ given in (\ref{eq:BQleftmonoid}), and on the right side the monoid action is the right monoid action on $S_{i}$-orbits given in (\ref{eq:secondfactormonoid}).  
 For $\alpha\in\Pi_{n+1,cpt}$ with $\alpha=\alpha_{j}$ for $j\neq 1,\, i,\, i+1$ (see (\ref{eq:compactroots})), we have
 \begin{equation}\label{eq:rightinter}
(m(s_{j})*_{R}\cZ)^{op}=m(s_{j-1})*_{L} {\cZ}^{op},
 \end{equation}
 where on the left-hand side, the monoid action is the right monoid action by $\Pi_{n+1,cpt}$ on $B\backslash Q_{1,i+1}$ given in (\ref{eq:BQrightmonoid}), and on the right-hand side, the monoid action on ${\cZ}^{op}$ is the left monoid action by roots of the set $\mathfrak{S}_i$ in (\ref{eq:extraroots}) and defined by Equation (\ref{eq:leftmonoidSi}).
 
 Moreover, the correspondences in (\ref{eq:leftinter}) and (\ref{eq:rightinter}) preserve root types.  That is to say $\alpha$ is complex stable, complex unstable, non-compact, etc. for $\cZ$ if and only if $\alpha$ is complex stable, complex unstable, non-compact, etc. for ${\cZ}^{op}$ in (\ref{eq:leftinter}) and similarly for $\alpha_{j}$ and $\alpha_{j-1}$ in (\ref{eq:rightinter}). 
 \end{thm}
 
 \begin{proof}
First, suppose $\alpha\in\Pi_{\fg}$.  Then it follows from 
(\ref{eq:BQleftmonoid}) that 
$\ms*_{L}\cZ$ is the open $B$-orbit in $\mathcal{Y}_{\cZ}:=P_{\alpha} gm B_{n+1}/ B_{n+1},$ where $m\in G_{n+1}$ satisfies $m\cdot \mathcal{E}_{n+1}=\mathcal{E}_{1,i+1}$, and $\mathcal{E}_{1,i+1}$ is the flag in (\ref{eq:Fioneflag}).    By Equation (\ref{eq:firstcorresp}), this orbit corresponds to the open $S_{i}$-orbit in the variety $\mathcal{Y}_{\cZ}^{op}=S_{i} g^{-1} P_{\alpha}/B$.  But this is exactly $\ms* {\cZ}^{op}$ by Equation (\ref{eq:secondfactormonoid}) and Equation (\ref{eq:leftinter}) follows.  

We now prove Equation (\ref{eq:rightinter}).  As above, we let $\cZ=Bgm B_{n+1}/B_{n+1}$ and let $\alpha\in\Pi_{n+1,cpt}$ so that $\alpha=\alpha_{j}$ with $j\not=1,\, i,\, i+1.$  By Equation (\ref{eq:BQrightmonoid}), $m(s_{j})*\cZ$ is the open $B$-orbit in the variety $\mathcal{X}_{\cZ}:=B gm P^{n+1}_{\alpha_{j}} /B_{n+1}$. But since $\alpha_{j}$ is compact imaginary for $Q_{1,i+1}$, we know that $mP_{\alpha_{j}}^{n+1}/B_{n+1}$ is a single $(G\cap \Ad(m)P^{n+1}_{\alpha_{j}})$-orbit, so that 
\begin{equation}\label{eq:firstiso}
mP^{n+1}_{\alpha_{j}}/B_{n+1}\cong (G\cap \Ad(m)P_{\alpha_{j}}^{n+1})/(G\cap \Ad(m)B_{n+1}).
\end{equation}
Thus, the variety $\mathcal{X}_{\cZ}$ becomes 
 \begin{equation}\label{eq:XZ}
 \mathcal{X}_{\cZ}=BgmP^{n+1}_{\alpha_{j}}\cdot \mathcal{E}_{n+1}=B g(G\cap\Ad(m) P_{\alpha_{j}}^{n+1})\cdot \mathcal{E}_{1,i+1}.
 \end{equation}
 We claim that 
 \begin{equation}\label{eq:groupclaim}
 G\cap \Ad(m)P_{\alpha_{j}}^{n+1}=\mbox{Stab}_{P_{\alpha_{j-1}}}(e_{i}),
 \end{equation}
 where $P_{\alpha_{j-1}}\subset G$ is the standard parabolic subgroup of $G$ corresponding to the root $\alpha_{j-1}$.   Suppose first that $2\leq j\leq i-1$.  Then by using Equation (\ref{eq:Fioneflag}), we see that the parabolic subgroup $\Ad(m)P_{\alpha_{j}}^{n+1}$ of $G_{n+1}$ stabilizes the partial flag in $\C^{n+1}$:
 $$
 \hat{e}_{i}\subset e_{1}\subset e_{2}\subset\dots\subset \{e_{j-1}, e_{j}\}\subset \dots \subset \underbrace{e_{n+1}}_{i+1}\subset \dots \subset e_{n}. 
 $$
 A linear algebra computation similar to the one in the proof of Theorem \ref{thm:firstbig} shows that $\Ad(m)P_{\alpha_{j}}^{n+1}\cap G$ fixes the vector $e_{i}$ and stabilizes the partial flag in $\C^{n}$;
 $$
 e_{1}\subset e_{2}\subset\dots\subset \{e_{j-1},e_{j}\}\subset \dots\subset e_{n}.  
 $$
Thus, $\Ad(m)P_{\alpha_{j}}^{n+1}\cap G\subset \mbox{Stab}_{P_{\alpha_{j-1}}}(e_{i}).$  The reverse inclusion is clear, and the proof in the case when $j\in \{i+2,\dots ,n\}$ is similar.  This establishes the claim in (\ref{eq:groupclaim}).  Since $\mbox{Stab}_{P_{\alpha_{j-1}}}[e_{i}]=Z\cdot \mbox{Stab}_{P_{\alpha_{j-1}}}(e_{i})$ where $Z\subset G$ is the centre, Equation (\ref{eq:XZ}) and Remark \ref{r:corres} yield
$$
\mathcal{X}_{\cZ}^{op}=(G\cap \Ad(m)P_{\alpha_{j}}^{n+1}) g^{-1}\cdot\mathcal{E}_{n}=(\mbox{Stab}_{P_{\alpha_{j-1}}}[e_{i}])g^{-1}\cdot \mathcal{E}_{n}.
$$  
Equation (\ref{eq:rightinter}) now follows from Equation (\ref{eq:leftmonoidSi}) and the observation that for $\alpha_{j}\in \Pi_{n+1, cpt}$, $\alpha_{j-1}\in\mathfrak{S}_{i}$ (see Equations (\ref{eq:extraroots}) and (\ref{eq:compactroots})).

It follows from Proposition \ref{p:stableandnc} and Remarks \ref{r:extended} and \ref{r:onlythree} that $\alpha\in\Pi_{\fg}$ is complex stable for $\cZ$ if and only if the variety $\mathcal{Y}_{\cZ}$ consists of two $B$-orbits with $\cZ$ codimension 1 in the closure of the other orbit.  By Remark \ref{r:corres}, the latter is equivalent to the statement that the variety $\mathcal{Y}_{\cZ}^{op}$ consists of two $S_{i}$-orbits with $\cZ^{op}$ being codimension 1 in the closure of the other orbit.  Again, by Proposition \ref{p:stableandnc} and Remark \ref{r:onlythree}, this is equivalent to the statement that $\alpha$ is complex stable for $\cZ^{op}$.  The other cases are handled similarly.


\end{proof}

\begin{rem}
When $i=n$, 
 the result specializes to Theorem 4.11 of \cite{CE21I}.
\end{rem}

\subsection{Monoid Actions Carrying $S_{i}$-orbits to $S_{i+1}$-orbits}

In the previous sections, we discuss monoid actions by subsets of roots $\Pi_{\fg}\coprod\Pi_{\fg}$ on $G\backslash(\B_{n}\times\B_{n}\times\mathbb{P}^{n-1})$ which move from one $S_{i}$-orbit to another.
  We now discuss a monoid action on $G\backslash(\B_{n}\times\B_{n}\times\mathbb{P}^{n-1})$ that allows us to move from an $S_{i}$-orbits to an $S_{i+1}$-orbit.
By Proposition \ref{p:Kmonoidaction}, the simple root $\alpha_{i+1}$ is complex stable
for the $G$-orbit $Q_{1,i+1}$ and $m(s_{i+1})*Q_{1,i+1}=Q_{1,i+2}$.  
Equation (\ref{p:old4.7}) then implies that 
\begin{equation}\label{eq:monoidmove}
m(s_{i+1})*_{R}(B\backslash Q_{1,i+1})\subset B\backslash Q_{1,i+2}.  
\end{equation}
We will see below that this monoid action corresponds to a monoid action on $G\backslash (\B_{n}\times\B_{n}\times \mathbb{P}^{n-1})$ by a simple root corresponding to the first factor of $\B_{n}$ as in (\ref{eq:firstfactorfirst}).  

 For $\fb^{\prime} \in \B_n,$ consider the $G$-orbit $\mathcal{O}_{\Delta,\fb^{\prime}}=G\cdot (\fb,\fb^{\prime}, [e_{i}])$ given in (\ref{eq:three}), and the corresponding $B$-orbit $\mathcal{O}_{\Delta,\fb^{\prime}}^{op}\in B\backslash \mathfrak{X}$ on $\mathfrak{X}$ given by Remark \ref{rem:orbitlabels}.


\begin{prop}\label{p:orbitchange}
 Let $\fb^{\prime} \in \B_n.$
\begin{enumerate}
\item The $G$-orbit $m(s_{i})*_{1}\mathcal{O}_{\Delta,\fb^{\prime}}=G\cdot (\fb, \tilde{\fb}, [e_{i+1}])$ for some $\tilde{\fb}\in\B_{n}$.  In particular, $m(s_{i})*_{1}\mathcal{O}_{\Delta,\fb^{\prime}}\neq \mathcal{O}_{\Delta,\fb^{\prime}}$.  
\item Further,
 \begin{equation}\label{eq:cplxstableinter}
(m(s_{i})*_{1}\mathcal{O}_{\Delta,\fb^{\prime}})^{op}=m(s_{i+1})*_{R}\mathcal{O}_{\Delta, \fb^{\prime}}^{op}.
\end{equation}
Moreover, the type of $\alpha_{i}$ for $\mathcal{O}_{\Delta,\fb^{\prime}}$ is the same as the type of $\alpha_{i+1}$ for $\mathcal{O}_{\Delta,\fb^{\prime}}^{op}$. 

\end{enumerate}
\end{prop}

\begin{proof}
Let $\alpha = \alpha_i.$  For (1), it follows from Equation (\ref{eq:firstfactorfirst}) that $\ms*_{1}\mathcal{O}_{\Delta,\fb^{\prime}}$ corresponds to the open $B$-orbit in $P_{\alpha}\cdot(\fb^{\prime}, [e_{i}])$.  Note that 
$P_{\alpha}\cdot [e_{i}]=B\cdot [e_{i}]\bigcup B\cdot [e_{i+1}]$, with $B\cdot [e_{i+1}]$ open in $P_{\alpha}\cdot[e_{i}].$   Let $\pi_{\mathbb{P}}:\B_{n}\times \mathbb{P}^{n-1}\to \mathbb{P}^{n-1}$ be the canonical projection.   Then $\pi_{\mathbb{P}}$ is a $B$-equivariant open map.  It follows that $\pi_{\mathbb{P}}(\ms*_{1}\mathcal{O}_{\Delta,\fb^{\prime}})=B\cdot[e_{i+1}]$, and part (1) follows.  

For (2), denote by $\tilde{\cZ}$ the $B$-orbit $ \tilde{\cZ}:=m(s_{i+1})*_{R}\mathcal{O}_{\Delta,\fb^{\prime}}^{op}$ on $\B_{n+1}.$ 
 Let $\tilde{\cZ}^{op}$ be the corresponding $G$-orbit on $\B_{n}\times \B_{n}\times \mathbb{P}^{n-1}$ as in (\ref{eq:secondcorresp}).  We want to show that  $\tilde{\cZ}^{op}=m(s_{i})*_{1} \mathcal{O}_{\Delta,\fb^{\prime}}$ (cf. Remark \ref{rem:oppusage}).  We first claim that  
\begin{equation}\label{eq:dimby1}
\dim \tilde{\cZ}^{op}=\dim\mathcal{O}_{\Delta,\fb^{\prime}}+1=\dim\left( m(s_{i})*_{1} \mathcal{O}_{\Delta,\fb^{\prime}} \right).
\end{equation}
We show that for \emph{any} $G$-orbit $\mathcal{O}_{\Delta,\hat{\fb}}=G\cdot(\fb,\hat{\fb}, [e_{j}])$, we have 
\begin{equation}\label{eq:relatedim}
\dim \mathcal{O}_{\Delta,\hat{\fb}}=\dim \mathcal{O}_{\Delta,\hat{\fb}}^{op}+\dim\B_{n}-1.
\end{equation}
 By Equation (\ref{e:dimcorr}), $\dim \mathcal{O}_{\Delta,\hat{\fb}}=\dim \B_{n}+j-1+\dim Q_{\hat{\fb}}.$  By Equation (\ref{eq:opdim}), $\dim Q_{\hat{\fb}} = \dim Q_{\hat{\fb}}^{op}-j,$ which implies Equation (\ref{eq:relatedim}) since $\mathcal{O}_{\Delta,\hat{\fb}}^{op}=Q_{\hat{\fb}}^{op}$ (see Remark \ref{rem:oppusage}).
  Further, $\mathcal{O}_{\Delta,{\fb}^{\prime}}^{op} \in B\backslash Q_{1,i+1}$ by Equation (\ref{eq:firstcorresp}), so $\tilde{\cZ} \in  B\backslash Q_{1,i+2}$ by (\ref{eq:monoidmove}).  Thus, $\dim(\tilde{\cZ})=\dim\mathcal{O}_{\Delta,{\fb}^{\prime}}^{op}+1$ by (\ref{eq:basicmonoid}).
Now applying (\ref{eq:relatedim}) to the orbits $\tilde{\cZ}$ and $\mathcal{O}_{\Delta,\fb^{\prime}}$, we obtain
$$
\dim \tilde{\cZ}^{op}=\dim\tilde{\cZ}+\dim\B_{n}-1=\dim\mathcal{O}_{\Delta,\fb^{\prime}}^{op}+1+\dim\B_{n}-1=\dim\mathcal{O}_{\Delta,\fb^{\prime}}+1, 
$$
yielding the first equality of (\ref{eq:dimby1}).  The second equality of (\ref{eq:dimby1}) follows from Part (1) and (\ref{eq:basicmonoid}).  Hence, by (\ref{eq:firstfactorfirst}) and (\ref{eq:dimby1}), to prove (2) it suffices to show that 
\begin{equation}\label{eq:thebigcheese}
\tilde{\cZ}^{op}\subset G\cdot (P_{\alpha_{i}}\cdot\fb, \fb^{\prime},[e_{i}]).
\end{equation}
Choose $g\in G$ so that $\Ad(g^{-1})\fb=\fb^{\prime}$.
Then it follows from Remark \ref{r:corres} that $\mathcal{O}_{\Delta,\fb^{\prime}}^{op}=B\cdot \Ad(gm)\fb_{n+1}$, where $m\in G_{n+1}$ is such that $\Ad(m)\fb_{n+1}=\fb_{1,i+1}$ with $\fb_{1,i+1}$ the Borel subalgebra stabilizing the flag $\mathcal{E}_{1,i+1}$ in (\ref{eq:Fioneflag}).  We prove (\ref{eq:thebigcheese}) on a case-by-case basis.  By (\ref{eq:monoidmove}), $\tilde{\cZ}=m(s_{i+1})*\mathcal{O}_{\Delta,\fb^{\prime}}^{op}\neq \mathcal{O}_{\Delta,\fb^{\prime}}^{op}$.   Proposition \ref{p:stableandnc} then implies that the root $\alpha_{i+1}$ is either complex stable or non-compact for $\mathcal{O}_{\Delta,\fb^{\prime}}^{op}$.  Suppose first that $\alpha_{i+1}$ is complex stable for the $B$-orbit $\mathcal{O}_{\Delta,\fb^{\prime}}^{op}$.  
By part (1) of Proposition \ref{p:stableandnc}, $\tilde{\cZ}=B\cdot \Ad(gm) s_{i+1}(\fb_{n+1})$.  To compute
$\tilde{\cZ}^{op}$, consider the Borel subalgebra $\tilde{\fb}:=\Ad(m) s_{i+1}(\fb_{n+1})$.  Using Equation (\ref{eq:Fioneflag}), we note that $\tilde{\fb}$ stabilizes the flag in $\C^{n+1}$ given by:
$$
\mathcal{G}:=(\hat{e}_{i}\subset e_{1}\subset\dots \subset e_{i-1}\subset e_{i+1}\subset\underbrace{e_{n+1}}_{i+2}\subset e_{i+2} \subset \dots\subset e_{n}).
$$
Let $\dot{s}_{i}$ be the representative of the simple reflection $s_{i}$ given by the permutation matrix corresponding to the transposition $(i, i+1),$ and note that
 $$
\dot{s}_{i}(\mathcal{G})=(\hat{e}_{i+1}\subset e_{1}\subset\dots \subset e_{i-1}\subset e_{i}\subset\underbrace{e_{n+1}}_{i+2}\subset\dots\subset e_{n})=\mathcal{E}_{1,i+2}.
 $$
 It follows that  $\tilde{\cZ}=B\cdot \Ad(g\dot{s}_{i})\fb_{1,i+2}$.  By Equation (\ref{eq:firstcorresp}) and the fact that $\fb^{\prime}=\Ad(g^{-1})\fb$ we can then compute: 
 \begin{equation}\label{eq:cplxcase}
 \begin{split}
 \tilde{\cZ}^{op}&=G\cdot(\fb, \Ad(\dot{s}_{i}g^{-1})\fb, [e_{i+1}])\\
 &=G\cdot (s_{i}(\fb), \fb^{\prime}, s_{i}\cdot[e_{i+1}])\\
 &=G\cdot (s_{i}(\fb), \fb^{\prime}, [e_{i}])\subset G\cdot (P_{\alpha_{i}}\cdot\fb, \fb^{\prime}, [e_{i}]).
 \end{split}
 \end{equation}
We thus obtain (\ref{eq:thebigcheese}) in this case.
  
 Now suppose that $\alpha_{i+1}$ is non-compact for the $B$-orbit $\mathcal{O}_{\Delta,\fb^{\prime}}^{op}$.  Then by 
 Part (2) of Proposition \ref{p:stableandnc}, the orbit 
 $\tilde{\cZ}=B\cdot \Ad(gmu_{\alpha_{i+1}})\fb_{n+1}$, where $u_{\alpha_{i+1}}\in G_{n+1}$ is the Cayley transform with respect to the root $\alpha_{i+1}$.  By Equation (\ref{eq:Fioneflag}) and Remark \ref{r:Cayley} the Borel subalgebra $\Ad(mu_{\alpha_{i+1}})\fb_{n+1}$ stabilizes the flag 
 $$
 \mathcal{H}:=(\hat{e}_{i}\subset e_{1}\subset \dots\subset e_{i-1}\subset e_{i+1}-e_{i}\subset \underbrace{e_{n+1}}_{i+2}\subset e_{i+2} \subset \dots\subset e_{n}).
 $$
Consider the element $b\in B$ whose action on the standard basis is given by: $b:e_{i+1}\mapsto e_{i+1}+e_{i}$ and $b:e_{\ell}\mapsto e_{\ell}$ for all other $\ell$. 
Computation shows that $\dot{s}_{i}b\cdot \mathcal{H}=\mathcal{E}_{1,i+2}$.  Thus, 
$\tilde{\cZ}=B\cdot \Ad(gb^{-1}\dot{s}_{i})\fb_{1,i+2}$.  Using Equation (\ref{eq:firstcorresp}) and $\fb^{\prime}=\Ad(g^{-1})\fb,$ we obtain 
$$
\tilde{\cZ}^{op}=G\cdot(\fb, \Ad(\dot{s}_{i}^{-1}b)\fb^{\prime}, [e_{i+1}])=G\cdot(b^{-1}\cdot s_{i}(\fb), \fb^{\prime}, b^{-1}s_{i}\cdot [e_{i+1}]).
$$
Computation shows that $b^{-1} s_{i}\cdot[e_{i+1}]=b^{-1}[e_{i}]=[e_{i}]$, and this proves (\ref{eq:thebigcheese}) in this case as well.  Thus, Equation (\ref{eq:cplxstableinter}) holds.

To prove the statement about root types, it suffices to show that 
$\alpha_{i}$ is complex stable for $\mathcal{O}_{\Delta, \fb^{\prime}}$ if and only if $\alpha_{i+1}$ is complex stable for $\mathcal{O}_{\Delta, \fb^{\prime}}^{op}$.  The computations in the complex stable case above demonstrate that if $\alpha_{i+1}$ is complex stable for $\mathcal{O}_{\Delta, \fb^{\prime}}^{op}$, then $\tilde{\cZ}^{op}=m(s_{i})*_{1}\mathcal{O}_{\Delta,\fb^{\prime}}=G\cdot (s_{i}(\fb), \fb^{\prime},[e_{i}])$ (see (\ref{eq:cplxcase})).  It follows from Proposition \ref{p:stableandnc} that $\alpha_{i}$ is complex stable for the $G$-orbit $\mathcal{O}_{\Delta,\fb^{\prime}}$.  The converse works similarly. 

\end{proof}



Recall the definition of the weak order in Definition \ref{d:weak}.  If $Q_1$ and $Q_2$ are two $S_i$-orbits on $\B_{n},$  we say $Q_2$ dominates $Q_1$ if $Q_1$ is less than or equal to $Q_2$ in the weak order.  Using Theorem \ref{thm:monoidcorresp}, we can prove the following result.
\begin{thm}\label{thm:Siweak}
Every $S_{i}$-orbit on $\B_{n}$ dominates a zero dimensional $S_{i}$-orbit in the weak order defined by the monoid action on $S_{i}\backslash\B_{n}$ by the simple roots $\mathfrak{S}_i\coprod \Pi_{\fg}$, where $\mathfrak{S}_i$ is the subset of simple roots given in (\ref{eq:extraroots}).  
\end{thm}
\begin{proof}
Let ${\cZ}^{op}$ be the $S_{i}$-orbit corresponding to $\cZ\in B\backslash Q_{1,i+1}$ via Theorem \ref{thm:firstbig}.  On p. 289 of the proof of Theorem 6.5 in \cite{CE21II}, we establish the existence of a $B$-orbit ${\cZ}^{\prime}$ that is closed in $Q_{1,i+1}$ and a sequence of simple roots $\alpha_{j_1}, \dots, \alpha_{j_k} \in \Pi_{\fg}\coprod \Pi_{\fg_{n+1}}$ such that $m(\vec{s})*{\cZ}^{\prime}={\cZ}.$  By analyzing the above argument and using Equations (6.15), (6.18), and (6.19) from \cite{CE21II} and Equation (4.12) from \cite{CE21I}, we see that the roots from $\Pi_{\fg_{n+1}}$ are in $\{\alpha_2, \dots, \alpha_{i-1} \}$, which is a subset of the set $\Pi_{n+1,cpt}$ defined in Equation (\ref{eq:compactroots}).  
Hence, by Theorem \ref{thm:monoidcorresp}, there exists a sequence of roots $\widetilde{\alpha_{i_{1}}}, \dots, \widetilde{\alpha_{i_{k}}}\in \mathfrak{S}_i\coprod\Pi_{\fg}$ such that $m(\vec{\widetilde{s}})*{{\cZ}^{\prime}}^{op}={\cZ}^{op}$ (see (\ref{eq:leftinter}) and (\ref{eq:rightinter})).  
To complete the proof, it remains to observe that since ${\cZ}^{\prime}$ is closed in $Q_{1,i+1}$, it follows that ${{\cZ}^{\prime}}^{op}$ is a closed $S_i$-orbit in $\B_n$ by Remark \ref{r:corres}.  However, a closed orbit of a connected solvable group $S$ on a projective variety is zero dimensional, since the stabilizer of a point in the orbit is a parabolic subgroup of $S$, and hence must be all of $S.$   Since $S_i$ is connected and solvable, we conclude that ${{\cZ}^{\prime}}^{op}$ is zero dimensional.
\end{proof}

\begin{cor}\label{c:closedSi}
There are exactly $i$ closed $S_{i}$-orbits on $\B_{n}$ all of which are zero dimensional.  
They are the orbits through the $i$-standard flags
\begin{equation}\label{eq:closedflags}
\mathcal{G}_{k}:=(e_{1}\subset \dots \subset e_{k-1}\subset\underbrace{e_{i}}_{k}\subset e_{k}\subset \dots\subset e_{i-2}\subset e_{i-1}\subset e_{i+1}\subset\dots \subset e_{n})
\end{equation}
for $k=1,\dots, i$.  
\end{cor}

\begin{proof}
By the assertion at the end of the proof of Theorem \ref{thm:Siweak}, each closed $S_i$-orbit is a point fixed by the diagonal torus $H_{n-1}$ of $S_i.$   Since $H_{n-1}$ acts on $\C^n$ with $n$ distinct weights, it follows that each $H_{n-1}$-fixed point
is given by applying a permutation to the standard flag.  An easy argument now shows that the $H_{n-1}$-fixed points that are fixed by $S_i$ are exactly the flags $\mathcal{G}_k$ for $k=1, \dots, i.$
\end{proof}

The standard ordering of Richardson-Springer on $S_i\backslash \B_n$ is characterized as the weakest partial order $\le$ on $S_i\backslash \B_n$ such that if $Q_1, Q_2 \in S_i\backslash \B_n$ and $\alpha \in \mathfrak{S}_i\coprod\Pi_{\fg},$ then (i) $Q_1 \le \ms * Q_1$, (ii) $Q_1 \le Q_2$ implies that $\ms * Q_1 \le \ms * Q_2$, and (iii) $Q_1 \le Q_2$ and $\dim(Q_2) \le \dim(Q_1)$ imply that $Q_1=Q_2$ (see Proposition 5.6 of \cite{RS}).  It can be computed in terms of the combinatorics of the monoid action and the closed orbits.

\begin{cor}\label{c:standardorder}
The closure ordering on $S_{i}\backslash\B_{n}$ is given by the standard ordering 
of Richardson-Springer with respect to the monoid action on $S_{i}\backslash\B_{n}$ via the simple roots $\mathfrak{S}_i\coprod\Pi_{\fg}$, where $\mathfrak{S}_i$ is the subset of roots given in Equation (\ref{eq:extraroots}). 
\end{cor}
\begin{proof}
By Theorem \ref{thm:Siweak}, the minimal $S_{i}$-orbits in the weak order are all zero dimensional.  We can therefore apply Theorem 6.2 of \cite{CE21II} to conclude that the closure order agrees with the standard order. 
\end{proof}

\begin{rem}
In \cite{CE21II}, we prove the corresponding result for $B_{n-1}$-orbits on $\B_{n}$, which is one of the main results of that paper (see Corollary 6.10 of \emph{loc. cit.}).  Corollary \ref{c:standardorder} generalizes that result to the setting of $S_{i}$-orbits for \emph{any} $i\in\{1,\dots, n\}$. 
\end{rem}


\section{Shareshian Data and the Correspondence}\label{s:Shareshian}

 
 
 In \cite{Shpairs}, we associate to each $B$-orbit on $\B_{n+1}$ a pair 
 of permutations in $W_{n+1}\times W_{n+1}$ as follows.  
 Let $B_{n+1}\subset G_{n+1}$ be the standard Borel subgroup of invertible $(n+1)\times (n+1)$ upper triangular matrices.  Let $B_{n+1}^{*}$ be the Borel subgroup stabilizing the flag 
 \begin{equation}\label{eq:E*flag}
 \mathcal{E}_{n+1}^{*}:=(e_{n+1}\subset e_{1}\subset \dots\subset e_{n-1}\subset e_{n}). 
 \end{equation}
 In Theorem 1.1 of \cite{Shpairs}, we show that the $B$-orbits on $\B_{n+1}$ are precisely the non-empty intersections of $B_{n+1}$-orbits and $B_{n+1}^{*}$-orbits on $\B_{n+1}.$  
  Thus, using the Bruhat decomposition, we see that for a $B$-orbit $Q$ on $\B_{n+1}$, 
 $Q=B_{n+1}\cdot w(\mathcal{E}_{n+1}) \cap B_{n+1}^{*}\cdot u^{*}(\mathcal{E}_{n+1}^{*})$ for some $w,\, u^{*}\in W_{n+1},$ the symmetric group on $n+1$ letters.
 We define the \emph{Shareshian pair} of the orbit $Q$ to be $\Sh(Q):=(w,u^{*})$.  
 \begin{nota}
 We let $\Sp\subset W_{n+1}\times W_{n+1}$ denote the set of all Shareshian pairs, i.e.,
 \begin{equation}\label{eq:Spdefn}
 \Sp:=\{(w,u^{*})\in W_{n+1}\times W_{n+1}|\; \exists\, Q\in B\backslash\B_{n+1}\mbox{ with } (w, u^{*})=\Sh(Q)\}. 
 \end{equation}
 
 \end{nota}
 
 The Shareshian pair of a $B$-orbit $Q$ can be computed as follows.  In Section 4.1 of \cite{CE21II}, we find a canonical set of representatives for 
elements of $\Borbitspace$ which we call \emph{flags in standard form} by using the notion of hat vector from Equation (\ref{e:hatvectordef}).  

\begin{dfn}\label{d:std}
We say that a flag 
\begin{equation}\label{eq:basicflag}
\mathcal{F}:=(v_{1}\subset \dots \subset v_{i}\subset\dots\subset v_{n} \subset v_{n+1})
\end{equation}
in the flag variety $\B_{n+1}$ for $G_{n+1}$
is in \emph{standard form} if $v_{i}=e_{j_{i}}$ or $v_{i}=\he_{j_{i}}$ for all $i=1,\dots, n+1$, and 
$\mathcal{F}$ satisfies the following three conditions:
\begin{enumerate}
\item $v_i = e_{n+1}$ for some $i$.
\item If $v_{i}=e_{n+1}$, then $v_{k}=e_{j_{k}}$ for all $k>i$.
 \item If $i<k$ with $v_{i}=\he_{j_{i}}$ and $v_{k}=\he_{j_{k}}$, then $j_{i}>j_{k}$. 
\end{enumerate}
\end{dfn}
Theorem 4.7 of \cite{CE21II} asserts that every $B-$orbit in $\B_{n+1}$ has a unique representative in standard form.
 
The Shareshian pair of a $B$-orbit $Q$ can easily be computed in terms of the unique flag $\F$ in standard form contained in the orbit $Q$.  We recall the following notation and results from Section 2 of \cite{Shpairs}. 
\begin{nota}\label{n:tildeandstarflags}
Let $\F\in\B_{n+1}$ be a flag in standard form.  We denote by $\tilde{\F}$ the unique $H_{n+1}$-stable flag in the $B_{n+1}$-orbit $B_{n+1}\cdot \F$ and
by $\F^{*}$ the unique $H_{n+1}$-stable flag in the $B_{n+1}^{*}$-orbit $B_{n+1}^{*}\cdot \F$.
\end{nota}

\begin{rem}\label{r:tildeandstarflags}
Let $Q=B\cdot \F\in B\backslash \B_{n+1}$ with $\F$ a flag in standard form and let $\Sh(Q)=(w, u^{*})$.  Then it follows from definitions that $\Sh(Q)=(w,u^{*})$ if and only if $\tilde{\F}=w(\mathcal{E}_{n+1})$ and $\F^{*}=u^{*}(\mathcal{E}^{*}_{n+1})$, where $\mathcal{E}^{*}_{n+1}$ is the flag in Equation (\ref{eq:E*flag}).
\end{rem}

The following proposition explains how to compute the flags 
$\tilde{\F}$ and $\F^{*}$ of Notation \ref{n:tildeandstarflags} from a flag $\F$ in standard form.   This will be used to compute $\Sh(Q).$

 \begin{prop}\label{p:orbits}(see Proposition 2.7 of \cite{Shpairs})
Let $\F\subset \B_{n+1}$ be a flag in standard form with 
$$
\F=(v_{1}\subset v_{2}\subset \dots\subset v_{p}\subset \dots\subset v_{n+1}).
$$  
(1) If $\F$ contains no hat vectors, then 
$\F=\tilde{\F}=\F^{*}$. 

\noindent (2) If $\F$ has hat vectors, we may assume that $\F$ has the form:
\begin{equation}\label{eq:hatvectorF}
\F=(v_{1}\subset\dots\subset v_{i_{k}-1}\subset\underbrace{ \he_{j_{k}}}_{i_{k}}\subset v_{i_{k}+1}\subset \dots \subset \underbrace{\he_{j_{k-1}}}_{i_{k-1}}\subset \dots \subset \underbrace{\he_{j_{1}}}_{i_{1}}\subset \dots\subset \underbrace{e_{n+1}}_{p}\subset v_{p+1}\subset \dots\subset v_{n+1}),
\end{equation}
with $j_{k}>j_{k-1}>\dots>j_{1}$ and $v_{m}$ a standard basis vector.  
Then
\begin{equation}\label{eq:tildeF}
\tilde{\F}=(v_{1}\subset\dots\subset v_{i_{k}-1}\subset\underbrace{e_{n+1}}_{i_{k}}\subset v_{i_{k}+1}\subset \dots\subset \underbrace{e_{j_{k}}}_{i_{k-1}}\subset\dots\subset \underbrace{e_{j_{2}}}_{i_{1}}\subset \dots\subset\underbrace{e_{j_{1}}}_{p}\subset v_{p+1}\subset\dots\subset v_{n+1}),
\end{equation}
and 
\begin{equation}\label{eq:starF}
\F^{*}=(v_{1}\subset \dots\subset v_{i_{k-1}}\subset \underbrace{e_{j_{k}}}_{i_{k}}\subset v_{i_{k}+1}\subset \dots\subset \underbrace{e_{j_{k-1}}}_{i_{k-1}}\subset \dots\subset \underbrace{e_{j_{1}}}_{i_{1}}\subset \dots\subset\underbrace{e_{n+1}}_{p}\subset v_{p+1}\subset \dots\subset v_{n+1}),
\end{equation}
where the $v_{m}$ are the same vectors that appear in the flag in Equation (\ref{eq:hatvectorF}).
\end{prop}

 For Shareshian pairs $\Sh(Q^{\prime})=(x, y^{*})$ and $\Sh(Q)=(w,u^{*})$ we say:
 \begin{equation}\label{eq:oldShorder}
 \Sh(Q^{\prime})\leq \Sh(Q)\Leftrightarrow x\leq w \mbox{ and } y^{*}\leq_{*} u^{*},
 \end{equation}
 where ``$\leq_{*}$" denotes the Bruhat order on $W_{n+1}$ defined by the Borel subalgebra $\fb_{n+1}^{*}$.  We refer to the ordering in (\ref{eq:oldShorder}) as the Bruhat ordering on Shareshian pairs.  
 Theorem 3.4 of \cite{Shpairs} asserts that the closure ordering on $B\backslash \B_{n+1}$ is given by the Bruhat ordering on the set of Shareshian pairs, i.e., 
 \begin{equation}\label{eq:oldShclosure}
 Q^{\prime}\subset\overline{Q}\Leftrightarrow \Sh(Q^{\prime})\leq \Sh(Q).
 \end{equation}
 

 In our most recent work in \cite{Siorbits}, we extend the theory 
 of Shareshian pairs to $S_{i}$-orbits on $\B_{n}$ for any $i\in \{1,\dots, n\}$.  The role of the Borel subgroup $B_{n+1}^{*}$ is replaced by the Borel subgroup $B^{i}$ of $G$ which stabilizes the flag 
  \begin{equation}\label{eq:Eiflag}
 \mathcal{E}^{i}=(e_{i}\subset e_{1}\subset\dots\subset e_{i-1}\subset e_{i+1}\subset\dots\subset e_{n}) = \sigma_{i}(\mathcal{E}_{n}),
 \end{equation}
where 
\begin{equation}\label{e:sigmai}
\sigma_{i} = (i,\,i-1,\dots, 2, \,1)
\end{equation}
is the indicated $i$-cycle.   
 In \cite{Siorbits}, we show that the $S_{i}$-orbits on $\B_{n}$ are precisely the non-empty intersections of $B$ and $B^{i}$-orbits on $\B_{n}$ and define the $i$-\emph{Shareshian pair} of $Q\in S_{i}\backslash\B_{n}$ to be:
 \begin{equation}\label{eq:iShdefn}
 \Sh_{i}(Q):=(w, u^{i})\in \Wn\times \Wn \Leftrightarrow Q=B\cdot w(\mathcal{E}_{n})\cap B^{i}\cdot u^{i}(\mathcal{E}^{i}).  
 \end{equation}
 
  It will at times be convenient to consider a slightly different version of an $i$-Shareshian pair. 
 \begin{dfn}\label{d:istd}
 Let $Q\in S_{i}\backslash\B_{n}$ with $\Sh_{i}(Q)=(w, u^{i})$.  The \emph{standardized} $i$-Shareshian pair of $Q$ is $\widetilde{\Sh}_{i}(Q)=(w, u)$ where $u:=\sigma_{i}^{-1}u^{i}\sigma_{i}$. 
 \end{dfn}
 We can define a Bruhat ordering on $i$-Shareshian pairs analogously to how we defined the one Shareshian pairs given in (\ref{eq:oldShorder}).  On the set of standardized Shareshian pairs this ordering is just the restriction of the product of standard Bruhat orders on $\Wn\times \Wn$, i.e., for $\widetilde{\Sh}_{i}(Q^{\prime})=(x,y)$ and $\widetilde{\Sh}_{i}(Q)=(w,u)$,
 \begin{equation}\label{eq:istdShorder}
 \widetilde{\Sh}_{i}(Q^{\prime})\leq \widetilde{\Sh}_{i}(Q)\Leftrightarrow x\leq w \mbox{ and } y\leq u. 
 \end{equation}
 
 \subsection{Monoid Actions on Shareshian and $i$-Shareshian Pairs}
 \label{ss:monoidSh}
 The theory of Section \ref{ss:monoidbackground} can be applied to the study of $G$-orbits 
 on $\B_{n}\times\B_{n}$ for the diagonal action.  These orbits are well-known to  correspond to $B$-orbits on $\B_{n}$, which we parameterize by elements of $W_{n}$ using the Bruhat decomposition.  We therefore obtain a monoid action via the standard simple roots $\Pi_{\fg}\coprod\Pi_{\fg}$ of $\fg\oplus\fg$ on $\Wn.$  The first factor of $\Pi_{\fg}$ acts on the left of $\Wn$, and the second factor acts on the right of $\Wn$.  In terms of Definition \ref{d:roottype}, a root $\alpha\in\Pi_{\fg}\coprod \Pi_{\fg}$ is always complex for $w\in \Wn$ (see Section 4.3 of \cite{Shpairs}).  Explicitly, for $\alpha \in \Pi_{\fg}$ and $w\in \Wn$,
 \begin{equation}\label{eq:Wmonoidact}
\begin{split}
\mbox{(Left Action) } &m(s_{\alpha})*_{L} w=w \mbox{ if } \ell(s_{\alpha}w)<\ell(w) \mbox{ and } m(s_{\alpha})*_{L}w=s_{\alpha}w\mbox{ if } \ell(s_{\alpha}w)>\ell(w).\\
\mbox{(Right Action) } &m(s_{\alpha})*_{R}w=w \mbox{ if } \ell(ws_{\alpha})<\ell(w) \mbox{ and } m(s_{\alpha})*_{R}w=ws_{\alpha}\mbox{ if } \ell(ws_{\alpha})>\ell(w).\\
\end{split}
\end{equation}

\begin{rem}\label{r:Demazure}
 Note that the left and right monoid actions in (\ref{eq:Wmonoidact}) are just given by the corresponding Demazure products from Equation (\ref{eq:introrightDem}).  In this case, the weak order of Definition \ref{d:weak} is the two-sided weak order on $W_{n}$ (see for example, Exercise 8, Chapter 3 of \cite{BB}). 
\end{rem}

Recall that the Borel subalgebra $\fb^{i}$ stabilizes the flag $\mathcal{E}^{i}$ of Equation (\ref{eq:Eiflag}) and $\fb^{i}=\sigma_{i}(\fb),$ where $\sigma_{i}=(i,i-1,\dots, 2, 1)$ is the $i$-cycle so that $\mathcal{E}^{i}=\sigma_{i}(\mathcal{E}_{n})$.   Let $\Pi^{i}_{\fg}=\{\alpha_{1}^{i},\dots, \alpha_{n-1}^{i}\}$ be the simple roots defined by the Borel subalgebra $\fb^{i}$ with corresponding set of simple reflections $S^{i}:=\{s^{i}_{1},\dots, s^{i}_{n}\}$.  It follows that $\alpha^{i}_{k}=\sigma_{i}(\alpha_{k})$ for $k\in \{1,\dots, n-1\}$.
\begin{nota}\label{n:Worders}
 By using the set  $S^{i}$ as our simple reflections in place of  $S,$ we obtain a new order relation on $W_{n}.$  We denote $\Wn$ with this non-standard order relation by $(\Wn, S^{i})$, and denote $\Wn$ with the standard order relation by $(\Wn, S).$  We denote elements in the poset $(\Wn,S)$ by $w, u, ...$ and denote elements in the poset $(\Wn, S^{i})$ by
 $w^{i}, u^{i}, ...$
\end{nota}

We also obtain left and right monoid actions of $S^{i}$ on $(\Wn, S^{i})$  analogous to the ones in Equation (\ref{eq:Wmonoidact}) with $\alpha\in\Pi_{\fg}$ replaced by $\alpha^{i}\in \Pi_{\fg}^{i}$. Recall $\mathfrak{S}_i$ from (\ref{eq:extraroots}).
 We now use the monoid actions of $S$ and $S^{i}$ on $(\Wn,S)$ and $(\Wn,S^{i})$ to define a monoid action of $\mathfrak{S}_i\coprod\Pi_{\fg}$ on $(\Wn,S)\times (\Wn,S^{i})$.  First, observe that $\mathfrak{S}_i\subset\Pi_{\fg}^{i}$.  Indeed, $\alpha_{k}=\alpha_{k+1}^{i}$ for $k=1,\dots, i-2$ and $\alpha_{k}=\alpha_{k}^{i}$ for $k=i+1,\dots, n-1$.   

\begin{dfn}\label{d:diagonal}\emph{(see Definition 5.6 of \cite{Siorbits})}
Define the \emph{restricted diagonal monoid action} on $(W,S)\times (W,S^{i})$ via simple roots $\mathfrak{S}_i\coprod \Pi_{\fg}^{i}$ as follows. 

  
For $(w,u^{i})\in (W,S)\times (W,S^{i})$, define
\begin{equation}\label{eq:diagonalmonoid}
\begin{split}
\mbox{ (Left action)} &\mbox{ For } \alpha\in\mathfrak{S}_i,\, m(s_{\alpha})*_{L}(w,u^{i})=(m(s_{\alpha})*_{L}w, m(s_{\alpha})*_{L}u^{i}).\\
\mbox{ (Right action)} &\mbox{ For } \alpha\in\Pi_{\fg}, m(s_{\alpha})*_{R}(w,u^{i})=(m(s_{\alpha})*_{R}w, m(s_{\alpha^{i}})*_{R}u^{i}),
\end{split}
\end{equation}
where the monoid actions $*_{L}$ and $*_{R}$ are given in Equation (\ref{eq:Wmonoidact}).
\end{dfn}
Henceforth, we will frequently drop the subscripts $L$ and $R$ on the monoid actions defined in Equation (\ref{eq:diagonalmonoid}) and use the convention that a simple root $\alpha\in\mathfrak{S}_i\coprod \Pi_{\fg}$ acts on the left whenever $\alpha\in \mathfrak{S}_i$ and on the right if $\alpha\in\Pi_{\fg}$. 
 Recall the extended monoid action on $S_{i}\backslash\B_{n}$ defined in Section \ref{ss:Simonoid}.  


\begin{thm}\label{thm:intertwine}\emph{[Theorem 5.7 of \cite{Siorbits}]}
 For $Q\in S_{i}\backslash\B_{n}$ and $\alpha\in\mathfrak{S}_i\coprod \Pi_{\fg}$,
 \begin{equation}\label{eq:Shinter}
 \Sh_{i}(m(s_{\alpha})*Q)=m(s_{\alpha})*\Sh_{i}(Q).  
 \end{equation}
Moreover, if $m(s_{\alpha})*Q\neq Q$, then the type of the root $\alpha$ is determined by the type of $\alpha$ for the corresponding $i$-Shareshian pair $\Sh_{i}(Q)=(w,u^{i})$.  More precisely, for a root $\alpha\in\mathfrak{S}_i\coprod \Pi_{\fg}$, 
\begin{enumerate}
\item The root $\alpha$ is complex stable for $Q$ if and only if it is complex stable for both $w$ and $u^{i}$.
\item The root $\alpha$ is non-compact for $Q$ if and only if $\alpha$ is complex stable for exactly one of $w$ and $u^{i}$ and unstable for the other. 
\item The root $\alpha$ is real or complex unstable for $Q$ if and only if $\alpha$ is complex unstable for both $w$ and $u^{i}$.
\end{enumerate}
\end{thm}
\begin{rem}\label{r:Theorem4.7}
When $i=n$ the group $S_{n}=B_{n-1}$ (up to centre) and the extended monoid action of Section \ref{ss:Simonoid} is the extended monoid action on $B_{n-1}\backslash\B_{n}$ constructed in Section 4 of \cite{CE21I} (see Remark \ref{r:extended}).  Theorem \ref{thm:intertwine} then specializes to Theorem 4.8 of \cite{Shpairs} with $\Sh_{i}(Q)=(w, u^{i})$ replaced with $\Sh(Q)=(w, u^{*})$ for $Q\in B_{n-1}\backslash\B_{n}$.  In this case we denote the simple roots $\Pi^{n}_{\fg}$ by $\Pi^{*}_{\fg}=\{\alpha_{1}^{*},\dots, \alpha_{n-1}^{*}\}$ and the corresponding simple reflections by $S^{*}=\{s_{1}^{*},\dots, s_{n-1}^{*}\}$. 
\end{rem}
 \begin{rem}\label{r:standardmonoid}
We define the monoid action on standardized $i$-Shareshian pairs.  For $\widetilde{\Sh}_{i}(Q)=(w,u)$, we define the right action by $\alpha\in \Pi_{\fg}$  by 
$$
\ms*(w,u)=(\ms*w, \ms*u).  
$$
For $\alpha\in \mathfrak{S}_{i},$ we define the left monoid action by:  
\begin{equation}\label{eq:twistinsecond}
 \ms*(w,u)=(\ms*w, m(s_{\sigma_{i}^{-1}(\alpha)})* u).
\end{equation}
 Note that for $\alpha\in\mathfrak{S}_{i}$, it follows from (\ref{e:sigmai}) that $\sigma_{i}^{-1}(\alpha)\in \Pi_{\fg}$ so that the operator $m(s_{\sigma_{i}^{-1}}(\alpha))$ is defined.
With this convention, for $Q\in S_{i}\backslash\B_{n}$ and $\alpha\in\mathfrak{S}_i\coprod \Pi_{\fg}$, the identity
\begin{equation}\label{eq:standardshar}
 \widetilde{\Sh}_{i}(m(s_{\alpha})*Q)=m(s_{\alpha})*\widetilde{\Sh}_{i}(Q) 
\end{equation}
is an easy consequence of (\ref{eq:Shinter}), and the analogues of (1)-(3) in Theorem \ref{thm:intertwine} are valid for standardized Shareshian pairs.
\end{rem}


 \subsection{Shareshian Pairs and $G$-orbits on $\B_{n+1}$}
 Given any $B$-orbit $\cZ$ on $\B_{n+1}$, $\cZ$ is contained in the $G$-orbit $Q_{G}:=G\cdot \cZ$ on $\B_{n+1}$.  Recall the classification of 
 $G$-orbits on $\B_{n+1}$ given in Section \ref{s:Gorbits}.  Recall the flag 
 $\mathcal{E}^{*}_{n+1}=(e_{n+1}\subset e_{1}\subset\dots\subset e_{n})$ defined in Equation (\ref{eq:E*flag}) and the Borel subalgebra of $\fb_{n+1}^{*}\subset\fg_{n+1}$ which stabilizes it.  Let $S^{*}=\{s_{1}^{*}, \dots, s_{n}^{*}\}$ be the set of simple reflections defined by $\fb^{*}_{n+1}$.  Then 
 \begin{equation}\label{eq:starreflections}
 s_{1}^{*}=(1,\, n+1)\mbox{ and } s_{j}^{*}=(j-1,\, j)\mbox{ for } j\in\{2,\dots, n\}.  
\end{equation}
The Weyl group $\Wn =\langle s_{1},\dots, s_{n-1}\rangle=\langle s_{2}^{*}, \dots, s_{n}^{*}\rangle.$  
For $i,\, j\in\{1,\dots, n+1\}$, define elements $w_{i}$ and $w_{j}^{*}$ in $W_{n+1}$ by 
\begin{equation}\label{eq:specialWelts}
w_{i}:=s_{n}\dots s_{i}\mbox{ and } w_{j}^{*}=s_{1}^{*}\dots s_{j-1}^{*} \mbox{ with the convention that } w_{n+1}=id\mbox{ and } w_{1}^{*}=id.  
\end{equation}
Note that $w_{i}$ and $w_j^{*}$ are the cycles
\begin{equation}\label{eq:specialexplicit}
w_{i}=(n+1, n, \dots, i+1, i), \ w_j^{*}=(1,2,\dots, j-1, n+1).
\end{equation}
We let $\sigma:=w_{1}$, and note that $\sigma\in W_{n+1}$ is the $n+1$-cycle 
$\sigma=(n+1,\dots, 1)$.

\begin{rem}\label{r:shortestreps}
It is easy to verify that the elements $\{w_{i}:\, i=1,\dots, n+1\}$ are a complete set of shortest length
representatives for the set of right cosets $\Wn\backslash {\cW}_{n+1}$.  Also, the 
set $\{w_{j}^{*}:\, j=1,\dots, n+1\}$ is a set of shortest length representatives for the set of right cosets 
 $\Wn\backslash {\cW}_{n+1}$ provided we define the length function on ${\cW}_{n+1}$  with respect to the simple reflections 
$\{s_{1}^{*}, \dots, s_{n}^{*}\}$.
\end{rem}

\begin{thm}\label{thm:ShandKorbits}
Let $\cZ\in B\backslash\B_{n+1}$ and $Q_{i,j}$ be the $G$-orbit given in Proposition \ref{prop:typeAflag}.   Then 
\begin{equation}\label{eq:ShandK}
\cZ\subset Q_{i,j}\Leftrightarrow \Sh(\cZ)\in \Wn w_{i}\times \Wn w_{j}^{*}.
\end{equation}
\end{thm}
\noindent The following two Lemmas are used to prove Theorem \ref{thm:ShandKorbits}.
\begin{lem}\label{l:oldlemma}(Lemma 5.4 of \cite{CE21II})
Let $\F=(v_{1}\subset v_{2}\subset\dots\subset v_{n}\subset v_{n+1})$ be a flag in standard form.  
\begin{enumerate}
\item The flag $\F\in Q_{j}$ if and only if $\F$ contains no hat vectors and $v_{j}=e_{n+1}$. 
\item  The flag $\F\in Q_{i,j}$ with $i<j$ if and only if $v_{i}$ is a hat vector, 
$v_{k}$ is a standard basis vector for all $k<i$, and $v_{j}=e_{n+1}$.  
\end{enumerate}
\end{lem}
\begin{lem}\label{l:cosets}
\begin{enumerate}
\item Let $w\in {\cW}_{n+1}$.  Then $w\in \Wn w_{i}\Leftrightarrow w(i)=n+1$.  
\item Let $u^{*}\in {\cW}_{n+1}$.  Then $u^{*}\in \Wn w_{j}^{*}\Leftrightarrow u^{*}(j-1)=n+1$.  
\end{enumerate}
\end{lem}
\begin{proof}
 By (\ref{eq:specialexplicit}), $w_{i}(i)=n+1$.  Since $\Wn$ fixes $n+1$, we conclude that $w(i)=n+1$ for all $w\in {\cW}_{n} w_{i}$.

To prove the converse,  suppose we are given $w\in {\cW}_{n+1}$ with $w(i)=n+1$.  Then by Remark \ref{r:shortestreps}, $w\in \Wn w_{\ell}$ for some $\ell\in\{1,\dots, n+1\}$.  But then by the previous paragraph $w(\ell)=n+1=w(i),$  so that $\ell=i$.  


For part (2), let $u^{*}\in \Wn w_{j}^{*}$.   By (\ref{eq:specialexplicit}),  $w_{j}^{*}(j-1)=n+1,$ so as above, $u^{*}(j-1)=n+1.$  The converse is proved in the same manner as for part (1).

\end{proof}
\begin{proof}[Proof of Theorem \ref{thm:ShandKorbits}]
Suppose $\cZ\subset Q_{i,j}$ with $\cZ=B\cdot \F$ with 
$$\F=(v_{1}\subset \dots\subset v_{i-1}\subset v_{i}\subset\dots\subset v_{j}\subset \dots\subset v_{n+1})$$
a flag in standard form and $\Sh(\cZ)=(w, u^{*})$.  First, suppose that $i=j$.  
Then part (1) of Lemma \ref{l:oldlemma} implies that each $v_{\ell}$ is a standard basis vector and $v_{i}=e_{n+1}$.   It then follows from part (1) of Proposition \ref{p:orbits} that 
$\F=\tilde{\F}=\F^{*}$, so that $w(i)=n+1$ and $u^{*}(i-1)=n+1$.  Next, suppose that $i\neq j$, so that Lemma \ref{l:oldlemma} implies that $v_{i}$ is a hat vector, $v_{k}$ is a standard basis vector for all $k<i$, and $v_{j}=e_{n+1}$.  Therefore, we may assume that $\F$ has the form of the flag in Equation (\ref{eq:hatvectorF}) with $i_{k}=i$ and $p=j$.  Now by part (2) of Proposition \ref{p:orbits}, $w(i)=n+1$ and $u^{*}(j-1)=n+1$ (see (\ref{eq:hatvectorF})-(\ref{eq:starF})).  In both cases, it follows from Lemma \ref{l:cosets} that $w\in \Wn w_{i}$ and $u^{*}\in \Wn w_{j}^{*}$.  Thus, $\Sh(\cZ)\in \Wn w_{i}\times \Wn w_{j}^{*}$. 

To prove the converse of (\ref{eq:ShandK}), suppose that $\cZ\in B\backslash\B_{n+1}$ with $\Sh(\cZ)\in \Wn w_{i}\times \Wn w_{j}^{*}$.  Since $\B_{n+1}=\cup_{1\leq i\leq j\leq n+1} Q_{i,j}$, the orbit $\cZ\subset Q_{k,l}$ for some $Q_{k,l}\in G\backslash\B_{n+1}$.  Then by the forward direction of (\ref{eq:ShandK}), $\Sh(\cZ)\in \Wn w_{k}\times \Wn w_{\ell}^*.$ 
It then follows from Remark \ref{r:shortestreps} that $k=i$ and $\ell=j$. 
\end{proof}

In the next subsection, we specialize to the case $i=1$ and $j=i+1$ and recall that $w_1=\sigma,$ in which case Theorem \ref{thm:ShandKorbits} specializes to the following Corollary.

\begin{cor}\label{c:specialcase}
Let $\cZ\in B\backslash\B_{n+1}$ with $\Sh(\cZ)=(w,u^{*})$.  The following statements 
are equivalent.
\begin{enumerate}
\item $\cZ\subset Q_{1,i+1}$.
\item $\Sh(\cZ)\in \Wn \sigma\times \Wn w_{i+1}^{*}.$
\item $w(1)=n+1$ and $u^{*}(i)=n+1$.  
\end{enumerate}

\end{cor}

 \subsection{Shareshian Pairs and the Orbit Correspondence}
 
  In this section, we show how the Shareshian pair of a $B$-orbit 
 $\cZ$ contained in $Q_{1,i+1}$ can be used to obtain the standardized $i$-Shareshian pair of the corresponding $S_{i}$-orbit ${\cZ}^{op}$ in $S_{i}\backslash\B_{n}$.  We use this result to prove Corollary \ref{c:closure}, which is the analogue of (\ref{eq:oldShclosure}) for $S_{i}$-orbits and $i$-Shareshian pairs.  This result will also play a substantial role in the proof of one of the main results of the paper, namely Equation (\ref{eq:introclosure}), in the next section.

\begin{thm}\label{thm:localShareshian}
Let $\cZ\in B\backslash Q_{1,i+1}$ with $\Sh(\cZ)=(w,z^{*})$.  Then the \emph{standardized $i$-Shareshian }pair of ${\cZ}^{op}\in S_{i}\backslash\B_{n}$ is:
\begin{equation}\label{eq:localSh}
 \widetilde{Sh}_{i}({\cZ}^{op})=(\sigma w^{-1}, w_{i+1}^{*} {z^{*}}^{-1}).
\end{equation}
\end{thm}
\begin{proof}
  We first prove Equation (\ref{eq:localSh})  when $\cZ$ is a closed $B$-orbit on $Q_{1,i+1}$ (or equivalently, $\cZ^{op}$ is a zero-dimensional $S_i$-orbit), and then deduce the general case using properties of the monoid action.  
By Equation (6.22) of \cite{CE21II},  each of these closed orbits  is 
$\cZ={\cZ}_{1,i+1}(k):=B\cdot \mathcal{E}_{1,i+1}(k)$, $k=1,\dots, i$, where $\mathcal{E}_{1,i+1}(k)$ is the flag in standard form 
\begin{equation}\label{eq:firststd}
 \mathcal{E}_{1,i+1}(k):=(\hat{e}_{k}\subset e_{1}\subset \dots\subset e_{k-1}\subset e_{k+1}\subset\dots\subset e_{i-1}\subset e_{i}\subset \underbrace{e_{n+1}}_{i+1}\subset e_{i+1}\subset \dots\subset e_{n}). 
\end{equation}
For brevity, we denote $\mathcal{E}_{1,i+1}(k)$ by $\F$ in our
 computation of $\Sh(\cZ)$.  By Proposition \ref{p:orbits}, the corresponding flags $\tilde{\F}$ and $\F^{*}$ of Notation \ref{n:tildeandstarflags} are given by
\begin{eqnarray}\label{eq:tildeflag}
\begin{split}
\tilde{\F}=(e_{n+1}\subset e_{1}\subset \dots\subset e_{k-1}\subset e_{k+1}\subset\dots\subset e_{i-1}\subset e_{i}\subset &\underbrace{e_{k}}_{i+1}\subset e_{i+1}\subset \dots \subset e_{n})\\
& and\\
\F^{*}=(e_{k}\subset e_{1}\subset \dots\subset e_{k-1}\subset e_{k+1} \subset\dots\subset e_{i-1}\subset e_{i}\subset &\underbrace{e_{n+1}}_{i+1}\subset e_{i+1}\subset \dots \subset e_{n}).
\end{split}
\end{eqnarray}
Thus,  $\Sh({\cZ}_{1,i+1}(k))=(w,z^{*})$, where
\begin{equation}\label{eq:Shdata}
w=(n+1,\, n, \, n-1,\dots, i+2,\, i+1,\, k,\, k-1,\dots, 2,\, 1) \mbox{ and } z^{*}=(k,\, k+1,\, \dots,\, i-1,\, i,\, n+1)
\end{equation}
 are cycles in $W_{n+1}.$

We now compute the standardized $i$-Shareshian pair of ${\cZ}^{op}$.  Recall the flag $\mathcal{E}_{1,i+1}$ in Equation (\ref{eq:Fioneflag}).  Consider the $i-k+1$-cycle $\tau=(k, k+1, \dots, i)$.  Observe that $\F=\tau(\mathcal{E}_{1,i+1}),$
so that $\cZ=B\cdot \tau (\mathcal{E}_{1,i+1}).$  Thus by Equation (\ref{eq:firstcorresp}), ${\cZ}^{op}= S_{i}\cdot\tau^{-1}(\mathcal{E}_{n})$.   Hence, $B\cdot {\cZ}^{op}=B\cdot \tau^{-1}(\mathcal{E}_{n})$ and by (\ref{eq:Eiflag}), 
$$B^{i}\cdot {\cZ}^{op}=B^{i}\cdot \tau^{-1}(\mathcal{E}_{n})=B^{i}\cdot \tau^{-1}\sigma_{i}^{-1}(\mathcal{E}^{i}).$$
Thus, by (\ref{eq:iShdefn}) and Definition \ref{d:istd},
$\Sh_i(\cZ^{op})=(\tau^{-1}, \tau^{-1}\sigma_i^{-1})$ and the standardized $i$-Shareshian pair is  
\begin{equation}\label{eq:twistedmin}
\widetilde{\Sh}_{i}(\cZ^{op})=(\tau^{-1}, \sigma_{i}^{-1}\tau^{-1}).   
\end{equation}
It remains to verify that the expressions in (\ref{eq:twistedmin}) and in (\ref{eq:localSh}) coincide.  For this equality, we verify that 
$\sigma w^{-1} = \tau^{-1}$ and $w_{i+1}^{*}{z^{*}}^{-1}=\sigma_{i}^{-1}\tau^{-1}$,
 using (\ref{eq:Shdata}) and (\ref{eq:specialexplicit}) and explicit computations in $W_{n+1}.$
Thus, Equation (\ref{eq:localSh}) holds for zero dimensional $S_{i}$-orbits ${\cZ}^{op}.$

Recall the set of all Shareshian pairs $\Sp\subset W_{n+1}\times W_{n+1}$ given in Equation (\ref{eq:Spdefn}).  
For the inductive step, we define a subset of Shareshian pairs 
$$\Sp_{1,i+1}:=\{(w, z^{*})\in \Sp|\; (w, z^{*})=\Sh(\cZ)\mbox{ for some } \cZ\in B\backslash Q_{1,i+1}\}.$$  
By Corollary \ref{c:specialcase},  
 \begin{equation}\label{eq:SHcoset}
  \Sp_{1,i+1}=\Sp\cap (\Wn \sigma\times \Wn w_{i+1}^{*}).
 \end{equation}
 Thus, for $(w,z^{*}) \in \Sp_{1,i+1},$ the products $w \sigma^{-1}$ and $z^{*}{w_{i+1}^{*-1}}$ are in $\Wn,$ so we obtain a map 
 \begin{equation}\label{eq:Psidefn}
 \Psi: \Sp_{1,i+1}\mapsto \Wn\times \Wn\mbox{ given by }
 \Psi((w,z^{*}))=(\sigma  w^{-1}, w_{i+1}^{*}{z^{*}}^{-1}).
 \end{equation}
For the remainder of the proof, for $\alpha \in \Pi_{\fg},$ we let $\alphatilde=\alpha$ and for $\alpha=\alpha_j \in \Pi_{n+1,cpt},$ we let $\alphatilde = \alpha_{j-1}.$
 Note that (\ref{eq:localSh}) is equivalent to the assertion that
\begin{equation}\label{eq:tildepsi}
 \widetilde{\Sh}_{i}(\cZ^{op})=\Psi(\Sh(\cZ)), \forall \cZ \in B\backslash Q_{1,i+1}.
\end{equation}
By the proof of Theorem \ref{thm:Siweak}, for each $\cZ$ as above, there is a closed $B$-orbit ${\cZ}_0$ in $Q_{1,i+1}$ and a sequence $\vec{s}$ of simple roots $\alpha_{j_1}, \dots, \alpha_{j_k} \in \Pi_{\fg}\coprod \Pi_{n+1,cpt}$ such that $m(\vec{s})*{\cZ}_0=\cZ.$  Hence, to complete the proof, it suffices to prove the claim that if (\ref{eq:tildepsi}) is true for an orbit $\cZ,$ then (\ref{eq:tildepsi}) holds with $\cZ^{\prime}$ in place of $\cZ,$ where $\cZ^{\prime}:=\ms*\cZ\neq \cZ$ with $\alpha \in\Pi_{\fg}\coprod \Pi_{n+1,cpt}.$ To establish the claim, it suffices to show that for $\alpha$ and $\cZ$ as above, then
 \begin{equation}\label{eq:Psiinter}
\Psi(\ms*\Sh(\cZ))= m(s_{\alphatilde})*\Psi(\Sh(\cZ)).
 \end{equation}
Here for $\alpha\in \Pi_{\fg},$ the monoid action is on the left  on the left hand side of the equation, and on the right on the right hand side of the equation, and for $\alpha \in \Pi_{n+1,cpt}$, the monoid action is on the right on the left hand side of the equation and on the left on the right hand side of the equation. 
 Indeed, assuming (\ref{eq:Psiinter}), then we deduce the following sequence of equalities which give (\ref{eq:tildepsi}) for $\cZ^{\prime}.$
 \begin{equation*}
 \begin{split}
 \Psi(\Sh(\cZ^{\prime}))&=\Psi(\ms*\Sh(\cZ))=m(s_{\alphatilde})*\Psi(\Sh(\cZ))\\
 &=m(s_{\alphatilde})*\widetilde{\Sh}_{i}(\cZ^{op})=\widetilde{\Sh}_{i}(m(s_{\alphatilde})*\cZ^{op})=\widetilde{\Sh}_
{i}((\ms*\cZ)^{op})\\
&= \widetilde{\Sh}_{i}({\cZ^{\prime}}^{op}).
 \end{split}
 \end{equation*}
 Indeed, the first equality follows from Remark \ref{r:Theorem4.7} for the case of $B_n$-orbits on $\B_{n+1},$ since ${\mathfrak{S}}_{n+1}=\Pi_{\fg}$ by Equation 
 (\ref{eq:extraroots}).  The second equality is from (\ref{eq:Psiinter}), and the third equality is by (\ref{eq:tildepsi}).  The fourth equality is by Remark \ref{r:standardmonoid},  which applies since $\alpha \in \Pi_{n+1,cpt}$ implies $\alphatilde \in {\mathfrak{S}}_{i}$ by Equations (\ref{eq:extraroots}) and (\ref{eq:compactroots}).  The fifth equality is by Equation (\ref{eq:leftinter}) for $\alpha \in \Pi_{\fg}$ and by Equation (\ref{eq:rightinter}) for $\alpha \in \Pi_{n+1,cpt}$  and the sixth equality is immediate from the definition of $\cZ^{\prime}.$

Now let $\Sh(\cZ)=(w,z^{*})$ and suppose first that $\alpha\in\Pi_{\fg}$ is complex stable for $\cZ.$  Then by Theorem \ref{thm:intertwine} and Remark \ref{r:Theorem4.7}, $\alpha$ is stable for both factors of $\Sh(\cZ),$ so that $\ms*(w,z^{*})= (s_{\alpha} w, \, s_{\alpha} z^{*}),$ and hence $\Psi(\ms*\Sh(\cZ))=(\sigma w^{-1} s_{\alpha},w_{i+1}^{*} {z^{*}}^{-1}s_{\alpha}).$   By Theorem \ref{thm:monoidcorresp}, $\alpha$ is complex stable for ${\cZ}^{op},$ so by Remark \ref{r:standardmonoid}, $\alpha$ is  stable for both factors of $\widetilde{\Sh}_{i}(\cZ^{op}),$ and hence by (\ref{eq:tildepsi}) for $\cZ$, for both factors of $(\sigma w^{-1}, w_{i+1}^{*} {z^{*}}^{-1}).$  Thus, $\ms*\Psi((w,z^{*}))=(\sigma w^{-1}s_{\alpha}, w_{i+1}^{*} {z^{*}}^{-1}s_{\alpha}),$ which proves (\ref{eq:Psiinter}) in this case.

Now suppose that $\alpha\in\Pi_{\fg}$ is non-compact for the Shareshian pair $(w,z^{*})$.  Then by Theorem \ref{thm:intertwine} and Remark \ref{r:Theorem4.7}, $\alpha$ is stable for either $w$ or $z^{*}$ and unstable for the other.  First, assume that $\alpha$ is stable for $w$, i.e., $s_{\alpha}w>w,$ or equivalently $w^{-1}(\alpha)$ is a positive root.  We claim that $\sigma w^{-1}(\alpha)$ is a positive root.    Indeed, by Corollary \ref{c:specialcase}(3), $w^{-1}(n+1)=1,$ so that for $\alpha \in \Pi_{\fg},$ $w^{-1}(\alpha)$ is in the span of $\{ \alpha_2, \dots, \alpha_n \}.$  Let $\Phi_{\sigma}:= \{ \beta \in \Phi^+ : \sigma(\beta) \in -\Phi^+ \}.$   Since $\sigma = s_n s_{n-1} \dots s_1$ is a reduced decomposition of $\sigma,$  then $\Phi_{\sigma}= \{ \alpha_1, s_{1}(\alpha_2), \dots, s_{1} s_{2} \dots s_{n-1}(\alpha_n) \}.$  It follows that $w^{-1}(\Pi_{\fg})\cap \Phi_{\sigma}$ is empty, which proves the claim. Thus, $\sigma w^{-1} s_{\alpha} > \sigma w^{-1}.$   By Theorem \ref{thm:monoidcorresp}, $\alpha$ is non-compact imaginary for ${\cZ}^{op},$ so by Remark \ref{r:standardmonoid} and (\ref{eq:tildepsi}), $\alpha$ is stable for exactly one factor of $\widetilde{\Sh}_{i}(\cZ)=(\sigma w^{-1}, w_{i+1}^{*} {z^{*}}^{-1}).$  It now follows that
\begin{equation}
\Psi(\ms*(w,z^{*}))=(\sigma w^{-1} s_{\alpha}, w_{i+1}^{*}{z^{*}}^{-1})=\ms*\Psi((w,z^{*})),
\end{equation}
which verifies (\ref{eq:Psiinter}) in this case.

  
Continuing the case where $\alpha\in\Pi_{\fg}$ is non-compact for the Shareshian pair $(w,z^*)$, but that $\alpha\in\Pi_{\fg}$ is unstable for $w$ but $\alpha$ is stable for $z^{*}.$  It follows that $w^{-1}(\alpha)$ is a negative root.  As in the previous paragraph, we deduce that $\sigma w^{-1}(\alpha)$ is a negative root, and the remainder of the verification of (\ref{eq:Psiinter}) in this case is similar to the argument in the previous paragraph.

  
 We now consider the case where $\alpha\in\Pi_{n+1,cpt}$.  Recall that $\Pi_{n+1,cpt}=\{\alpha_{2},\dots, \alpha_{i-1}\}\cup \{\alpha_{i+2},\dots, \alpha_{n}\}$ and by Equation (\ref{eq:rightinter}), letting $\alpha=\alpha_{j}$, then $(m(s_{j})*\cZ)^{op}=m(s_{j-1})*\cZ^{op}$.   By Theorem \ref{thm:intertwine} and Remark \ref{r:Theorem4.7}, $\alpha$ is complex stable for $\cZ$ if and only if $\alpha$ is stable for both factors of $\Sh(\cZ)$ and $\alpha$ is non-compact for $\cZ$ if and only if $\alpha$ is stable for one factor of $\Sh(\cZ)$ and unstable for the other factor.  We first consider the case when $\alpha$ is complex stable for $\cZ$.   Hence, by (\ref{eq:diagonalmonoid}), $\ms*\Sh(\cZ)=(ws_{j},z^*s_{j}^{*})$, so
\begin{equation}\label{eq:extraeq}
\Psi(\ms*\Sh(\cZ))=(\sigma s_jw^{-1},w_{i+1}^{*}s_{j}^{*}{z^{*}}^{-1}).
\end{equation}
 On the other hand, by Theorem \ref{thm:monoidcorresp} and Equation (\ref{eq:rightinter}), $\tilde{\alpha}=\alpha_{j-1}$ is complex stable for ${\cZ}^{op}$ acting on the left, so by Remark \ref{r:standardmonoid}, $\alphatilde$ is complex stable for both factors of $\widetilde{\Sh}_{i}({\cZ}^{op}),$ which coincides with $\Psi(\Sh(\cZ))$ by (\ref {eq:tildepsi}) for $\cZ.$  Thus, 
\begin{equation}\label{eq:firstnew}
m(s_{\alphatilde})*\Psi(\Sh(\cZ))=m(s_{j-1})*\Psi(w,z^*)=m(s_{j-1})*(\sigma w^{-1}, w_{i+1}^{*}{z^{*}}^{-1}). 
\end{equation}
By (\ref{eq:twistinsecond}), this last expression is
\begin{equation}\label{eq:secondnew}
(s_{j-1}\sigma w^{-1}, s_{\sigma_{i}^{-1}(j-1)}w_{i+1}^{*}{z^*}^{-1})=(\sigma s_{\sigma^{-1}(j-1)}w^{-1},w_{i+1}^{*}s_{(\sigma_{i} w_{i+1}^{*})^{-1}(j-1)}{z^{*}}^{-1}).
\end{equation}
However, $\sigma^{-1}(j-1)=j$ by (\ref{eq:specialexplicit}), and $\sigma_{i} w_{i+1}^{*}$ is the transposition $(i, n+1)$ by (\ref{e:sigmai}) and (\ref{eq:specialexplicit}), so $\sigma_{i} w_{i+1}^{*}$ fixes each $\alpha_{j-1}$ since  $j\in \{ 2, \dots, i-2, i-1, i+2, \dots, n+1 \}.$  For such $j,$
$\alpha_{j}^{*}=\alpha_{j-1},$ so combining (\ref{eq:firstnew}) and (\ref{eq:secondnew}), we obtain $m(s_{\alphatilde})*\Psi(\Sh(\cZ))=(\sigma s_{j} w^{-1},w_{i+1}^{*}s_{j}^{*}{z^*}^{-1}),$ which coincides with (\ref{eq:extraeq}).  This verifies (\ref{eq:Psiinter}) when $\alpha$ is complex stable for $\cZ.$

In case $\alpha$ is non-compact for $\cZ$, then by the above, $\alpha$ is complex stable for exactly one of $w$ and $z^*$ and is unstable for the other.   By arguments similar to those used in the the case when $\alpha \in \Pi_{\fg},$ it follows that $\alphatilde$ is non-compact for ${\cZ}^{op}$ and complex stable for exactly one component of $\widetilde{\Sh}_{i}({\cZ}^{op})$ and unstable for the other component.   Suppose first that $\alpha$ is complex stable for $w,$ so that $w(\alpha)\in \Phi^+.$    Then $w\sigma^{-1}(\alphatilde)=w(\alpha) \in \Phi^+$, so $\alphatilde$ is stable for $\sigma w^{-1}.$  The verification of (\ref{eq:Psiinter}) now proceeds in the same way as for the first component for the case where $\alpha$ is complex for $\cZ.$   In case $\alpha$ is unstable for $w$, then a small variant of the above argument shows that $\alphatilde$ is unstable for $\sigma w^{-1}$ and hence stable for $w_{i+1}^{*}{z^{*}}^{-1}.$  The verification of (\ref{eq:Psiinter}) now proceeds in the same way as for the second component in the case where $\alpha$ is complex for $\cZ.$

\end{proof}

We now show that the map in Equation (\ref{eq:localSh}) preserves the Bruhat order on Shareshian pairs.   For this, we use the following combinatorial result, which can be found in Chapter 2, Exercise 21 of \cite{BB}.
 \begin{prop}\label{p:BBresult}
 Let $x,\, y, \, w\in{\cW}_{n+1}$.  Suppose $\ell(xw)=\ell(x)+\ell(w)$ and $\ell(yw)=\ell(y)+\ell(w)$.  Then $xw\leq yw\Leftrightarrow x\leq y.$
\end{prop}

\begin{thm}\label{thm:Shorder}
Let $\cZ^{\prime},\, \cZ\in B\backslash Q_{1,i+1}$ with ${\cZ^{\prime}}^{op}$ and $\cZ^{op}$ 
the corresponding $S_{i}$-orbits on $\B_{n}$.  Then 
\begin{equation}\label{eq:orderequiv}
\Sh(\cZ^{\prime})\leq \Sh(\cZ)\Leftrightarrow \widetilde{\Sh}_{i}({\cZ^{\prime}}^{op})\leq \widetilde{\Sh}_{i}(\cZ^{op}). 
\end{equation}
\end{thm}

\begin{proof}
Let $\Sh(\cZ^{\prime})=(v,y^{*})$ and let $\Sh(\cZ)=(w,z^{*})$.  Then by
Equation (\ref{eq:localSh}), $\widetilde{\Sh}_{i}({\cZ^{\prime}}^{op})=(\sigma v^{-1}, w_{i+1}^{*} {y^{*}}^{-1})$ and $\widetilde{\Sh}_{i}(\cZ^{op})=(\sigma w^{-1}, w_{i+1}^{*}{z^{*}}^{-1})$.  By Part (2) of Corollary \ref{c:specialcase} $v=\alpha\sigma$ and $w=\beta \sigma$ with $\alpha,\, \beta\in \Wn$.  By Remark \ref{r:shortestreps}, $\sigma$ is a shortest right coset representative for $W_n \subset W_{n+1},$ so that 
$\ell(v)=\ell(\alpha)+\ell(\sigma)$ and $\ell(w)=\ell(\beta)+\ell(\sigma)$.  
Proposition \ref{p:BBresult} and the fact that the Bruhat order is invariant under inversion imply that $v\leq w\Leftrightarrow \alpha\leq\beta\Leftrightarrow \alpha^{-1}\leq \beta^{-1}$.  But $\alpha^{-1}=\sigma v^{-1}$ and $\beta^{-1}=\sigma w^{-1}$, so we obtain 
\begin{equation}\label{eq:firstcpt}
v\leq w\Leftrightarrow \sigma v^{-1}\leq \sigma w^{-1}.
\end{equation}
For the second components of the Shareshian pairs, it follows from (2) of Corollary \ref{c:specialcase} that 
$y^{*}=\gamma^{*}w_{i+1}^{*}$ and $z^{*}=\delta^{*}w_{i+1}^{*}$ 
with $\gamma^{*}, \, \delta^{*}\in \Wn$.  Recall from (\ref{eq:starreflections}) that $s_{i}^{*}=s_{i-1}$ for $2\leq i\leq n$, so that the restriction of the Bruhat order $\leq_{*}$  on $({\cW}_{n+1}, S^{*})$ to 
the subgroup $\Wn=\langle s_{2}^{*},\dots, s_{n}^{*}\rangle$ is just the standard Bruhat order on $\Wn$.  Arguing as above, it follows from Proposition \ref{p:BBresult} and Remark \ref{r:shortestreps} that 
\begin{equation}\label{eq:secondcpt}
y^{*}\leq_{*} z^{*}\Leftrightarrow w_{i+1}^{*} {y^{*}}^{-1}\leq w_{i+1}^{*}{z^{*}}^{-1}.
\end{equation}
Equation (\ref{eq:orderequiv}) now follows from Equations (\ref{eq:firstcpt}) and  (\ref{eq:secondcpt}) and the definition of the Bruhat order on standardized $i$-Shareshian pairs given in (\ref{eq:istdShorder}). 
\end{proof}

We now deduce the following result, which generalizes Theorem 3.4 of \cite{Shpairs} from $B_{n-1}$-orbits to $S_i$-orbits.
\begin{cor}\label{c:closure}
The closure ordering on $S_{i}$-orbits on $\B_{n}$ coincides with the Bruhat order on 
$i$-Shareshian pairs.
\end{cor}
 \begin{proof}
 Let ${\cZ^{\prime}}^{op}, \, \cZ^{op}\in S_{i}\backslash\B_{n}$ correspond to 
 $\cZ^{\prime},\, \cZ\in B\backslash Q_{1,i+1}$.
 By Equations (\ref{eq:closurecorres}) and (\ref{eq:oldShclosure}) and Theorem \ref{thm:Shorder}, we have the following equivalences:
 $$
 {\cZ^{\prime}}^{op}\subset \overline{\cZ^{op}}\Leftrightarrow \cZ^{\prime}\subset \overline{\cZ}\Leftrightarrow \Sh(\cZ^{\prime})\leq \Sh(\cZ)\Leftrightarrow\widetilde{\Sh}_{i}({\cZ^{\prime}}^{op})\leq \widetilde{\Sh}_{i}(\cZ^{op}). 
 $$
 \end{proof}

Using Corollary \ref{c:closure}, we can prove: 
\begin{cor}\label{c:Schubertintersect}
Let $Q\in S_{i}\backslash\B_{n}$ with $\Sh_{i}(Q)=(w, u^{i})$ so that 
$Q=B\cdot w(\mathcal{E})\cap B^{i}\cdot u^{i}(\mathcal{E}^{i})$ by (\ref{eq:iShdefn}).  Then 
$$
\overline{Q}=\overline{B\cdot w(\mathcal{E})}\cap\overline{B^{i}\cdot u^{i}(\mathcal{E}^{i})}
$$
is the intersection of a $B$-Schubert variety and a $B^{i}$-Schubert variety. 
\end{cor}
\begin{proof}
By the definition of an $i$-Shareshian pair in (\ref{eq:iShdefn}), we have $\overline{Q}=\overline{B\cdot w(\mathcal{E})\cap B^{i}\cdot u^{i}(\mathcal{E}^{i})}\subset \overline{B\cdot w(\mathcal{E})}\cap\overline{B^{i}\cdot u^{i}(\mathcal{E}^{i})}$.  For the other inclusion, let $Q^{\prime}\subset  \overline{B\cdot w(\mathcal{E})}\cap\overline{B^{i}\cdot u^{i}(\mathcal{E}^{i})}
.$  Then $B\cdot Q^{\prime}\subset \overline{B\cdot w(\mathcal{E})}$ and $B^{i}\cdot Q^{\prime}\subset \overline{B^{i}\cdot u^{i}(\mathcal{E}^{i})}$.  Thus, $\Sh_{i}(Q^{\prime})\leq \Sh_{i}(Q)$.  It follows from Corollary \ref{c:closure} that $Q^{\prime}\subset \overline{Q}$, and the proof is complete. 
\end{proof}

It would be interesting to apply this observation to the study of the singularities of closures of $S_i$-orbits.


 \subsection{Additional Monoid Actions on $S_{i}$-orbits}
There are two other monoid actions on $S_{i}\backslash\B_{n}$ that are not discussed in Section \ref{ss:Simonoid}.  
These monoid actions were discussed briefly in Remark 5.8 of \cite{Siorbits} and are included here to provide a more complete picture, but are not needed to prove the main results of this paper.

 Consider $\alpha=\alpha_{i-1}\in\Pi_{\fg}$.  Then $P_{\alpha_{i-1}}\cdot[e_{i}]=B\cdot [e_{i-1}]\cup B\cdot[e_{i}]$ with $B\cdot [e_{i}]$ the open orbit.
Consider the $S_{i}$-orbit $Q_{\fb^{\prime}}=S_{i}\cdot\fb^{\prime}$ corresponding to the $G$ and $B$-orbits $\mathcal{O}_{\Delta,\fb^{\prime}}=G\cdot(\fb,\fb^{\prime}, [e_{i}])$ and $\mathcal{O}_{B,\fb^{\prime}}=B\cdot(\fb^{\prime}, [e_{i}])$ respectively as in Equation (\ref{eq:three}).  By Equation (\ref{eq:firstfactorfirst}),
$m(s_{i-1})*_{1} \mathcal{O}_{\Delta,\fb^{\prime}}$ corresponds to the open  $B$-orbit in  $P_{\alpha_{i-1}}\cdot (\fb^{\prime}, [e_{i}])$,
which is easily seen to be contained in the variety 
$\B_{n}\times\mathcal{O}_{i}$ and hence corresponds to a unique $S_{i}$-orbit under the correspondence in (\ref{eq:Sicorres}).  This allows us to define:
\begin{equation}\label{eq:firstexoticdefn}
m(s_{i-1})*_{L} Q_{\fb^{\prime}}:=\mbox{ unique } S_{i}-\mbox{orbit corresponding to open } B-\mbox{orbit in } P_{\alpha_{i-1}}\cdot(\fb^{\prime}, [e_{i}]). 
\end{equation}
We now verify the assertions of Remark 5.8 of \cite{Siorbits}.
\begin{prop}\label{p:firstweirdmonoid}
Let $Q_{\fb^{\prime}}=S_{i}\cdot \fb^{\prime}$ and $m(s_{i-1})*_{L}Q_{\fb^{\prime}}$ be given by Equation (\ref{eq:firstexoticdefn}).  Then 
\begin{enumerate}
\item The root $\alpha_{i-1}$ is never complex stable for $Q_{\fb^{\prime}}$.
\item  If $\Sh_{i}(Q_{\fb^{\prime}})=(w, u^{i})$, then 
 \begin{equation}\label{eq:firstSh}
 \Sh_{i}(m(s_{i-1})*_{L} Q_{\fb^{\prime}})=(m(s_{i-1})*_{L} w, u^{i}),
 \end{equation}
 where $m(s_{i-1})*_{L}w$ denotes the left action of $s_{i-1}$ on $w$ as in (\ref{eq:Wmonoidact}). 
\end{enumerate}
\end{prop}
\begin{proof}
By way of contradiction suppose that $\alpha=\alpha_{i-1}$ is complex stable 
for $Q_{\fb^{\prime}}$.  Let $\mathcal{O}_{\Delta, \fb^{\prime}}$ be the corresponding 
$G$-orbit as in Equation (\ref{eq:three}).  It follows from definitions and part (1) of Proposition \ref{p:stableandnc} that
\begin{equation*} 
\begin{split}
m(s_{i-1})*_{1}\mathcal{O}_{\Delta, \fb^{\prime}}&=G\cdot (s_{i-1}(\fb), \fb^{\prime}, [e_{i}])=G\cdot (\fb, \dot{s}_{i-1}\cdot \fb^{\prime}, \dot{s}_{i-1}\cdot[e_{i}])\\
&=G\cdot(\fb, \dot{s}_{i-1}\cdot\fb^{\prime}, [e_{i-1}])\leftrightarrow B\cdot (\dot{s}_{i-1}\cdot\fb^{\prime}, [e_{i-1}]),
\end{split}
\end{equation*} 
where $\dot{s}_{i-1}$ is a representative of $s_{i-1}\in W_{n}$.  But the latter $B$-orbit is not contained in the open $B$-orbit in $P_{\alpha_{i-1}}\cdot (\fb^{\prime}, [e_{i}])$, which gives a contradiction. 

For $Q_{\fb^{\prime}}\in S_{i}\backslash\B_{n}$ let $(Q_{\fb^{\prime}})_{B}:=B\cdot Q_{\fb^{\prime}}$ be the unique $B$-orbit on $\B_{n}$ containing the $S_{i}$-orbit $Q_{\fb^{\prime}}$.  Similarly, let $(Q_{\fb^{\prime}})_{B^{i}}:=B^{i}\cdot Q_{\fb^{\prime}}$.  To prove part (2), note that it suffices to show that 
\begin{equation}\label{eq:leftcpt}
(m(s_{i-1})*_{L}Q_{\fb^{\prime}})_{B}=m(s_{i-1})*_{L}(Q_{\fb^{\prime}})_{B}
\end{equation}
and
\begin{equation}\label{eq:rightcpt}
(m(s_{i-1})*_{L} Q_{\fb^{\prime}})_{B^{i}}=(Q_{\fb^{\prime}})_{B^{i}}. 
\end{equation}
To demonstrate (\ref{eq:leftcpt}), observe that if 
$Q_{\fb^{\prime}}$ corresponds to the $B$-orbit $\mathcal{O}_{B,\fb^{\prime}}=B\cdot (\fb^{\prime}, [e_{i}])$ as in (\ref{eq:three}), then $(Q_{\fb^{\prime}})_{B}=B\cdot \fb^{\prime}=\pi_{1}(\mathcal{O}_{B,\fb^{\prime}})$ where $\pi_{1}:\B_{n}\times\mathbb{P}^{n-1}\to \B_{n}$ is projection onto the first factor.  
Thus, since $\pi_{1}$ is an open map, by Equation (\ref{eq:firstexoticdefn})
\begin{equation*}
\begin{split}
(m(s_{i-1})*_{L}(Q_{\fb^{\prime}}))_{B}&=\pi_{1}(\mbox{open } B-\mbox{orbit in } P_{\alpha_{i-1}} \cdot (\fb^{\prime}, [e_{i}]))\\
&=\mbox{ open } B-\mbox{orbit in } P_{\alpha_{i-1}}\cdot\fb^{\prime}\\&=m(s_{i-1})*_{L}(Q_{\fb^{\prime}})_{B},
\end{split}
\end{equation*}
yielding (\ref{eq:leftcpt}).  

To demonstrate (\ref{eq:rightcpt}), we use part (1) of the Proposition.  
 If $m(s_{i-1})*_{L}Q_{\fb^{\prime}}=Q_{\fb^{\prime}}$, then (\ref{eq:rightcpt}) is clear, so we may assume that $m(s_{i-1})*_{L}Q_{\fb^{\prime}}\neq Q_{\fb^{\prime}}$.  Then by (1), $\alpha_{i-1}$ must be non-compact imaginary for 
$Q_{\fb^{\prime}}$.  Thus, $m(s_{i-1})*_{L}\mathcal{O}_{\Delta,\fb^{\prime}}=G\cdot (\Ad(u_{i-1})\cdot \fb,\fb^{\prime}, [e_i]),$ by Proposition \ref{p:stableandnc}, where $u_{i-1}$ is the Cayley transform with respect to $\alpha_{i-1}$.  Thus, $m(s_{i-1})*_{L}Q_{\fb^{\prime}}=S_{i}\cdot \Ad(u_{i-1}^{-1})\cdot \fb^{\prime}$.  By Remark \ref{r:Cayley}, the element $u_{i-1}$ acts on the standard flag in the same way as the element $b_{i}$ given by $b_{i}:e_{i-1}\mapsto e_{i-1}+e_{i}$ and 
$b_{i}: e_{\ell}\mapsto e_{\ell}$ for all other $\ell$.  Since, $b_{i}\in B^{i}$ (see Equation (\ref{eq:Eiflag})) and (\ref{eq:rightcpt}) follows.
\end{proof}

We now show that the monoid action in (\ref{eq:firstexoticdefn}) corresponds to a natural monoid action on $B\backslash Q_{1,i+1}$ via the correspondence in Theorem \ref{thm:firstbig}.  By Proposition \ref{p:Kmonoidaction}, the root $\alpha_{i}\in\Pi_{\fg_{n+1}}$ with $i>1$ is complex unstable for the $G$-orbit $Q_{1,i+1}$, so that  $m(s_{i})*Q_{1,i+1}=Q_{1,i+1}$.  For $\cZ\in B\backslash Q_{1,i+1}$, Equation (\ref{p:old4.7}) implies that $m(s_{i})*_{R}\cZ\in B\backslash Q_{1,i+1}$.  
 
  \begin{prop}\label{p:unstablecorrespondence}
 Let $\alpha=\alpha_{i}\in\Pi_{\fg_{n+1}}$ with $i>1$ and $\cZ\in B\backslash Q_{1,i+1}$ with $\cZ^{op}\in S_{i}\backslash\B_{n}$ as in Notation \ref{n:op}.
 Then 
 \begin{enumerate}
 \item If $\Sh(\cZ)=(w,u^{*})$, then $\Sh(m(s_{i})*_{R}\cZ)=(m(s_{i})*_{R} w, u^{*})$.  Thus, $m(s_{i})*_{R}\cZ\neq \cZ$ if and only if $\alpha$ is non-compact for $\cZ$.  
 \item The corresponding $S_{i}$-orbit $(m(s_{i})*_{R}\cZ)^{op}=m(s_{i-1})*_{L}\cZ^{op}$, where the monoid action on the right hand side of the equation is defined in Equation (\ref{eq:firstexoticdefn}).
 \end{enumerate}
 \end{prop}
 \begin{proof}
 By Theorem \ref{thm:ShandKorbits}, $\Sh(\cZ)\in W\sigma\times W w_{i+1}^{*}$, where $w_{i+1}^{*}=s_{1}^{*}\dots s_{i}^{*}$ by Equation (\ref{eq:specialWelts}).  Hence, if $\Sh(\cZ)=(w,u^{*})$, then $u^{*}s_{i}^{*}<u^{*}$, so 
 $\alpha^{*}$ is necessarily unstable for $u^{*}$.  Therefore, by Theorem \ref{thm:intertwine} for $\fg_{n+1}$ with $i=n+1$ (see Remark \ref{r:Theorem4.7}), the root $\alpha$ can never be complex stable for $\cZ$, and $\alpha$ is non-compact for $\cZ$ if and only if $ws_{i}>w$.  Then $\Sh(m(s_{i})*_{R} \cZ)=m(s_{i})*_{R} \Sh(\cZ)=(m(s_{i})*_{R} w, u^{*})$.  By Theorem \ref{thm:intertwine} and Remark \ref{r:Theorem4.7}, the rest of (1) follows.  
 
Since $\widetilde{\Sh}_{i}:S_i\backslash\B_n \to W_{n}\times W_{n}$ is injective by Corollary 3.18 and Definition 3.22 of \cite{Siorbits}, to prove part (2) of the Proposition, it suffices to show that $\widetilde{\Sh}_{i}((m(s_{i})*_{R}\cZ)^{op})=\widetilde{\Sh}_{i}(m(s_{i-1})*_{L}\cZ^{op}).$   By (\ref{eq:localSh}), $\widetilde{Sh}_{i}({\cZ}^{op})=(\sigma w^{-1}, w_{i+1}^{*} {u^{*}}^{-1}).$   Hence, by (\ref{eq:firstSh}) and Definition \ref{d:istd},  
\begin{equation}\label{e:miminusoneside}
\widetilde{Sh}_{i}(m(s_{i-1})*_{L}{\cZ}^{op})=((m(s_{i-1})*_{L}(\sigma w^{-1}), w_{i+1}^{*} {u^{*}}^{-1}).
\end{equation}
By part (1) and (\ref{eq:localSh}),
\begin{equation}\label{e:miside}
\widetilde{\Sh}_{i}((m(s_{i})*_{R}\cZ)^{op})=(\sigma(m(s_i)*_{R} w)^{-1}, w_{i+1}^{*} {u^{*}}^{-1}).
\end{equation}
Thus,  part (2) of the Proposition follows from the assertion that
\begin{equation}\label{eq:wequality}
 m(s_{i-1})*_{L}(\sigma w^{-1})=\sigma(m(s_i)*_{R} w)^{-1}.
\end{equation}
Let $w=x\sigma$ with $x\in W_{n}$.  Since by Remark \ref{r:shortestreps}, $\sigma$ is a minimal length representative in $\Wn \backslash {\cW}_{n+1},$ Proposition \ref{p:BBresult} and inversion imply, $$s_{i-1}x^{-1} > x^{-1}\iff xs_{i-1} > x\iff xs_{i-1}\sigma > x\sigma.$$  Since $x^{-1}=\sigma w^{-1},$ if $s_{i-1}x^{-1} > x^{-1},$ then  $m(s_{i-1})*_{L}(\sigma w^{-1})=s_{i-1}\sigma w^{-1},$ which equals $\sigma s_i w^{-1}$ by (\ref{eq:specialexplicit}).  But $xs_{i-1}\sigma = w s_i,$ so $w s_i > w,$ and hence $\sigma (m(s_i)*_{R} w)^{-1}=\sigma s_i w^{-1}.$  This verifies (\ref{eq:wequality}) in case $s_{i-1} x^{-1} > x^{-1},$ and the verification of (\ref{eq:wequality}) in case $s_{i-1}x^{-1} < x^{-1}$ is easier and left to the reader.
 
 \end{proof}

 There is one more monoid action on $B\backslash Q_{1,i+1}$.  Let $\alpha=\alpha_{1}\in\Pi_{\fg_{n+1}}$.  Then by Proposition \ref{p:Kmonoidaction}, $\alpha$ is real for the 
 $G$-orbit $Q_{1,i+1}$ if $i=1$, and $\alpha$ is complex unstable otherwise.  Thus, $m(s_{1})*Q_{1,i+1}=Q_{1,i+1}$, so we obtain a right monoid action by $\alpha$ on the orbits $B\backslash Q_{1,i+1}$ by (\ref{p:old4.7}).  The proof of the following result is analogous to the proof of Part (1) of Proposition \ref{p:unstablecorrespondence}.
 \begin{prop}\label{p:alpha1corres}
 Let $\alpha=\alpha_{1}\in\Pi_{\fg_{n+1}}$ and $\cZ\in B\backslash Q_{1,i+1}$ with $\cZ^{op}\in S_{i}\backslash\B_{n}$ as in Notation \ref{n:op}.  
 The orbit $m(s_{1})*\cZ\neq \cZ$ if and only if $\alpha$ is non-compact for $\cZ$.  Further, if $\Sh(\cZ)=(w, u^{*})$, then 
 \begin{equation}\label{eq:2ndcptleft}
 \Sh(m(s_{1})*\cZ)=(w, m(s_{1}^{*})*_{R} u^{*}),
 \end{equation} 
 where $s_{1}^{*}=(1,n+1)$ is the first simple reflection defined by the Borel subalgebra $\fb_{n+1}^{*}$ and given in (\ref{eq:starreflections}).
  
 \end{prop}

Using Proposition \ref{p:alpha1corres}, we can define a corresponding monoid action on $S_{i}\backslash \B_{n}$. This is the ``hidden" monoid action discussed in the Introduction.
Let $\cZ\in B\backslash Q_{1,i+1}$ with $\Sh(\cZ)=(w,u^{*})$, so that $\widetilde{\Sh}_{i}(\cZ^{op})=(\sigma w^{-1}, w_{i+1}^{*} {u^{*}}^{-1})$ by (\ref{eq:localSh}).  
We claim that 
\begin{equation}\label{eq:2ndcptlefttoright}
w_{i+1}^{*}(m(s_{1}^{*})*_{R} u^{*})^{-1}=m(s_{1})*_{L} (w_{i+1}^{*} {u^{*}}^{-1}).
\end{equation}

Indeed, decompose $u^{*}=yw_{i+1}^{*}$ with $y\in \Wn$.  Since by Remark \ref{r:shortestreps}, $w_{i+1}^{*}$ is a minimal 
length representative in $\Wn\backslash {\cW}_{n+1}$, it follows from Proposition \ref{p:BBresult} and inversion that
\begin{equation*}
\begin{split}
&yw_{i+1}^{*}s_{1}^{*}>yw_{i+1}^{*}\iff yw_{i+1}^{*}s_{1}^{*}{w_{i+1}^{*-1}} w_{i+1}^{*}>yw_{i+1}^{*}\iff  \\
&yw_{i+1}^{*}s_{1}^{*}{w_{i+1}^{*-1}}>y\iff w_{i+1}^{*}s_{1}^{*}{w_{i+1}^{*-1}}y^{-1}>y^{-1} \iff s_{1}y^{-1}>y^{-1}, 
\end{split}
\end{equation*}
where the last equivalence follows from the fact that $w_{i+1}^{*}s_{1}^{*}{w_{i+1}^{*-1}}=s_{1}$ by (\ref{eq:specialexplicit}).  This means $m(s_{1}^{*})*_{R} u^{*} > u^{*}$ if and only if $m(s_{1})*_{L} (w_{i+1}^{*} {u^{*}}^{-1}) > (w_{i+1}^{*} {u^{*}}^{-1}),$ and (\ref{eq:2ndcptlefttoright}) now follows from direct computation in the cases $m(s_{1}^{*})*_{R} u^{*}=u^{*}s_{1}^{*}$ and $m(s_{1}^{*})*_{R} u^{*}=u^{*}$ separately, using $w_{i+1}^{*}s_{1}^{*}{w_{i+1}^{*-1}}=s_{1}$ again.

It follows from Equations (\ref{eq:localSh}), (\ref{eq:2ndcptleft}), and (\ref{eq:2ndcptlefttoright}) that for $\cZ\in B\backslash Q_{1,i+1}$ with 
$\Sh(\cZ)=(w, u^{*})$ 
\begin{equation}\label{eq:firstrootop}
\widetilde{\Sh}_{i}((m(s_{1})*_{R}\cZ)^{op})=(\sigma w^{-1}, m(s_{1})*_{L}( w_{i+1}^{*}{u^{*}}^{-1})). 
\end{equation}
Let $Q\in S_{i}\backslash\B_{n}$ with $Q^{op}=\cZ\in B\backslash Q_{1,i+1}$ and with $\widetilde{\Sh}_{i}(Q)=(w,u)$.  We can use (\ref{eq:firstrootop}) to define a ``monoid" action on $Q$ via the first simple root 
$\alpha_{1}^{i}=\eps_{i}-\eps_{1}$ defined by the Borel subalgebra $\fb^{i}\subset\fg$ (see (\ref{eq:Eiflag})) as follows:   
\begin{equation}\label{eq:ms1defn}
\mbox{ define } m(s_{1}^{i})*_{L} Q\mbox{ to be the unique } S_{i}-\mbox{orbit with } m(s_{1}^{i})*_{L} Q=(m(s_{1})*_{R} \cZ)^{op}.
\end{equation}
Then it follows from (\ref{eq:firstrootop}) that 
\begin{equation}\label{eq:ms1shpair}
\widetilde{\Sh}_{i}(m(s_{1}^{i})*_{L} Q)=(w, m(s_{1})*_{L} u).
\end{equation}
By Equations (\ref{eq:opdim}) and (\ref{eq:basicmonoid}), we have $Q\neq m(s_{1}^{i})*_{L} Q\iff \dim(m(s_{1}^{i})*_{L} Q)=\dim(Q)+1$.  It follows from (\ref{e:sigmai}) and Definition \ref{d:istd} that $\sigma_{i}s_{1}\sigma_{i}^{-1}=(1,i)=s_{1}^{i}$ and $\Sh_{i}(m(s_{1}^{i})*Q)=(w, m(s_{1}^{i})*_{L} u^{i})$, where $u^{i}=\sigma_{i} u\sigma_{i}^{-1}.$  Since $s_{1}^{i}$ acts only on the second component of the Shareshian pair of $Q$, we say that $m(s_{1}^{i})*Q\neq Q$ if and only if $\alpha_{1}^{i}$ is non-compact for $Q$. 

\begin{rem} We note that the monoid action $m(s_1^{i})$ on $S_{i}\backslash \B_{n}$ given in (\ref{eq:ms1defn}) may be put into the framework of the monoid action of $\Pi_{\fg \oplus \fg}$ on $G$-orbits on a smooth variety with finitely many orbits, using a variant of the construction from Section \ref{ss:Gmonoid}.  However, to do so we must replace $\PR^{n-1}$ with a smooth Schubert variety.  We hope to return to this issue in future work. 
\end{rem}

 \begin{exam}\label{ex:Q13orbits}
 
 Below we give the Bruhat graph for the $B_{3}$-orbits contained in the $GL(3)$-orbit $Q_{1,3}$ in flag variety $\B_{4}$ of $\fgl(4)$.  In this example, $B_{3}$-orbits are labelled by their Shareshian pairs.  In the diagram below, the top row consists of the $B_{3}$-orbits of minimal dimension in $Q_{1,3}$, and the dimension of the orbits increases by $1$ as we descend from row to row.   If $Q$ and $Q^{\prime}$ are two $B_{3}$-orbits, we indicate that $Q^{\prime} \subset \overline{Q}$ by exhibiting a sequence of downward lines from $Q^{\prime}$ to $Q.$  We also indicate the monoid actions. 
 For the $GL(3)$-orbit $Q_{1,3}$, there are no compact imaginary roots, so that the set $\Pi_{4,cpt}$ of Equation (\ref{eq:compactroots}) is empty.  Therefore, the only right monoid actions we have are the ones given in Propositions \ref{p:unstablecorrespondence} and \ref{p:alpha1corres}.  We denote the standard simple roots of $\Pi_{\fg}=\{\alpha,\, \beta,\gamma\}$ and the corresponding simple reflections by $s=s_{\alpha},\, t=s_{\beta}$, and $u=s_{\gamma}$ respectively.  
 For the simple roots $\Pi_{\fg}^{*}=\{\alpha^{*},\, \beta^{*},\, \gamma^{*}\}$ defined with respect to the Borel subalgebra $\fb^{*}_{4}$, we label the corresponding simple reflections as $s^{*}=(1,4)$, $t^{*}=(1,2)=s$, $u^{*}=(2,3)=t$.  Note that ${\cW}_{3}=\langle t^{*},\, u^{*}\rangle$.  Recall the element $w_{3}^{*}=s^{*}t^{*}$ of Equation (\ref{eq:specialWelts}).  For $n=3$,  the cycle $\sigma\in {\cW}_{4}$ is $\sigma=(4,\,3,\, 2,\, 1)$.  Recall also that for any $Q\in B\backslash Q_{1,3}$, $\Sh(Q)\in {\cW}_{3}\sigma \times {\cW}_{3}w_{3}^{*}$ by Theorem \ref{thm:ShandKorbits}.  In this example, we use a solid line to denote right monoid actions as given by Propositions \ref{p:unstablecorrespondence} and \ref{p:alpha1corres} and a dashed line to denote a left monoid action by simple roots $\Pi_{\fgl(3)}$ as in (\ref{eq:BQleftmonoid}).   A red line denotes a non-compact root and a blue line denotes a complex stable root.  A green line
indicates a closure relation that is not obtained from a monoid action.

 
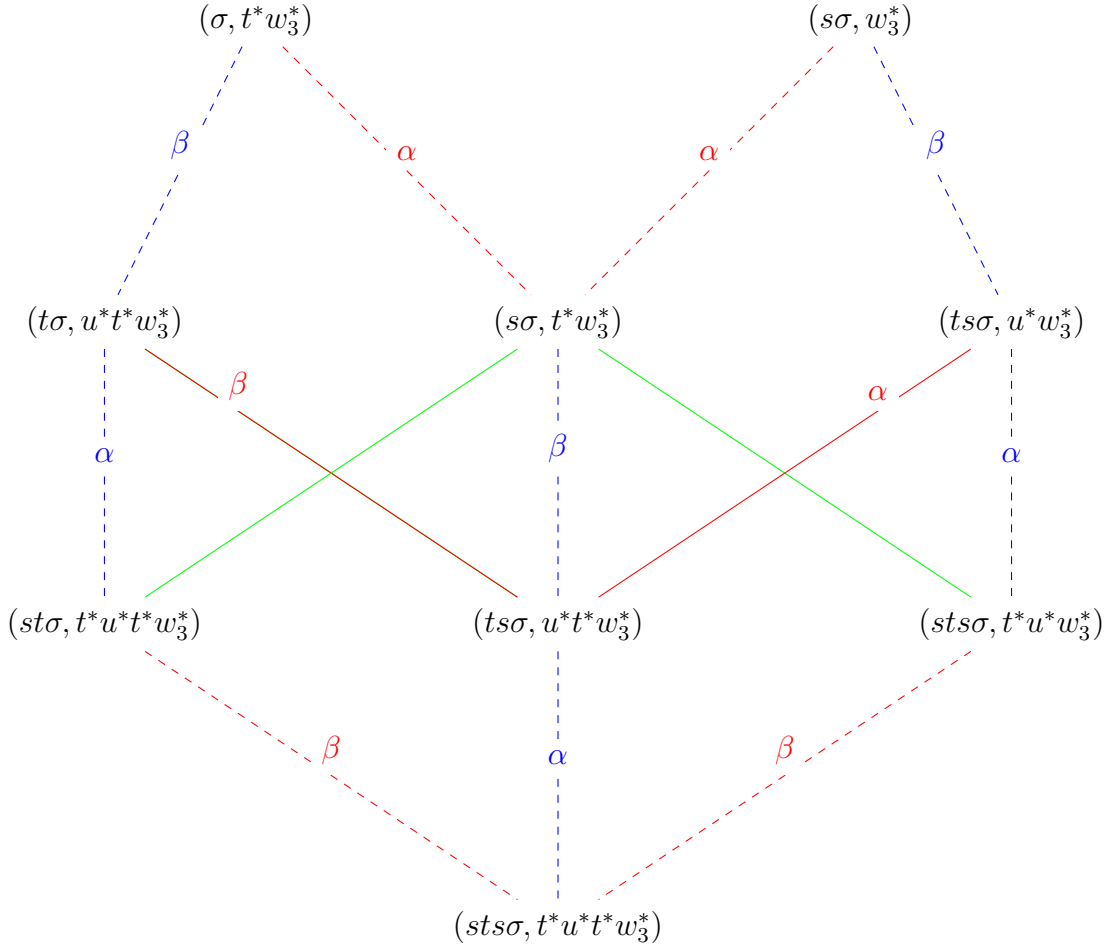
\begin{center}
 \begin{tikzpicture}  
 [scale=2.0,auto=center,every node/.style={rectangle,fill=white!20}] 
\node (a1) at  (-1,5) {$(\sigma, t^{*}w_{3}^{*})$};
\node (a2) at (3,5) {$(s\sigma, w_{3}^{*}) $};
\node (a3) at (-2,3) {$(t\sigma,u^{*}t^{*} w_{3}^{*})$};
\node (a4) at (1,3) {$(s\sigma,t^{*}w_{3}^{*})$};
\node (a5) at (4,3) {$(ts\sigma,u^{*}w_{3}^{*})$};
\node (a6) at (-2,1) {$(st\sigma, t^{*}u^{*}t^{*}w_{3}^{*})$};
\node (a7) at (1,1) {$(ts\sigma,u^{*}t^{*}w_{3}^{*})$};
\node (a8) at (4,1) {$(sts\sigma,t^{*}u^{*}w_{3}^{*})$} ;
\node (a9) at (1,-1) {$(sts\sigma,t^{*}u^{*}t^{*}w_{3}^{*})$};
\draw [blue] [dashed] (a1) -- (a3) node[midway, above] {$\beta$}; 
  \draw [red] [dashed] (a1) -- (a4) node[midway, above] {$\alpha$};  
 \draw [red] [dashed] (a2) -- (a4) node[midway, above] {$\alpha$};  
 \draw [blue] [dashed] (a2)--(a5) node[midway, above] {$\beta$};
 \draw [blue] [dashed](a3)--(a6) node[midway, above] {$\alpha$};
  \draw [green] (a3)--(a7) ;
  \draw[blue][dashed] (a4)--(a7) node[midway, above] {$\beta$};
  \draw[red] (a3)--(a7)  node[near start, above] {$\beta$};
  \draw[green] (a4)--(a6);
  \draw[green] (a4)--(a8);
  \draw [red] (a5)--(a7) node[near start, above] {$\alpha$};
  \draw[blue][dashed] (a5)--(a8) node[midway, above] {$\alpha$};
  \draw[red][dashed]  (a6)--(a9) node[midway, above] {$\beta$};
    \draw[blue] [dashed] (a7)--(a9) node[midway, above] {$\alpha$};
       \draw[red] [dashed] (a8)--(a9) node[midway, above] {$\beta$}; 
\end{tikzpicture}
\end{center}

	It is instructive to compare this graph with the Bruhat graph for the  $S_{2}$-orbits on $\B_{3}$ given in Example 5.9 of \cite{Siorbits}, which is equivalent to this graph by Theorem \ref{thm:firstbig} and Remark \ref{r:corres}.  Note also that 
the left monoid actions in this graph correspond to the right monoid actions of Example 5.9 of \emph{loc. cit.}, illustrating Equation (\ref{eq:leftinter}) of Theorem \ref{thm:monoidcorresp}.  Also, the right monoid action by the root $\beta$ corresponds to the left monoid action of the root $\alpha$ in Example 5.9 of \emph{loc. cit.}, illustrating Proposition \ref{p:unstablecorrespondence}.  Finally, the right monoid action by $\alpha$ given in Proposition \ref{p:alpha1corres} corresponds to the left monoid action by $\alpha^{2}$ in Example 5.9 of \emph{loc. cit.} given in (\ref{eq:ms1defn}) and (\ref{eq:ms1shpair}).  Note also that this poset may be identified as the subgraph of either of the equivalent graphs in Example \ref{ex:3by3} given by the blue edges.

\end{exam}

 \section{The Orbit Correspondence and Closure Relations}\label{s:global}
 \subsection{The Orbit Correspondence in Terms of Decorated Permutations}
 In this section, we recall Magyar's parameterization of $B\backslash (\B_{n}\times\mathbb{P}^{n-1})$ in terms of decorated permutations from Magyar's paper \cite{Magyar}.  We then use our results to show that the order relation on decorated permutations studied by Magyar can be described in terms of the well-understood Bruhat order on $\cW_{n+1}.$

 
 \begin{dfn}\label{d:marked}
We define a \emph{decorated permutation} of $\{1,\dots, n\}$ to be a pair $(w,\Delta)$ where $w\in\Wn$ and $\Delta=\{j_{1},\dots, j_{k}\}$ is a nonempty subset of $\{1,\dots,n\}$ which is a decreasing sequence for $w^{-1}$, i.e., $j_1 < \dots < j_{k}$ but $w^{-1}(j_{k})<\dots< w^{-1}(j_{1})$.
\end{dfn}

We denote the set of all decorated permutations of $\{1,\dots, n\}$ by $\widehat{\cW}_{n}$.  Given  $(w,\Delta) \in \widehat{\cW}_{n}$ with 
$\Delta=\{j_{1},\dots, j_{k}\}$ as above,
 we consider the $B$-orbit 
\[
{\mathcal{O}}_{(w,\Delta)}:= B\cdot (w(\fb), [v_{\Delta}]), \ \text{ where } v_{\Delta}=e_{j_1}+\dots + e_{j_k}.
\]

\begin{thm}\label{thm:magyar}\cite{Magyar}
Each $B$-orbit on $\B_{n}\times\mathbb{P}^{n-1}$ is of the form ${\mathcal{O}}_{(w,\Delta)}$ for a unique decorated permutation $(w,\Delta).$
\end{thm}
 
 \noindent See Theorem 3.8 in \cite{Siorbits} for a simple proof of Theorem \ref{thm:magyar}.

 \begin{nota} \label{n:marked}  The orbit ${\mathcal O}_{(w,\Delta)}\in B\backslash(\B_{n}\times\mathcal{O}_{i})$ if and only if $j_{k}=i$.  In this case, we refer to the decorated permutation $(w,\Delta)$ as an $i$-decorated permutation.  We denote the set of $i$-decorated permutations by $\widehat{\cW}_{n}(i)$.  For $(w,\Delta)\in \widehat{\cW}_{n}(i)$, we denote by $Q_{(w,\Delta)}$ the  $S_{i}$-orbit on $\B_{n}$ corresponding to the $B$-orbit $\mathcal{O}_{(w,\Delta)}$ given by the rule in (\ref{eq:Sicorres}).  Explicitly, this means
\begin{equation}\label{eq:markedcorres}
(w,\Delta)\leftrightarrow \calO_{(w,\Delta)}=B\cdot (w(\fb), [v_{\Delta}])\leftrightarrow Q_{(w,\Delta)}= S_{i} b\cdot w(\fb) \in S_{i}\backslash\B_{n},
\end{equation}
where $b\in B$ is such that $b\cdot [v_{\Delta}]=[e_{i}]$.

\end{nota}

\begin{rem}\label{rem:iShform}
Suppose $(w,\Delta)\in \widehat{\cW}_{n}(i)$ is an $i$-decorated permutation with $\Delta=\{j_{1}<j_{2}<\dots< j_{k-1}<i\}$ and $Q=Q_{(w,\Delta)}$ as in Equation (\ref{eq:markedcorres}).  Then by Remark 3.21 of \cite{Siorbits}, $\Sh_{i}(Q)=(w, \nu_{\Delta} w\sigma_{i}^{-1}),$ where 
$\nu_{\Delta}=(i, \, j_{k-1},\dots, j_{2},\, j_{1})$ and $\sigma_{i}=(i,\, i-1,\, \dots, 2,\, 1)$ are cycles of $\Wn.$  Conversely, if $Q\in S_{i}\backslash\B_{n}$ with $\Sh_{i}(Q)=(w,u^{i})$, then the permutation $\nu_{\Delta}:=u^{i}\sigma_{i}w^{-1}$ is a cycle, and $Q=Q_{(w,\Delta)}$ with $\Delta=\mbox{support }({\nu_{\Delta}})$.
\end{rem}
 
 \begin{nota}\label{nota:biggerop}
Let $(w,\Delta)\in\widehat{\cW}_{n}$ with the corresponding $B$-orbit $\mathcal{O}_{(w,\Delta)}$ on $\B_{n}\times\mathbb{P}^{n-1}$ as in Theorem \ref{thm:magyar}.  We denote  by $\mathcal{O}_{(w,\Delta)}^{op}$ the orbit in $B\backslash\mathfrak{X}$ given by the second correspondence in Equation (\ref{eq:secondcorresp}). 
\end{nota}
 

\begin{thm}\label{thm:globalSh}
Let $\mathcal{O}_{(w,\Delta)} $ be the $B$-orbit on $\B_{n}\times \mathbb{P}^{n-1}$ corresponding to the decorated permutation $(w,\Delta)\in\widehat{W}_{n}$ as in Theorem \ref{thm:magyar}.  Suppose that $\Delta=\{j_{1}<j_{2}<\dots<j_{k}\}$ and consider the $k+1$-cycle $\kappa_{\Delta}:=(j_{1}, j_{2},\dots, j_{k}, n+1)\in \cW_{n+1}$.  Let $\mathcal{O}_{(w,\Delta)}^{op}$ be the $B$-orbit on $\mathfrak{X}$ given by the correspondence in Equation (\ref{eq:secondcorresp}) as in Notation \ref{nota:biggerop}. Then 
\begin{equation}\label{eq:globalSh}
\Sh(\mathcal{O}_{(w,\Delta)}^{op})=(w^{-1}\sigma, w^{-1}\kappa_{\Delta}).  
\end{equation}
\end{thm}
\begin{proof}
By Notation \ref{n:marked}, $\mathcal{O}_{(w,\Delta)}\in B\backslash (\B_{n}\times \mathcal{O}_{i})$, where $i=j_{k}$.  By  Equation (\ref{eq:markedcorres}), $\mathcal{O}_{(w,\Delta)}$ corresponds to an $S_{i}-$orbit $Q_{(w,\Delta)}$ on $\B_{n}.$  By  Remark \ref{rem:iShform}, $\Sh_{i}(Q_{(w,\Delta)})=(w, \nu_{\Delta} w\sigma_{i}^{-1}).$
  By Definition \ref{d:istd}, 
 the standardized $i$-Shareshian pair 
\begin{equation}\label{eq:firsttwisted}
\widetilde{\Sh}_{i}(Q_{(w,\Delta)})=(w, \sigma_{i}^{-1}\nu_{\Delta} w).
\end{equation}
 
Let $\Sh(\mathcal{O}_{(w,\Delta)}^{op})=(u,z)$.  To prove the theorem, we must show that the Shareshian pair $(u,z)$ is given by the right hand side of Equation (\ref{eq:globalSh}).  By Theorem \ref{thm:localShareshian}, 
 \begin{equation}\label{eq:secondtwisted}
 \widetilde{\Sh}_{i}(Q_{(w,\Delta)})=(\sigma u^{-1}, w_{i+1}^{*} z^{-1}).
 \end{equation}
  Using (\ref{eq:firsttwisted}), we deduce that $u=w^{-1}\sigma$ and  $z=w^{-1}\nu_{\Delta}^{-1} \sigma_{i} w_{i+1}^{*}$.  
By Equations (\ref{e:sigmai}) and (\ref{eq:specialexplicit}), we have $\nu_{\Delta}^{-1} \sigma_{i} w_{i+1}^{*}=\kappa_{\Delta}.$
Thus $z=w^{-1}\kappa_{\Delta}$, which completes the proof.
\end{proof}


\begin{rem} For completeness, we state several formulas for $\dim(\mathcal{O}_{(w,\Delta)}).$  The first was known by Magyar and is given directly in terms of $w$ and $\Delta.$   For a decorated permutation $(w,\Delta)$ with $\Delta = \{ j_1 < \dots < j_k\},$ let 
\[
m(w,\Delta)=|\{ r \not\in \{j_1, \dots, j_k \} : \exists j_s \in \{ j_1, \dots, j_k \} \mbox { with } r < j_s \mbox{ and } w^{-1}(r) < w^{-1}(j_s)\}|.
\]
Then 
\begin{equation} \label{e:dimfirst}
\dim(\mathcal{O}_{(w,\Delta)})=\ell(w)+m(w,\Delta) + |\Delta| - 1.
\end{equation}
This may be proved by projecting $\mathcal{O}_{(w,\Delta)}$ to $\B_n$, and computing the dimension of the fiber, which is identified with $(B\cap w(B))\cdot [v_{\Delta}].$ For the second formula, recall that if a $B$-orbit ${\cZ}$ has Shareshian pair $(y,u^{*}),$ then by Definition 4.13 of \cite{Shpairs}, its standardized Shareshian pair $\widetilde{\Sh}(\cZ)$ is $(y,\sigma^{-1}u^{*}\sigma).$  By Corollary 4.14 of loc. cit., if the standardized Shareshian pair of  $\mathcal{O}_{(w,\Delta)}^{op}$ is $(y,v),$ then $\dim(\mathcal{O}_{(w,\Delta)}^{op})=\frac{\ell(y)+\ell(v)-\ell(vy^{-1})}{2}.$  By applying Equations (\ref{eq:opdim}) and (\ref{e:dimcorr}) with $i=j_k,$
\[
\dim(\mathcal{O}_{(w,\Delta)})=\dim(Q_{(w,\Delta)})+i-1=\dim(\mathcal{O}_{(w,\Delta)}^{op})-1.
\]
Hence,
\begin{equation}\label{e:second}
\dim(\mathcal{O}_{(w,\Delta)})=\frac{\ell(y)+\ell(v)-\ell(vy^{-1})}{2} - 1, \mbox{ if } \widetilde{\Sh}(\mathcal{O}_{(w,\Delta)}^{op})=(y,v).
\end{equation}
To apply (\ref{e:second}), note that by Theorem \ref{thm:globalSh}, $\widetilde{\Sh}(\mathcal{O}_{(w,\Delta)}^{op})=(w^{-1}\sigma, \sigma^{-1}w^{-1}\kappa_{\Delta}\sigma).$
\end{rem}

The correspondence in (\ref{eq:secondcorresp}) can be used to construct an embedding of orbit sets 
\begin{equation}\label{eq:orbitsembedding}
G\backslash (\B_{n}\times\B_{n}\times \mathbb{P}^{n-1})\hookrightarrow G_{n+1}\backslash (\B_{n+1}\times\B_{n+1}\times \mathbb{P}^{n}).
\end{equation}
as follows.  Equation (\ref{eq:secondcorresp}) gives us a bijection 
$$
G\backslash (\B_{n}\times\B_{n}\times \mathbb{P}^{n-1})\leftrightarrow B\backslash\mathfrak{X} \subset \B_{n+1},
$$
where $\mathfrak{X}$ is the open subvariety of $\B_{n+1}$ defined in (\ref{eq:frakX}).
Now $B$-orbits on $\B_{n+1}$ correspond to $B_{n+1}$-orbits on $\B_{n+1}\times\mathcal{O}_{n+1}$, which in turn correspond to a subset of $G_{n+1}$-orbits on the triple product $\B_{n+1}\times\B_{n+1}\times\mathbb{P}^{n}$.  So we can embed $B\backslash \mathfrak{X}$ into  $G_{n+1}\backslash (\B_{n+1}\times\B_{n+1}\times \mathbb{P}^{n})$ to obtain the embedding of (\ref{eq:orbitsembedding}).

We now describe the map on decorated permutations induced by (\ref{eq:orbitsembedding}).  For $J=\{ j_1 < j_2 < \cdots < j_k \}$ an increasing sequence of integers in $\{ 1, \dots, n \},$ let $[J:n+1]:=\{ j_1 < j_2 < \cdots < j_k < n+1 \}$ be the corresponding sequence in $\{ 1, \dots, n+1 \}.$
\begin{prop}\label{p:marked}
Let $(w,\Delta)\in\widehat{\cW}_{n}$ and consider the $n+1$-cycle $\sigma=(n+1, n, n-1,\dots, 2, 1)$ of $\cW_{n+1}.$    Let $\phi: \widehat{\cW}_{n}\hookrightarrow \widehat{\cW}_{n+1}$ be the embedding on decorated permutations corresponding to the orbit embedding in Equation (\ref{eq:orbitsembedding}).  Then 
\begin{equation}\label{eq:permutation}
\phi((w,\Delta))=(w^{-1}\sigma, [w^{-1}(\Delta):n+1]).
\end{equation}
In particular, 
\begin{equation}\label{eq:opdecorated}
\mathcal{O}_{(w,\Delta)}^{op}=\mathcal{O}_{\phi(w,\Delta)}.
\end{equation}
\end{prop}
\begin{proof}
Let $\mathcal{O}_{(w,\Delta)}$ be the $B$-orbit on $\B_{n}\times \mathbb{P}^{n-1}$ corresponding to the decorated permutation $(w, \Delta),$ and let $\mathcal{O}_{(w,\Delta)}^{op}=\mathcal{O}_{(u,\tilde{\Delta})}$ for some $(u,\tilde{\Delta}) \in\widehat{\cW}_{n+1}$.  
 Then by Theorem \ref{thm:globalSh}, $\Sh(\mathcal{O}_{(w,\Delta)}^{op})=(w^{-1}\sigma, w^{-1}\kappa_{\Delta}),$ so by Remark  \ref{rem:iShform}, $u=w^{-1}\sigma$.  By Remark \ref{rem:iShform} again with $B=S_{n+1}$ (up to centre) we obtain  
\[
\nu_{\tilde{\Delta}}= w^{-1}\kappa_{\Delta} \sigma (w^{-1}\sigma)^{-1} = 
w^{-1}\kappa_{\Delta} w=(n+1, w^{-1}(j_{1}), \dots, w^{-1}(j_{k})).
\]
  The result follows, noting that $(w^{-1}\sigma, [w^{-1}(\Delta):n+1])$ is a decorated permutation.
\end{proof} 

\begin{thm}\label{thm:globalclosure}
Let $(w,\Delta),\, (y,\Gamma)\in\widehat{W}_{n}$, and let 
$\mathcal{O}_{(w,\Delta)}$ and $\mathcal{O}_{(y,\Gamma)}$ be the associated $B$-orbits on $\B_{n}\times\mathbb{P}^{n-1}$.  Let $\mathcal{O}_{(w,\Delta)}^{op}$ and 
$\mathcal{O}_{(y,\Gamma)}^{op}$ be the corresponding $B$-orbits on $\mathfrak{X}$.  
Then 
\begin{equation}\label{eq:preserveclosure}
\mathcal{O}_{(w,\Delta)}\subset\overline{\mathcal{O}_{(y,\Gamma)}}\Leftrightarrow \mathcal{O}_{(w,\Delta)}^{op}\subset\overline{\mathcal{O}_{(y,\Gamma)}^{op}}.
\end{equation}
\end{thm}
We will prove this Theorem using Magyar's combinatorial characterization of the closure ordering on $B\backslash(\B_{n}\times\mathbb{P}^{n-1})$ using $\widehat{\cW}_{n}$ from \cite{Magyar}.
We now recall Magyar's ordering $\unlhd$ on $\widehat{\cW}_{n}.$  Consider the subspaces $\mathcal{E}_{p}=\mbox{span}\{e_{1},\dots, e_{p}\}$ for $p=0, \dots, n$ so ${\mathcal{E}}_{0}=0.$  Let 
$(w,\Delta)\in\widehat{\cW}_{n}$ and associate to $(w,\Delta)$ the following statistics.
\begin{equation}\label{eq:firststat}
\mbox{For } p,\, q\in\{0,\dots, n\}\mbox{ define } r_{p,q}(w):=\dim(\mathcal{E}_{p}\cap w(\mathcal{E}_{q}))=|\{1,\dots, p\}\cap \{w(1),\dots, w(q)\}|.
\end{equation}    

\begin{equation}\label{eq:littledelta}
\mbox{ For } p,\,q\in\{0,1,\dots, n\}, \mbox{ define } \delta_{p, q} (w,\Delta):=\left\{\begin{array}{ccc} 1 & \mbox { if for each } \ell\in\Delta, & \ell\leq p \mbox{ or } w^{-1}(\ell)\leq q \\
0 &\mbox{ otherwise } & \end{array}\right\}_{.}
\end{equation}
Define:
\begin{equation}\label{eq:secondstat}
\overline{r}_{p,q}(w,\Delta):=r_{p,q}(w)+\delta_{p,q}(w,\Delta)\mbox{ for any } p,\, q\in\{0,1,\dots, n\}.
\end{equation}
Define an ordering on $\widehat{\cW}_{n}$ by declaring 
\begin{equation}\label{eq:Morder}
\begin{split}
&(w,\Delta)\unlhd (y,\Gamma)\Leftrightarrow \\
&(1)\; r_{p,q}(w)\geq r_{p,q}(y) \; \forall\, p,\, q\in\{1,\dots, n\} \mbox{ and } \\
&(2)\; \overline{r}_{p,q}(w,\Delta)\geq \overline{r}_{p,q}(y,\Gamma)\;\forall\; p,\, q\in\{0,1, \dots, n\}.
\end{split}
\end{equation} 

\begin{rem}\label{r:isBruhat}
Note that condition (1) in (\ref{eq:Morder}) is equivalent to the condition
that $w\leq y$ in the Bruhat order on $\Wn$ by Proposition 7 in
Section 10.5 of \cite{Fulton}.  
\end{rem}

\begin{thm}\label{thm:Magyarclosure}\emph{[Theorem 2.2 of \cite{Magyar}]}
Let $(w,\Delta),\, (y,\Gamma)\in \widehat{\cW}_{n}$ and let $\mathcal{O}_{(w,\Delta)}$ and $\mathcal{O}_{(y,\Gamma)}$ be the corresponding $B$-orbits on $\B_{n}\times\mathbb{P}^{n-1}.$ Then 
$$
\mathcal{O}_{(w,\Delta)}\subset\overline{\mathcal{O}_{(y,\Gamma)}}\Leftrightarrow (w,\Delta)\unlhd (y,\Gamma),
$$
where $\unlhd$ is the order on $\widehat{\cW}_{n}$ defined in (\ref{eq:Morder}).
\end{thm}
\begin{prop}\label{p:orderpreserving}
The map $\phi: \widehat{\cW}_{n}\hookrightarrow\widehat{\cW}_{n+1}$ given in Equation (\ref{eq:permutation}) preserves the ordering on decorated permutations defined in (\ref{eq:Morder}), i.e., 
\begin{equation}\label{eq:orderpreserve}
(w,\Delta)\unlhd (y,\Gamma)\Leftrightarrow \phi((w,\Delta))\unlhd \phi((y,\Gamma)).
\end{equation}
\end{prop}
\noindent 
  To prove Proposition \ref{p:orderpreserving},  the following observation is crucial.   
\begin{lem}\label{l:compare}
Let $p\in\{0,\dots, n\}$ and $q\in\{1,\dots, n+1\}$.  
\begin{equation}\label{eq:shift}
\overline{r}_{p,q}(\phi(w,\Delta))=\overline{r}_{q-1,p} (w,\Delta).
\end{equation}
\end{lem}
\begin{proof}
 For $1\le k \le n+1,$ let ${\mathcal{E}}^{n+1}_{k}$ be the span of $e_1, \dots, e_k$ in $\C^{n+1}.$  Recall from Proposition \ref{p:marked} the element $\sigma^{-1} w \in {\cW}_{n+1}.$  We claim that
\begin{equation}\label{eq:rpqshift}
r_{p,q}(w)=r_{p+1, q}(\sigma^{-1} w)\mbox{ for } p, \, q\in\{0,\dots, n\},
\end{equation}
where as above, $r_{p+1, q}(\sigma^{-1}w)=\dim(\mathcal{E}^{n+1}_{p+1}\cap \sigma^{-1}w(\mathcal{E}^{n+1}_{q})).$  Since $w\in \Wn$, for any index $\ell\in\{1,\dots, q\}$, 
$1\leq w(\ell)\leq n$, so that $\sigma^{-1}w(\ell)=w(\ell)+1$.  
Therefore, for $p\geq 1$, we have a bijection between the sets $\{1,\dots, p\}\cap \{w(1),\dots, w(q)\}$ and $\{1,\dots, p+1\}\cap \sigma^{-1}w (\{1,\dots, q\})$ given by $x\mapsto x+1,$ which verifies (\ref{eq:rpqshift}) when $p \ge 1$ and $q \ge 1.$  In case $p=0,$ then $r_{0,q}(w)=0$, and since $\sigma^{-1}w(\ell)\geq 2$ for all $\ell\in\{1,\dots, n\}$, $r_{1,q}(\sigma^{-1} w)=0.$  For $q=0$, both sides of (\ref{eq:rpqshift}) are zero and the claim is trivial, so that (\ref{eq:rpqshift}) is verified in all cases.   

Suppose now that $p\in\{0,\dots, n\}$ and $q\in\{1,\dots, n+1\}.$  We assert that 
\begin{equation} \label{eq:deltashift}
\delta_{p,q}(\phi(w,\Delta))=\delta_{p,q}(w^{-1}\sigma, [w^{-1}(\Delta): n+1])=\delta_{q-1, p}(w,\Delta).
\end{equation}
The first equality is from Proposition \ref{p:marked}, and for the second equality, by Equation (\ref{eq:littledelta}), we have 
\begin{equation}\label{eq:firstdelta}
\delta_{p,q}(w^{-1}\sigma, [w^{-1}(\Delta): n+1])=\left\{\begin{array}{ccc} 1 & \mbox { if for each } \ell\in\Delta, & w^{-1}(\ell)\leq p \mbox{ or } \sigma^{-1}w(w^{-1}(\ell))\leq q \\
&\;\;\;\;\;\;\;\;\;\;\;\;\;\;\;\;\;\;\;  & \mbox{ and } n+1\leq p\mbox{ or } \sigma^{-1}w(n+1)\leq q\\
0 &\mbox{ otherwise } & \end{array}\right\}_{.}
\end{equation}
Because $\Delta\subset\{1,\dots, n\}$, $w\in\Wn,$ and $\sigma^{-1}(n+1)=1,$  (\ref{eq:firstdelta}) becomes
\begin{equation}\label{eq:seconddelta}
\delta_{p,q}(w^{-1}\sigma, [w^{-1}(\Delta): n+1])=\left\{\begin{array}{ccc} 1 & \mbox { if for each } \ell\in\Delta, & w^{-1}(\ell)\leq p \mbox{ or } \ell\leq q-1 \\
&\;\;\;\;\;\;\;\;\;\;\;\;\;\;\;\;\;\;\;   & \mbox{ and } 1\leq q\;\;\;\;\;\;\;\;\;\;\;\;\;\;\;\;\;\;\;\;\;\;\;\;\;\\
0 &\mbox{ otherwise } & \end{array}\right\}_{.}
\end{equation}
Since $q\geq 1$, the second condition is always satisfied and the second equality in (\ref{eq:deltashift}) now follows from the definition of $\delta_{p,q}$ in (\ref{eq:littledelta}). 
Note that for any $y\in \cW_{n+1}$ and any indices $p,\, q\in \{0,\dots, n+1\}$, it follows from the definition of $r_{p,q}(y)$ in (\ref{eq:firststat}) 
 that 
\begin{equation}\label{eq:rpqinv}
r_{p,q}(y)=r_{q,p}(y^{-1}).
\end{equation}  
For all $p, q$ as above,  
\begin{equation*}
\begin{split}
 \overline{r}_{p,q}(w^{-1}\sigma, [w^{-1}(\Delta): n+1])&=r_{p,q}(w^{-1}\sigma)+\delta_{p,q}(w^{-1}\sigma,  [w^{-1}(\Delta): n+1])\\
&=r_{q,p}(\sigma^{-1}w)+\delta_{p,q}(w^{-1}\sigma,  [w^{-1}(\Delta): n+1]) \mbox{ by } (\ref{eq:rpqinv})\\
&=r_{q-1, p}(w)+\delta_{q-1,p}(w,\Delta)\mbox{ by } (\ref{eq:rpqshift}) \mbox{ and }(\ref{eq:deltashift})\\
&=\overline{r}_{q-1, p}(w, \Delta),
\end{split}
\end{equation*}
which by Proposition \ref{p:marked}, establishes Equation (\ref{eq:shift}) and the Lemma.
\end{proof}
\begin{proof}[Proof of Proposition \ref{p:orderpreserving}]
Let $(w,\Delta),\, (y,\Gamma)\in\widehat{\cW}_{n}$.  By definition
$(w,\Delta) \unlhd (y,\Gamma)$ if and only if both conditions (1) and (2) of (\ref{eq:Morder}) are satisfied and similarly for $\phi(w,\Delta)$ and $\phi(y,\Gamma)$.  By Remark \ref{r:isBruhat}, the condition (1) in (\ref{eq:Morder}) is equivalent to $w\leq y$ in the usual Bruhat ordering on $\Wn$. But by Proposition \ref{p:BBresult} and invariance of the Bruhat ordering under inversion,
\begin{equation}\label{eq:firstequiv}
w\leq y\Leftrightarrow w^{-1} \leq y^{-1} \Leftrightarrow w^{-1}\sigma\leq y^{-1}\sigma.
\end{equation} 
  By the formula for $\phi$ in (\ref{eq:permutation}), Equation (\ref{eq:firstequiv}) states that condition (1) of (\ref{eq:Morder}) is satisfied for the marked permutations $(w,\Delta)$ and $(y,\Gamma)$ if and only if its analogue is satisfied for the marked permutations $\phi(w,\Delta)$ and $\phi(y,\Gamma)$.  

 We now consider the condition (2) in (\ref{eq:Morder}): if $(w,\Delta)\unlhd (y,\Gamma)$, then $\overline{r}_{p,q}(w,\Delta)\geq \overline{r}_{p,q}(y,\Gamma)\; \forall\, p,\,q\in\{0,\dots, n\}.$  But by Lemma \ref{l:compare}, we have 
\begin{equation}\label{eq:overlinecompare}
\overline{r}_{p,q}(w,\Delta)\geq \overline{r}_{p,q}(y,\Gamma)\Leftrightarrow \overline{r}_{q,p+1}(\phi(w,\Delta))\geq \overline{r}_{q,p+1}(\phi(y,\Gamma))\; \forall \, p,\,q \in\{0, 1,\dots, n\}. 
\end{equation}
We need to consider the cases $\overline{r}_{n+1, q}(\phi(w,\Delta))$ and 
$\overline{r}_{p,0}(\phi(w,\Delta))$ for $p \le n$ separately.  Using Equation (\ref{eq:seconddelta}), we 
see that $\delta_{n+1,q}(\phi(w,\Delta))=1$ for all $q\in\{1,\dots, n+1\}$, and that 
$\delta_{p,0}(\phi(w,\Delta))=0$ for any $p$.     Since $r_{n+1,q}(w^{-1}\sigma)=q$ for any $q\geq 1$, and $r_{p,0}(w^{-1}\sigma)=0$ for any $p$, it follows from the definition of 
$\overline{r}_{p,q}(w,\Delta)$ in (\ref{eq:secondstat}) that for any two decorated permutations $(w,\Delta)$ and $(y,\Gamma)$ 
\begin{equation}\label{eq:extremecase}
\overline{r}_{n+1,q}(\phi(w,\Delta))=\overline{r}_{n+1, q}(\phi(y,\Gamma))\;\forall\, q\in \{1,\dots, n+1\},
\end{equation}
and 
\begin{equation}\label{eq:seccase}
\overline{r}_{p,0}(\phi(w,\Delta))=\overline{r}_{p, 0}(\phi(y,\Gamma))\;\forall\, p\in \{0,1,\dots, n+1\},
\end{equation}
It follows from Equations (\ref{eq:overlinecompare})-(\ref{eq:seccase}) that 
$$
\overline{r}_{p,q}(w,\Delta)\geq \overline{r}_{p,q}(y,\Gamma)\; \forall\, p,\,q \in\{0,\dots, n\}\Leftrightarrow \overline{r}_{p,q}(\phi(w,\Delta))\geq \overline{r}_{p,q}(\phi(y,\Gamma))\; \forall\, p,\,q \in\{0,\dots, n+1\}.
$$
Thus, the condition (2) of (\ref{eq:Morder}) is satisfied for the decorated permutations 
$(w,\Delta)$ and $(y,\Gamma)$ if and only if it is satisfied for the decorated permutations
$\phi(w,\Delta)$ and $\phi(y,\Gamma)$.  Equation (\ref{eq:orderpreserve}) now follows, and the Proposition is proven.
\end{proof}

We now prove our main result.

\begin{proof}[Proof of Theorem \ref{thm:globalclosure}]
By Theorem \ref{thm:Magyarclosure}, $\mathcal{O}_{(w,\Delta)}\subset\overline{\mathcal{O}_{(y,\Gamma)}}$ if and only if $(w,\Delta)\unlhd (y,\Gamma),$ which is equivalent to $\phi(w,\Delta)\unlhd \phi(y,\Gamma)$ by Proposition \ref{p:orderpreserving}.  This last condition is equivalent to $\mathcal{O}_{\phi(w,\Delta)} \subset \overline{\mathcal{O}_{\phi(y,\Gamma)}}$ using Theorem \ref{thm:Magyarclosure} again, which in turn is equivalent to $\mathcal{O}_{(w,\Delta)}^{op}\subset\overline{\mathcal{O}_{(y,\Gamma)}^{op}}$ by (\ref{eq:opdecorated}).
\end{proof}

\begin{cor}\label{c:canuseSh}
Let $\mathcal{O}_{(w,\Delta)}$ and $\mathcal{O}_{(y,\Gamma)}$ be $B$-orbits on 
$\B_{n}\times\mathbb{P}^{n-1}$ corresponding to the decorated permutations $(w,\Delta)$ and $(y,\Gamma)$ in $\widehat{\cW}_{n}$.   Then 
$$
\mathcal{O}_{(w,\Delta)}\subset\overline{\mathcal{O}_{(y,\Gamma)}}\Leftrightarrow \Sh(\mathcal{O}^{op}_{(w,\Delta)})\leq \Sh(\mathcal{O}_{(y,\Gamma)}^{op}).
$$
\end{cor}

\begin{proof}
This follows from Theorem \ref{thm:globalclosure} and Theorem 3.4 of \cite{Shpairs}.

\end{proof}
Corollary \ref{c:canuseSh} states that the closure ordering on 
$B\backslash (\B_{n}\times\mathbb{P}^{n-1})$ can be computed using the correspondence of Theorem \ref{thm:firstbig} and the Bruhat ordering on Shareshian pairs in $\cW_{n+1} \times \cW_{n+1}$ discussed in \cite{Shpairs}.  As was observed in \emph{loc.cit.}, the Bruhat ordering on Shareshian pairs is only slightly more complicated than the product of Bruhat orders on $\cW_{n+1}\times\cW_{n+1}$ and seems  simpler to us than the order $``\unlhd"$ on $\widehat{W}_{n}$ given in (\ref{eq:Morder}).  See Example 4.15 of \cite{Shpairs} and Examples \ref{eq:2x2example} and  \ref{ex:3by3} below.

\begin{cor}
Let $(w,\Delta)\in\widehat{\cW}_{n}$ with $\Delta=\{j_{1}<\dots<j_{k}\}$.  Recall the cycle $\kappa_{\Delta}=(j_{1},\, j_{2},\dots, j_{k},\, n+1)\in\cW_{n+1}$ of Theorem \ref{thm:globalSh}.  The map $\psi: \widehat{\cW}_{n}\to \Sp$ given by 
$$
\psi(w,\Delta)=(w^{-1}\sigma, w^{-1}\kappa_{\Delta}),
$$
intertwines the order on $\widehat{\cW}_{n}$ given in (\ref{eq:Morder}) with the Bruhat order on Shareshian pairs.
\end{cor}
\begin{proof}
This is a restatement of Corollary \ref{c:canuseSh} using Equation (\ref{eq:globalSh}).
\end{proof}

\subsection{Examples}

We conclude the paper by presenting two examples which illustrate our results.  The first example gives the poset graph for the $B_{2}$-orbits on $\B_{2}\times\mathbb{P}^{1}$ and the second example gives the poset graph graph for $B_{3}$-orbits on $\B_{3}\times\mathbb{P}^{2}.$  In each example, the zero-dimensional orbit is the top row, the one-dimensional orbits are in the next row, etc.

\begin{exam}\label{eq:2x2example}
In this example, we give the Bruhat graphs of the $B_{2}$-orbits on 
$\B_{2}\times\mathbb{P}^{1}$ and the $B_{2}$-orbits on the open subvariety $\mathfrak{X}$ of $\B_{3}$.  The former are labelled by their corresponding decorated permutations and the latter by their Shareshian pairs.  We also indicate all of the monoid actions.  The colour conventions are the same as in Example \ref{ex:Q13orbits}: A red line indicates a non-compact root, and a blue line indicates a complex stable root, etc.


We begin with the Bruhat graph for $B_{2}\backslash (\B_{2}\times\mathbb{P}^{1})$.  There is only one simple root which we label by $\alpha$.   The superscript on the root below indicates whether the monoid action is defined using the first factor of $\B_{2}$ in the triple product $\B_{2}\times\B_{2}\times \mathbb{P}^{1}$ or the second factor.  

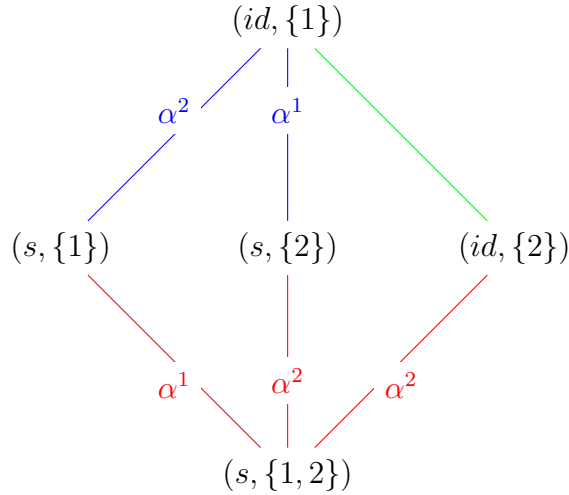
\begin{center}
\begin{tikzpicture}
[scale=2.0,auto=center,every node/.style={rectangle,fill=white!20}]
\node (a1) at  (0,5) {$(id, \{1\})$};
\node (a2) at (-1.5,3.5) {$(s, \{1\})$};
\node (a3) at (0,3.5) {$(s, \{2\})$};
\node (a4) at (1.5,3.5) {$(id, \{2\})$};
\node (a5) at (0,2) {$ (s, \{1,2\})$};

\draw  [blue] (a1) -- (a2) node[midway, above] {$\alpha^{2}$};
\draw  [blue] (a1)--(a3) node[midway, above] {$\alpha^{1}$};
\draw [green] (a1) -- (a4) ;
\draw [red] (a2) -- (a5) node[midway, below] {$\alpha^{1}$};
\draw  [red] (a3) -- (a5) node[near end, above] {$\alpha^{2}$};
\draw  [red] (a4) -- (a5) node[midway, below] {$\alpha^{2}$};
\end{tikzpicture}
\end{center}

We compare the above Bruhat graph with the one for $B_{2}\backslash\mathfrak{X}$.  We let $\Pi_{\fg_{3}}=\{\alpha,\beta\}$ and $\Pi_{\fg_{2}}=\{\alpha\}$.  A dashed line indicates the left monoid action by a root of $\Pi_{\fg_{2}}=\{\alpha\}$ and a solid line indicates a right monoid action by a root of $\Pi_{\fg_{3}}=\{\alpha,\beta\}$ as in Example \ref{ex:Q13orbits}.  We also maintain the other notational conventions of Example \ref{ex:Q13orbits}, i.e., $s=s_{\alpha}$, $s^{*}=s_{\alpha^{*}},\, t=s_{\beta}$, etc.

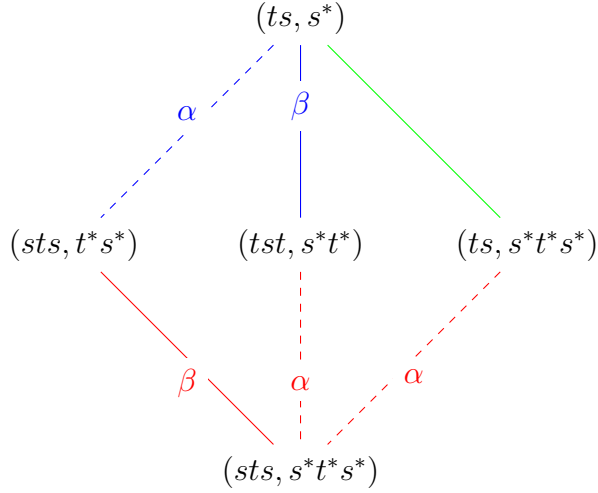
\begin{center}
\begin{tikzpicture}
[scale=2.0,auto=center,every node/.style={rectangle,fill=white!20}]
\node (a1) at  (0,5) {$(ts,s^{*})$};
\node (a2) at (-1.5,3.5) {$(sts,t^{*}s^{*})$};
\node (a3) at (0,3.5) {$(tst, s^{*}t^{*})$};
\node (a4) at (1.5,3.5) {$(ts, s^{*}t^{*}s^{*})$};
\node (a5) at (0,2) {$ (sts, s^{*}t^{*}s^{*})$};

\draw [dashed] [blue] (a1) -- (a2) node[midway, above] {$\alpha$};
\draw  [blue] (a1)--(a3) node[midway, above] {$\beta$};
\draw [green] (a1) -- (a4) ;
\draw [red] (a2) -- (a5) node[midway, below] {$\beta$};
\draw [dashed] [red] (a3) -- (a5) node[near end, above] {$\alpha$};
\draw [dashed] [red] (a4) -- (a5) node[midway, below] {$\alpha$};
\end{tikzpicture}
\end{center}

A comparison of these graphs illustrates Theorems \ref{thm:globalclosure} and \ref{thm:intromonoid}.  The closure order in the second graph is simple and intuitive from the Bruhat order for $W_3,$ while the closure order in the first graph is less intuitive.   Analysis of the monoid action also shows the advantage of using Shareshian pairs.   For example, consider the orbit $B_{2}$-orbit $\cZ$ on $\mathfrak{X}$ whose Shareshian pair is $(ts, s^{*})$.  Since $tst>ts$ and $s^{*}t^{*}>s^{*},$ it follows from Theorem \ref{thm:intertwine} and Remark \ref{r:Theorem4.7} that the root $\beta\in\Pi_{\fg_{3}}$ is complex stable for $\cZ$ and by Theorem \ref{thm:globalSh}, $\cZ^{op}=\mathcal{O}_{(w,\Delta)}$ where $(w,\Delta)=(id, \{1\})$.  We see in the first graph that $\alpha^{1}$ is complex stable for $\cZ^{op}$ with $\ms*_{1} \cZ^{op}=(m(s_{\beta})*\cZ)^{op}$, illustrating Proposition \ref{p:orbitchange}.  Similarly, Proposition \ref{p:orbitchange} can also be used to see that that monoid action $m(s_{\alpha})*_{1}(s,\{1\})=(s,\{1,2\})$ in the first graph corresponds to the action 
$m(s_{\beta})*(sts, t^{*}s^{*})=(sts, s^{*}t^{*}s^{*})$ in the second graph.  The monoid actions by $\alpha^{2}$ in the first graph correspond to left monoid actions by $\alpha$ on $B_{2}\backslash\mathfrak{X}$ as asserted in Equation (\ref{eq:leftinter}) of Theorem \ref{thm:monoidcorresp}.   
\end{exam}

 \begin{rem}\label{r:minimal}
 Note in Example \ref{eq:2x2example} that although $\mathcal{O}_{id,\{ 1\}}$ is the unique closed orbit, there are two minimal orbits with respect to the weak order (see Definition \ref{d:weak}).  This can be easily seen from the graph of the orbits in terms of their Shareshian pairs.  There is no Demazure product on either the left or the right which maps $s^{*}\to s^{*}t^{*}s^{*}$.  For general $n$, although it is easy to see that  $\mathcal{O}_{id,\{ 1\}}$ is the unique closed orbit, and one can show that there are exactly $n$ minimal elements for the weak order labelled by the decorated permutations $(id, \{1\}),\, (id, \{2\}),\dots, (id,\{n\})$.
\end{rem}

\begin{exam}\label{ex:3by3}
In this example, we present Bruhat graphs for $B_{3}\backslash (\B_{3}\times\mathbb{P}^{2})$ and $B_{3}\backslash\mathfrak{X}\subset \B_{4}$.  The $B_{3}$-orbits on $\B_{3}\times\mathbb{P}^{2}$ are labelled by their corresponding decorated permutations (see Theorem \ref{thm:magyar}) and the $B_{3}$-orbits on $\mathfrak{X}$ are labelled by their Shareshian pairs.  For the Shareshian pairs we use the same notation as in Example \ref{ex:Q13orbits}.  However, we do not label monoid actions, and the colour scheme in these graphs has a different meaning than the one in Example \ref{ex:Q13orbits}.  In the graph of $B_{3}\backslash( \B_{3}\times\mathbb{P}^{2})$ two orbits $\mathcal{O}_{(w,\Delta)}$ and $\mathcal{O}_{(y,\Delta^{\prime})}$ with $\mathcal{O}_{(w,\Delta)}\subset \overline{ \mathcal{O}_{(y,\Delta^{\prime})}}$ are joined by a coloured line if $(w,\Delta),\, (y, \Delta^{\prime})\in\widehat{\cW}_{n}(i)$, where $\widehat{\cW}_{n}(i)$ denotes the set of $i$-decorated permutations (see Notation \ref{n:marked}).  In other words, the $B$-orbits $\mathcal{O}_{(w,\Delta)}$ and $\mathcal{O}_{(y,\Delta^{\prime})}$ with $\mathcal{O}_{(w,\Delta)}\subset \overline{ \mathcal{O}_{(y,\Delta^{\prime})}}$ are joined by a coloured line if they both correspond to $S_{i}$-orbits with the same value of $i$ (see Equation (\ref{eq:markedcorres})).  When $i=1$, we use a red line, a blue line for $i=2$, and a green line for $i=3$.  Other closure relations are simply indicated by a black line.  We use the same colour scheme for $B_{3}$-orbits on $\mathfrak{X}$.  Closure relations amongst $B_{3}$-orbits in $Q_{1,2}$ are indicated by red lines, we use blue lines for $B_{3}\backslash Q_{1,3}$, and green lines for $B_{3}\backslash Q_{1,4}$.  All other closure relations are indicated by black lines.   
In each case there are a total of $28$ orbits.  This follows from Example 4.10 of \cite{Siorbits}.  We first present the Bruhat graph of $B_{3}\backslash(\B_{3}\times \B_{3}\times\mathbb{P}^{2})$.  Here the closure relation is given by the ordering on decorated permutations defined in Equation (\ref{eq:Morder}).  For the second graph of the poset $B_{3}\backslash\mathfrak{X}$, the closure relation is given by the Bruhat order on Shareshian pairs defined in Equation (\ref{eq:oldShorder}).  Using Proposition \ref{p:BBresult} and the fact that $w_{3}^{*}=s^{*}t^{*}$ and $\sigma^{-1}=s^{*}t^{*}u^{*}=w_{4}^{*}$, it is easy to verify the relations in the Bruhat graph for $B_{3}\backslash\mathfrak{X}$.  Using Theorem \ref{thm:globalclosure}, we can then easily deduce the closure relations amongst $B_{3}$-orbits on $\B_{3}\times\mathbb{P}^{2},$ showing again the advantage of parametrizing orbits using Shareshian pairs.  



\newpage
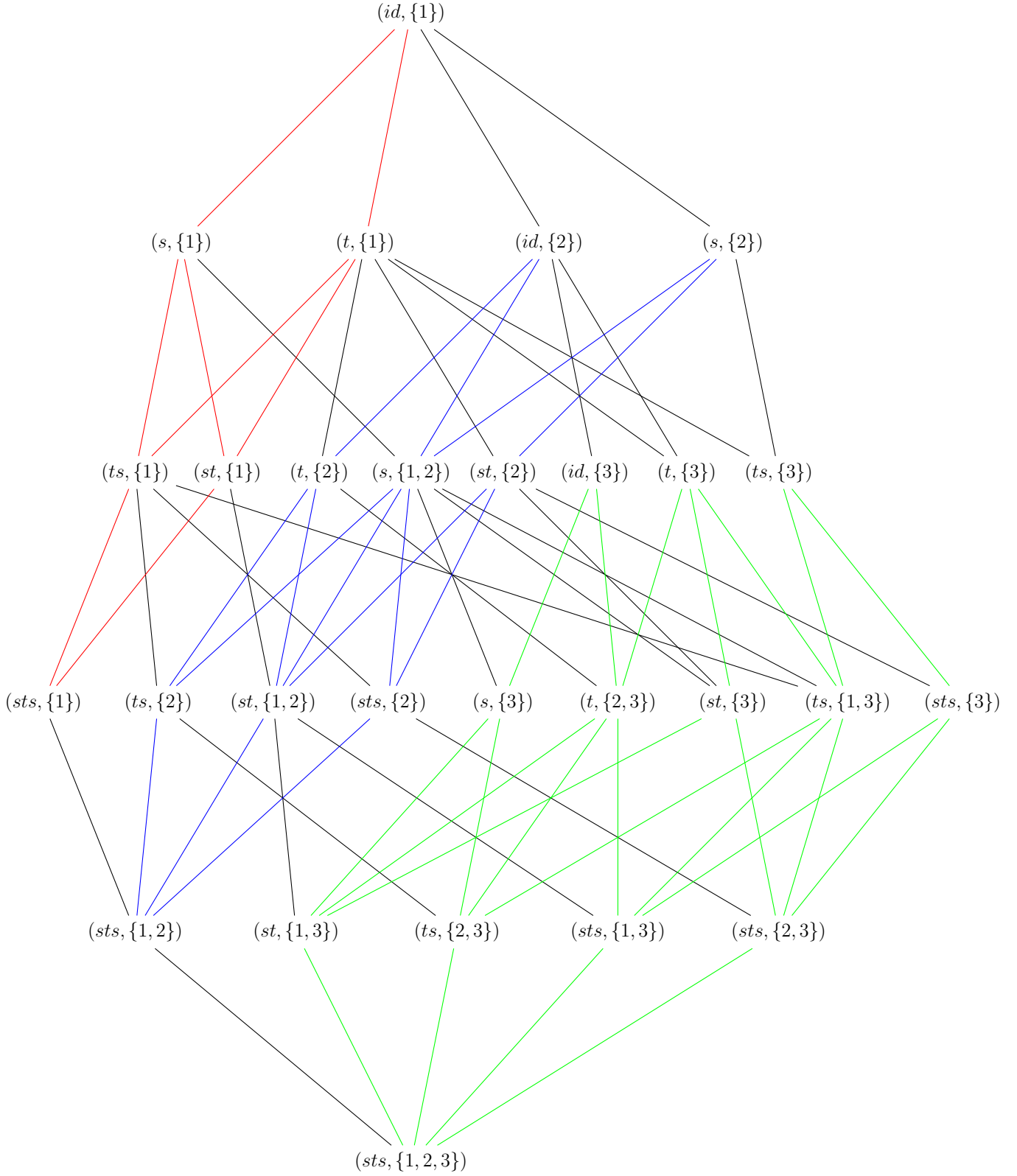
\begin{center}
\begin{tikzpicture}  
 [scale=0.825,auto=center,every node/.style={scale=0.825}] 
\node (a1) at  (-2,5) {$(id, \{1\})$};
\node (b1) at (-7,0) {$(s,\{1\})$};

\node (b2) at (-3,0) {$(t,\{1\})$};

\node (b3) at (1,0) {$ (id, \{2\})$};

\node (b4) at (5,0) {$(s, \{2\})$};
\node (c1) at (-8,-5) { $(ts, \{1\})$};

\node (c2) at (-6,-5) { $(st, \{1\})$};

\node (c3) at (-4,-5) { $(t,\{2\})$};

\node (c4) at (-2, -5) {$(s,\{1,2\})$};

\node (c5) at (0, -5) { $(st,\{2\})$}; 

\node (c6) at (2,-5) { $(id,\{3\})$}; 

\node (c7) at (4, -5) { $(t,\{3\})$}; 

\node (c8) at (6, -5) { $(ts, \{3\})$}; 

\node (d1) at (-10, -10) {$(sts, \{1\})$};

\node (d2) at (-7.5, -10) {$(ts, \{2\})$};

\node (d3) at (-5, -10) {$(st,\{1,2\})$}; 

\node (d4) at (-2.5, -10) { $(sts, \{2\})$};

\node (d5) at (0, -10) { $(s, \{3\})$};

\node (d6) at (2.5, -10) { $(t, \{2,3\})$}; 

\node (d7) at (5, -10) {$(st,\{3\})$ };

\node (d8) at (7.5, -10) {$ (ts, \{1,3\})$ };

\node (d9) at (10, -10) { $(sts,\{3\})$ };


\node (e1) at (-8, -15) { $(sts,\{1,2\})$};

\node (e2) at (-4.5, -15) {$(st, \{1,3\})$};

\node (e3) at (-1, -15) {$(ts, \{2,3\})$};

\node (e4) at (2.5, -15) {$(sts,\{1,3\})$};

\node (e5) at (6, -15) {$(sts, \{2,3\})$};


\node (f1) at (-2, -20) {$(sts, \{1,2,3\})$};

\draw [red]  (a1) -- (b1);
  \draw [red] (a1)--(b2);
  \draw  (a1)--(b3);
  \draw (a1)--(b4); 
  
 \draw [red] (b1)--(c1);
 \draw [red](b1)--(c2);
 \draw  (b1)--(c4); 
 
 \draw [red] (b2)--(c1);
 \draw [red] (b2)--(c2); 
 \draw (b2)--(c3); 
 \draw (b2)--(c5); 
 \draw (b2)--(c7);
 \draw (b2)--(c8);
 
  \draw [blue] (b3)--(c3);
  \draw [blue] (b3)--(c4);
  \draw (b3)--(c6);
  \draw (b3)--(c7);
 
  \draw [blue] (b4)--(c4);
  \draw [blue] (b4)--(c5); 
  \draw (b4)--(c8);

 
 \draw [red] (c1)--(d1);
 \draw (c1)--(d2);
 \draw (c1)--(d4);
 \draw (c1)--(d8);
 
 \draw [red] (c2)--(d1);
 \draw (c2)--(d3);
 
 \draw [blue] (c3)--(d2);
 \draw [blue] (c3)--(d3);
 \draw (c3)--(d6);
 
 \draw [blue] (c4)--(d2);
 \draw [blue] (c4)--(d3); 
 \draw [blue] (c4)--(d4);
 \draw (c4)--(d5); 
 \draw (c4)--(d7);
 \draw (c4)--(d8);
 
 \draw [blue] (c5)--(d3);
 \draw [blue] (c5)--(d4);
 \draw (c5)--(d7);
 \draw (c5)--(d9);
 
 \draw [green] (c6)--(d5);
 \draw [green] (c6)--(d6);
 
 \draw [green] (c7)--(d6);
 \draw [green] (c7)--(d7);
\draw [green] (c7)--(d8);

\draw [green] (c8)--(d8);
\draw [green] (c8)--(d9);


\draw (d1)--(e1);

\draw [blue] (d2)--(e1); 
\draw (d2)--(e3); 

\draw [blue] (d3)--(e1); 
\draw (d3)--(e2);
\draw (d3)--(e4); 

\draw [blue] (d4)--(e1);
\draw(d4)--(e5); 

\draw [green] (d5)--(e2);
\draw [green] (d5)--(e3);

\draw [green] (d6)--(e2);
\draw [green] (d6)--(e3);
\draw [green] (d6)--(e4);

\draw [green] (d7)--(e2);
\draw [green] (d7)--(e5);

\draw [green] (d8)--(e3); 
\draw [green] (d8)--(e4); 
\draw [green] (d8)--(e5);

\draw [green] (d9)--(e4);
\draw [green] (d9)--(e5); 


\draw (e1)--(f1);

\draw [green] (e2)--(f1);

\draw [green] (e3)--(f1);

\draw [green] (e4)--(f1);

\draw [green] (e5)--(f1);

\end{tikzpicture}
\end{center}

\newpage 

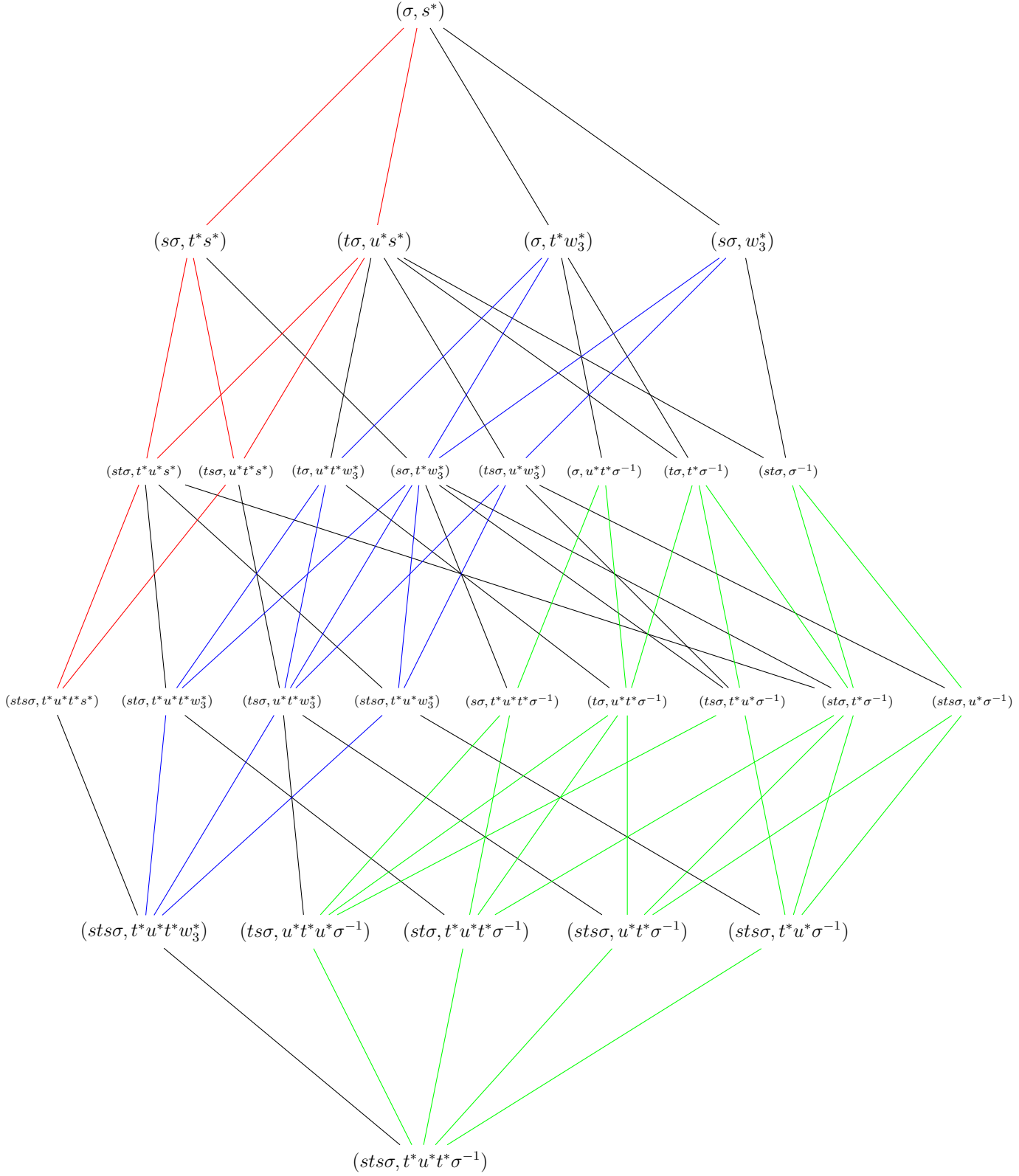
\begin{center}
 \begin{tikzpicture}  
 [scale=0.825,auto=center,every node/.style={scale=0.825}] 
\node (a1) at  (-2,5) {$(\sigma, s^{*})$};
\node (b1) at (-7,0) {$(s\sigma,t^{*}s^{*})$};

\node (b2) at (-3,0) {$(t\sigma,u^{*}s^{*})$};

\node (b3) at (1,0) {$(\sigma, t^{*}w_{3}^{*})$};

\node (b4) at (5,0) {$(s\sigma,w_{3}^{*})$};
\node (c1) at (-8,-5) {\tiny $(st\sigma, t^{*}u^{*}s^{*})$};

\node (c2) at (-6,-5) {\tiny $(ts\sigma, u^{*}t^{*}s^{*})$};

\node (c3) at (-4,-5) {\tiny $(t\sigma, u^{*}t^{*}w_{3}^{*})$};

\node (c4) at (-2, -5) {\tiny $(s\sigma, t^{*}w_{3}^{*})$};

\node (c5) at (0, -5) {\tiny $(ts\sigma, u^{*}w_{3}^{*})$}; 

\node (c6) at (2,-5) {\tiny $(\sigma, u^{*}t^{*}\sigma^{-1})$}; 

\node (c7) at (4, -5) {\tiny $(t\sigma, t^{*}\sigma^{-1})$}; 

\node (c8) at (6, -5) {\tiny $(st\sigma, \sigma^{-1})$}; 

\node (d1) at (-10, -10) {\tiny $(sts\sigma, t^{*}u^{*}t^{*}s^{*})$};

\node (d2) at (-7.5, -10) {\tiny $(st\sigma, t^{*}u^{*}t^{*}w_{3}^{*})$};

\node (d3) at (-5, -10) {\tiny $(ts\sigma, u^{*}t^{*}w_{3}^{*})$}; 

\node (d4) at (-2.5, -10) {\tiny $(sts\sigma, t^{*}u^{*} w_{3}^{*})$};

\node (d5) at (0, -10) {\tiny $(s\sigma, t^{*}u^{*}t^{*}\sigma^{-1})$};

\node (d6) at (2.5, -10) {\tiny $(t\sigma, u^{*}t^{*}\sigma^{-1})$}; 

\node (d7) at (5, -10) {\tiny $(ts\sigma, t^{*}u^{*}\sigma^{-1})$ };

\node (d8) at (7.5, -10) {\tiny $ (st\sigma, t^{*} \sigma^{-1})$ };

\node (d9) at (10, -10) {\tiny $(sts\sigma, u^{*}\sigma^{-1})$ };


\node (e1) at (-8, -15) { $(sts\sigma, t^{*}u^{*}t^{*} w_{3}^{*})$};

\node (e2) at (-4.5, -15) {$(ts\sigma, u^{*}t^{*}u^{*} \sigma^{-1})$};

\node (e3) at (-1, -15) {$(st\sigma, t^{*}u^{*}t^{*} \sigma^{-1})$};

\node (e4) at (2.5, -15) {$(sts\sigma, u^{*}t^{*} \sigma^{-1})$};

\node (e5) at (6, -15) {$(sts\sigma, t^{*}u^{*} \sigma^{-1})$};


\node (f1) at (-2, -20) {$(sts\sigma, t^{*}u^{*}t^{*}\sigma^{-1})$};

\draw [red]  (a1) -- (b1);
  \draw [red] (a1)--(b2);
  \draw  (a1)--(b3);
  \draw (a1)--(b4); 
  
 \draw [red] (b1)--(c1);
 \draw [red](b1)--(c2);
 \draw  (b1)--(c4); 
 
 \draw [red] (b2)--(c1);
 \draw [red] (b2)--(c2); 
 \draw (b2)--(c3); 
 \draw (b2)--(c5); 
 \draw (b2)--(c7);
 \draw (b2)--(c8);
 
  \draw [blue] (b3)--(c3);
  \draw [blue] (b3)--(c4);
  \draw (b3)--(c6);
  \draw (b3)--(c7);
 
  \draw [blue] (b4)--(c4);
  \draw [blue] (b4)--(c5); 
  \draw (b4)--(c8);

 
 \draw [red] (c1)--(d1);
 \draw (c1)--(d2);
 \draw (c1)--(d4);
 \draw (c1)--(d8);
 
 \draw [red] (c2)--(d1);
 \draw (c2)--(d3);
 
 \draw [blue] (c3)--(d2);
 \draw [blue] (c3)--(d3);
 \draw (c3)--(d6);
 
 \draw [blue] (c4)--(d2);
 \draw [blue] (c4)--(d3); 
 \draw [blue] (c4)--(d4);
 \draw (c4)--(d5); 
 \draw (c4)--(d7);
 \draw (c4)--(d8);
 
 \draw [blue] (c5)--(d3);
 \draw [blue] (c5)--(d4);
 \draw (c5)--(d7);
 \draw (c5)--(d9);
 
 \draw [green] (c6)--(d5);
 \draw [green] (c6)--(d6);
 
 \draw [green] (c7)--(d6);
 \draw [green] (c7)--(d7);
\draw [green] (c7)--(d8);

\draw [green] (c8)--(d8);
\draw [green] (c8)--(d9);


\draw (d1)--(e1);

\draw [blue] (d2)--(e1); 
\draw (d2)--(e3); 

\draw [blue] (d3)--(e1); 
\draw (d3)--(e2);
\draw (d3)--(e4); 

\draw [blue] (d4)--(e1);
\draw(d4)--(e5); 

\draw [green] (d5)--(e2);
\draw [green] (d5)--(e3);

\draw [green] (d6)--(e2);
\draw [green] (d6)--(e3);
\draw [green] (d6)--(e4);

\draw [green] (d7)--(e2);
\draw [green] (d7)--(e5);

\draw [green] (d8)--(e3); 
\draw [green] (d8)--(e4); 
\draw [green] (d8)--(e5);

\draw [green] (d9)--(e4);
\draw [green] (d9)--(e5); 


\draw (e1)--(f1);

\draw [green] (e2)--(f1);

\draw [green] (e3)--(f1);

\draw [green] (e4)--(f1);

\draw [green] (e5)--(f1);

\end{tikzpicture}
\end{center}
If we consider the subgraph of either of the above two graphs consisting only of the red edges and adjoining vertices, then we obtain the graph of the Bruhat order on $\cW_{3}$.  This follows easily from the fact that the posets of orbits $S_{1}\backslash\B_{3}=B_{3}\backslash\B_{3}=B_{3}\backslash Q_{1,2}$.  If we consider the subgraph consisting of the blue edges and adjoining vertices, then we obtain the Bruhat graph of the posets $S_{2}\backslash\B_{3}=B_{3}\backslash Q_{1,3}$ as given in Example \ref{ex:Q13orbits}.  Finally, if we consider the subgraph consisting of the green edges and adjoining vertices, we see the Bruhat graph of posets $B_{2}\backslash\B_{3}=B_{3}\backslash Q_{1,4}$ as given in Example 7.1 of \cite{CE21II}. 

\end{exam}

\bibliographystyle{amsalpha.bst}

\bibliography{bibliography-1}

\end{document}